\documentclass[11pt,letter]{amsart}
\usepackage{amssymb}
\usepackage{amsfonts}
\usepackage{amsmath}
\usepackage{graphicx}
\usepackage{color}
\usepackage{amsmath}
\usepackage{mathrsfs}
\usepackage{stmaryrd}
\usepackage{epsfig,color}
\usepackage{blindtext}
\usepackage{enumerate}
\usepackage{hyperref}
\usepackage{url}
\usepackage{bbm}
\usepackage{filecontents}
\usepackage{nicefrac,mathtools}
\usepackage{bm}
\DeclareGraphicsExtensions{.pdf,.jpeg,.png}
\usepackage{epstopdf}
\usepackage{cancel} 
\usepackage[normalem]{ulem} 
\usepackage{mdwlist}
\usepackage{verbatim}		
\usepackage{enumitem} 

\usepackage{color}
\usepackage[msc-links, lite]{amsrefs}
\usepackage[margin=1.2in]{geometry}

\newtheorem{theorem}{Theorem}[section]

\newtheorem{proposition}[theorem]{Proposition}
\newtheorem{lemma}[theorem]{Lemma}
\newtheorem{corollary}[theorem]{Corollary}

\newtheorem{claim}[]{Claim}

\theoremstyle{definition}
\newtheorem{definition}[theorem]{Definition}
\theoremstyle{remark}
\newtheorem{remark}[theorem]{Remark}

\numberwithin{equation}{section}

\newcommand{\dv}{\mathrm{div}}

\newcommand{\mf}{\mathbf}
\newcommand{\mb}{\mathbb}
\newcommand{\mc}{\mathcal}
\newcommand{\ms}{\mathscr}
\newcommand{\mk}{\mathfrak}

\newcommand{\wti}{\widetilde}

\newcommand{\M}{\mathbf M}

\newcommand{\F}{\mathbf F}
\newcommand{\B}{\overline{\mathbf B}}
\newcommand{\oB}{\overline{B}}

\newcommand{\A}{\mathcal A}
\newcommand{\C}{\mathcal C}

\newcommand{\Vol}{\mathrm{Vol}}
\newcommand{\Area}{\mathrm{Area}}
\newcommand{\Id}{\mathrm{Id}}

\newcommand{\R}{\mathbb{R}}

\newcommand{\dist}{\operatorname{dist}}

\newcommand{\n}{\mathbf n}

\DeclareMathOperator{\Index}{index}

\DeclareMathOperator{\Ric}{Ric}

\DeclareMathOperator{\spt}{spt}

\title[Multiplicity one in compact manifolds with boundary]{Multiplicity one for min-max theory in compact manifolds with boundary and its applications}

\date{\today}

\author{Ao Sun}

\address{Department of Mathematics,
	University of Chicago,
	5734 S. University Avenue,
	Chicago, IL, 60637}
\email{aosun@uchicago.edu}

\author{Zhichao Wang}
\address{Department of Mathematics, University of Toronto, Toronto, ON M5S2E4, Canada}
\email{zhichao.wang@utoronto.ca}

\author{Xin Zhou}
\address{Department of Mathematics, Cornell University, Ithaca, NY 14853, and Department of Mathematics, University of California Santa Barbara, Santa Barbara, CA 93106, USA}
\email{xinzhou@cornell.edu}

\begin{document}

	\begin{abstract}
		We prove the multiplicity one theorem for min-max free boundary minimal hypersurfaces in compact manifolds with boundary of dimension between 3 and 7 for generic metrics. To approach this, we develop existence and regularity theory for free boundary hypersurface with prescribed mean curvature, which includes the regularity theory for minimizers, compactness theory, and a generic min-max theory with Morse index bounds. As applications, we construct new free boundary minimal hypersurfaces in the unit balls in Euclidean spaces and self-shrinkers of the mean curvature flows with arbitrarily large entropy.
	\end{abstract}

	\maketitle
	
	\section{Introduction}
	
	The main motivations of this article are two-folded. On one hand, 
	we further advance the existence theory for minimal hypersurfaces with free boundary in compact manifolds with boundary, inspired by the recent exciting developments for closed minimal hypersurfaces. We have seen immense existence results concerning free boundary minimal hypersurfaces in recent years, c.f. \cite{Fra00}, \cite{FrSch16}, \cite{MaNS17}, \cite{LZ16}, \cite{KL17}, \cite{Lin-Sun-Zhou20}, \cite{GLWZ19}, \cite{Wang20}, \cite{Pigati20}, \cite{CFS20}. In this article, we prove the free boundary version of the Multiplicity One Conjecture in min-max theory. This directly gives the existence of new free boundary minimal hypersurfaces even in the unit balls of the Euclidean spaces.  On the other hand, we propose to use the results developed here and before by the authors to construct new minimal hypersurfaces in singular spaces. The idea is to use compact manifolds with boundary to approximate a given singular space. In fact, consider the compliments of tubular neighborhoods of the singular set, which are compact manifolds with nontrivial boundary; one can apply the min-max theory to obtain free boundary minimal hypersurfaces and take the limits.  We apply this strategy to the singular Gaussian space, and construct connected minimal hypersurfaces therein with arbitrarily large Gaussian area. These minimal hypersurfaces are self-shrinkers of the mean curvature flow with arbitrarily large entropy, c.f. \cite{CM12_1}. 
	

	\subsection{Multiplicity one theorem for the free boundary min-max theory} 
	
	Denote by $(M^{n+1},\partial M, g)$ a compact Riemannian manifold with smooth boundary of dimension $n+1$. In \cite{Alm62}, Almgren proved that the space of mod 2 relative cycles $\mc Z_n(M,\partial M;\mb Z_2)$ is weakly homotopic to $\mb {RP}^\infty$; see also \cite{LMN16}*{\S 2.5}. By performing a min-max procedure, Gromov \cite{Gro88} defined the {\em volume spectrum of $M$}, which is a sequence of non-decreasing positive numbers 
	\[ 0<\omega_1(M, g)\leq\omega_2(M, g)\leq\cdots \leq\omega_k(M, g) \to \infty, \]
	depending only on $M$ and $g$. Moreover, Gromov \cite{Gro88} also showed that for each $g$, $\omega_k$ grows like $k^{\frac{1}{n+1}}$; 
	see also Guth \cite{Guth09}. Recently, Liokumovich-Marques-Neves \cite{LMN16} proved that the volume spectrum obey a uniform Weyl Law.

	When $M$ is closed with no boundary, by adapting the regularity theory given by Pitts \cite{Pi} and Schoen-Simon \cite{SS}, and the index bounds given by Marques-Neves \cite{MN16}, it is known that, when $3\leq (n+1)\leq 7$,  each $k$-width $\omega_k$ is realized by a disjoint collection of closed, embedded, minimal hypersurfaces with integer multiplicities, i.e. there exist integers $\{m_j^k\}$ and minimal hypersurfaces $\{\Sigma_j^k\}$ so that 
	\begin{equation}\label{eq:min-max for closed}
		\omega_{k}(M;g)=\sum_{j}^{l_k} m_j^k\cdot\Area (\Sigma_j^k) \ \ \ \text{ and } \ \  \ \sum_j^{l_k}\Index(\Sigma_j^k)\leq k; 
	\end{equation}
	see \cite{IMN17}*{Proposition 2.2}. In \cite{MN16, MN18}, Marques-Neves raised the Multiplicity One Conjecture, which asserts that for bumpy metrics\footnote{Here a metric is said to be {\em bumpy} if each closed immersed minimal hypersurfaces has no Jacobi field, and the set of bumpy metrics in a closed manifold is generic in the Baire sense by White \citelist{\cite{Whi91}\cite{Whi17}}.}, $m_j^k=1$ and $\Sigma_j^k$ is two-sided for all $k \in \mathbb N$ and $1\leq j\leq l_k$. 
	This conjecture was later confirmed by the last-named author \cite{Zhou19}; (see also Chodosh-Mantoulidis \cite{CM20}). 
	This, together with \cite{MN18}, implies the existence of a closed, embedded minimal hypersurface of Morse index $k$ for each $k\in \mathbb{N}$ under a bumpy metric. 
	This body of works established the Morse theory for the area functional. We also refer to \cite{MN17, IMN17, MNS17, Song18, YY.Li19, Marques-Montezuma-Neves20, Song-Zhou20} for recent developments on the existence theory of closed minimal hypersurfaces, particularly on the resolution of Yau's conjecture.

	
	In this paper, we focus on compact manifolds with non-empty boundary. 
	In \cite{LZ16}, Li-Zhou proved the general existence of free boundary minimal hypersurfaces by adapting the Almgren-Pitts theory. Inspired by \citelist{\cite{MN16}\cite{IMN17}}, the last two authors together with Q. Guang and M. Li generalized \eqref{eq:min-max for closed} to compact manifolds in \cite{GLWZ19}. Precisely, for each positive integer $k$, there exist integers $\{m_j^k\}$ and disjoint, embedded, free boundary minimal hypersurfaces $\{\Sigma_j^k\}$ so that 
	\begin{equation}\label{eq:min-max for cmpt with bdry}
		\omega_{k}(M;g)=\sum_{j}^{l_k} m_j^k\cdot\Area (\Sigma_j^k) \ \ \ \text{ and } \ \  \ \sum_j^{l_k} \Index(\Sigma_j^k)\leq k.
	\end{equation}

	Inspired by \cite{Zhou19}, it is natural to conjecture that $m_j^k=1$ for generic metrics on compact manifolds with boundary. In this article we confirm this conjecture. 
	\begin{theorem}\label{thm:intro:multi 1}
		Let $(M^{n+1},\partial M)$ be a compact manifold with boundary of dimension $3\leq (n+1)\leq 7$. For generic metrics $g$ on $(M,\partial M)$ and every integer $k>0$, there exist a sequence of disjoint, two-sided, connected, embedded free boundary minimal hypersurfaces $\{\Sigma_j^k\}$ satisfying
		\[   \omega_{k}(M;g)=\sum_{j} \Area (\Sigma_j^k) \ \ \ \text{ and } \ \  \ \sum_j\Index(\Sigma_j^k)\leq k. \] 	
	\end{theorem}
	
	\begin{remark}
		Our proof follows generally the approach in \cite{Zhou19}, but there are several essential new challenges which stimulate novel ideas (we refer to Section \ref{SS:challenges} for  elucidations). Recall that a hypersurface with mean curvature prescribed by an ambient function $h$ will be called an $h$-hypersurface.
		\begin{enumerate}[label=(\alph*)]
			\item\label{challenge:1} Generic properness for free boundary minimal hypersurfaces;
			
			\item\label{challenge:2} Existence and regularity for free boundary $h$-hypersurfaces, including: the regularity for minimizers, compactness, and a generic min-max theory;
			
			\item\label{challenge:3} Genericity of good mean curvature prescribing functions $h$ for which the min-max theory for free boundary $h$-hpersurfaces can be applied;
			
			\item\label{challenge:4} Generic countability of free boundary $h$-hypersurface -- an essential ingredient in the proof of Morse index estimates in the free boundary $h$-hypersurface min-max theory.
		\end{enumerate}
	\end{remark}

	 To solve the above challenge \ref{challenge:1}, we prove the following strong bumpy metric theorem for compact manifolds with boundary.
	\begin{theorem}\label{thm:intro:generic strongly bumpy}
		For a generic metric on $(M,\partial M)$, every embedded free boundary minimal hypersurface is both non-degenerate and proper.
	\end{theorem}
	
	Here a 
	hypersurface $\Sigma$ is {\em proper} if it has empty touching set, i.e. $\mathrm{Int}(\Sigma) \cap \partial M=\emptyset$.
	
	\begin{remark}
		In \cite{ACS17}*{Theorem 9}, Ambrozio-Carlotto-Sharp proved a similar bumpy metric theorem, which says that for generic metrics, each properly embedded free boundary minimal hypersurface is non-degenerate. Our result is a strengthened version in the sense that each embedded free boundary minimal hypersurface is automatically proper in our generic metrics.
	\end{remark}
	
	As a direct corollary of Theorem \ref{thm:intro:multi 1}, we construct free boundary minimal hypersurfaces with arbitrarily large area in the unit ball $\mb B^{n+1}$, and more generally in any strictly convex domain of $\mb R^{n+1}$. Recall that a compact Riemannian manifold with non-negative Ricci curvature and strictly convex boundary does not contain any stable free boundary minimal hypersurfaces. Moreover, by Fraser-Li \cite{FrLi}*{Lemma 2.4}, any two free boundary minimal hypersurfaces in this type of manifolds intersect with each other. Making use of the compactness \citelist{\cite{ACS17}\cite{GWZ18}}, we have the following result.
	\begin{corollary}
		Let $(M^{n+1},\partial M,g)$ be a compact Riemannian manifold with non-negative Ricci curvature and strictly convex boundary of dimension $3\leq (n+1)\leq 7$. Then there exist a sequence of two-sided, connected, embedded, free boundary minimal hypersurfaces $\{\Sigma_k\}$ satisfying
		\[  \Area(\Sigma_k)\sim k^{\frac{1}{n+1}} \ \ \ \ \text{ and }\ \ \ \ \Index(\Sigma_k)\leq k. \]
	\end{corollary}
	
 In particular, the unit ball $\mb B^{n+1}$ with $3\leq (n+1)\leq 7$ contains a sequence of free boundary minimal hypersurfaces with areas going to $+\infty$. These are different from the hypersurfaces constructed by Fraser-Schoen \cite{FrSch16}, Kapouleas-Li \cite{KL17} and Carlotto-Franz-Schulz \cite{CFS20}.

	\subsection{Application to Gaussian spaces}
	
	Besides the independent interest, the multiplicity one result for free boundary minimal hypersurfaces may also have many other applications. Here we present an application to construct new minimal hypersurfaces in certain non-compact and singular manifolds.
	
	An embedded hypersurface in $\mb R^{n+1}$ is called a {\em self-shrinker} if 
	\[ H=\frac{\langle x,\nu\rangle}{2}, \]
	where $H$ is the mean curvature with respect to the normal vector field $\nu$. Equivalently, a self-shrinker is a minimal hypersurface in $\mb R^{n+1}$ under the Gaussian metric $\mc G=(4\pi)^{-1}e^{-\frac{|x|^2}{2n}}\delta_{ij}$; see \cite{CM12_1} for more details. The area of such a hypersurface under Gaussian metric is called the {\em entropy}, defined by Colding-Minicozzi in \cite{CM12_1}.
	
	Self-shrinkers are models of the singularities of mean curvature flows, and the index of a self-shrinker (as a minimal hypersurface in $(\mb R^{n+1},\mc G)$) characterizes the instability of the mean curvature flow at the singularity; see the discussions in \cite{CM12_1}. Understanding self-shrinkers is very important to the study of singular behaviors of mean curvature flows.	
	
	The main challenge of using min-max method to construct self-shrinkers is that the Gaussian metric space $(\R^{n+1},\mc G)$ is non-compact, and the space becomes singular at infinity -- the curvature blows up at infinity, meanwhile the spheres with increasing radii have exponentially decreasing areas. Therefore, the classical min-max method can not be used directly. We note that Ketover-Zhou \cite{Ketover-Zhou18} have established a version of the min-max theory in $(\mb R^{3}, \mc G)$ using a special flow near the singularity.
	
	Our approximation method provides the following new constructions of embedded self-shrinkers. We believe that our method can also be applied to other non-compact singular spaces.
	
	\begin{theorem}\label{thm:intro:min-max in gaussian space}
		Given $3\leq (n+1)\leq 7$, there exists a sequence of embedded self-shrinkers $\{\Sigma_k\}$ with entropy growth $k^{\frac{1}{n+1}}$ and $\mathrm{index}(\Sigma_k)\leq k$.	
	\end{theorem}

 A direct corollary is the existence of self-shrinkers with arbitrary large entropy.
	
		\begin{corollary}\label{cor:intro:large entropy and finite index}
		For any  $\mb R^{n+1}$ ($n\geq 2$) and $\Lambda>0$, there exists an embedded self-shrinker with $\lambda(\Sigma)\in (\Lambda,\infty)$ and $\mathrm{index}<\infty$.
	\end{corollary}
	\begin{proof}[Proof of Corollary \ref{cor:intro:large entropy and finite index}]
		The corollary follows from Theorem \ref{thm:intro:min-max in gaussian space} if $(n+1)\leq 7$. Now we assume that $n\geq 7$. Note that $\lambda(\Sigma\times\mb R)=\lambda(\Sigma)$ and $\Sigma\times \mb R$ is a self-shrinker if and only if $\Sigma$ is a self-shrinker. Moreover, $\Sigma$ has finite index if and only if $\Sigma\times \mb R$ has finite index by Theorem \ref{thm:shrinker finite index}. 
	\end{proof}
	
	There are several pioneer works which provided constructions of embedded self-shrinkers in $\R^3$ by other techniques, see \citelist{\cite{An92}\cite{Ng14}\cite{KKM18} \cite{Mo11}}. However, the entropy of all the self-shrinkers constructed in these works is bounded from above by a universal constant. In contrast, our min-max method constructs self-shrinkers with arbitrarily large entropy. 
	
	The entropy of a mean curvature flow is monotone non-increasing. As a consequence, any tangent flow of a mean curvature flow has entropy bounded from above by the entropy of the initial hypersurface. Based on this fact, there has been much research on low entropy mean curvature flow; cf. \citelist{\cite{Bernstein-Wang} \cite{mramor2018low} \cite{bernstein2020closed}}. In contrast, our theorem implies that large entropy mean curvature flow may have very complicated tangent flows, which addresses the complication of mean curvature flow singularities.

	\subsection{Challenges and ideas}\label{SS:challenges}
	Inspired by the proof of the Multiplicity One Theorem in closed manifolds \cite{Zhou19}, we prove Theorem \ref{thm:intro:multi 1} by establishing the min-max theory for free boundary $h$-hypersurfaces. Let $(M,\partial M,g)$ be a compact Riemannian manifold with boundary of dimension $3\leq (n+1)\leq 7$. Moreover, we assume that every embedded free boundary minimal hypersurface in $(M,\partial M,g)$ is proper and does not have non-trivial Jacobi fields. Then the $k$-the volume spectrum $\omega_k(M;g)$ is realized by the min-max value of a homotopy class of $k$-sweepouts. Note that the sweepouts are maps into $\mc Z_n(M,\partial M;\mb Z_2)$. 
	We can lift the $k$-sweepouts to the space of Cappccioppoli set $\mc C(M)$, and consider the associated relative homotopy class, denoted by $\wti \Pi$. We will show that the relative min-max width is equal to $\omega_k(M;g)$.
	
	Take $h\in C^\infty(M)$ to be defined later. Let $\epsilon_j\rightarrow 0$ and denote by $\Sigma_j=\partial \Omega_j$ the free boundary $\epsilon_jh$-hypersurface, produced by relative min-max theory (associated with $\wti \Pi$). Letting $j\rightarrow\infty$, and using compactness results, we get a free boundary minimal hypersurface $\Sigma$ with multiplicity $m$. The aim is to prove $m=1$ by taking suitable $h$. Indeed, if $m>1$, we can construct a non-trivial positive Jacobi field $u$ so that $Lu=\pm 2h$ on $\Sigma$ and $\frac{\partial u}{\partial \bm\eta}=h^{\partial M}(\nu,\nu)u$, where $L$ is the Jacobi operator, $\nu$ and $\bm\eta$ are respectively the unit normal and co-normal vector of $\Sigma$, and $h^{\partial M}(\cdot,\cdot)$ is the second fundamental form of $\partial M$. By choosing a suitable $h$, we prove that this kind of $u$ does not exist.

	\vspace{1em}
	However, there are several new challenges in this process. First, we need to show the genericity of the metrics under which every embedded free boundary minimal hypersurface is proper and has no non-trivial Jacobi field. Such a metric is called {\em strongly bumpy}. The denseness of such metrics can be obtained by perturbing any given metric on a closed manifold $\wti M$ (which contains $M$ as a domain) so that each minimal hypersurface in $\wti M$ with free boundary on $\partial M$ is non-degenerate and transverse to $\partial M$. To obtain openness, we use the compactness to show that given a strongly bumpy metric $g$ and a constant $C$, there is an open neighborhood of $g$ so that for each metric $g'$ in that neighborhood, any free boundary minimal hypersurface in $(M,\partial M,g')$ with weak Morse index and area no more than $C$ is non-degenerate and proper. By taking intersections of such open and dense sets of metrics for a sequence $C_n\rightarrow\infty$, we obtain the desired genericity. 
	
	Second, to establish the min-max theory for free boundary $h$-hypersurfaces, we need to generalize a number of regularity results. In Theorem \ref{thm:regularity for minimizer}, we prove the regularity for minimizing free boundary $h$-hypersurfaces by using the reflecting method in \cite{Gr87} and the regularity results in \cite{Mor03}. The major step for this min-max theory is to establish the regularity for $h$-almost minimizing varifolds. To finish the key gluing step therein, we apply the interior gluing procedure by by Zhou-Zhu \cite{ZZ18}, and then take the advantage of unique continuation to extend the gluing all the way to boundary. Comparing with the proof of the corresponding regularity results for free boundary minimal hypersurfaces in \cite{LZ16}, our arguments for free boundary $h$-hypersurfaces are much simpler since we only need to consider the case where the touching sets are contained in some $(n-1)$-dimensional subsets.	

   To obtain the aforementioned control on the touching sets, we need to restrict to ``{\em good}" prescribing functions $h$, where each free boundary $h$-hypersurface has at most $(n-1)$-dimensional touching set (including self-touching and touching with $\partial M$). To show the genericity of such good functions, we extend a given smooth function $h$ on $M$ to $\wti h$ on $\wti M$ and perturb it to $\wti h_1$ by \cite{ZZ18}*{Proposition 3.8} so that each $\wti h_1$-hypersurface has small self-touching set. By genericity of Morse functions we can find another perturbed function $\wti h_2$ so that each $\wti h_2$-hypersurface touches $\partial M$ in a small subset as desired. This implies $\wti h_2|_M$ is a good choice. Moreover,  each function in a neighborhood of $\wti h_2|_M$ is also a good choice. This gives the genericity of good prescribing functions.
	
	Finally, we have to prove the index upper bounds in the min-max procedure, and a crucial ingredient is the countability of free boundary $h$-hypersurfaces for good pairs (see Subsection \ref{subsec:good pairs}). Our proof of the countability follows from the following observation of the compactness of free boundary $h$-hypersurfaces. Let $\Sigma_j$ be a sequence of free boundary $h$-hypersurfaces with bounded index and area. Then it converges as varifolds 
	to some limit hypersurface $\Sigma$; (see Theorem \ref{thm:compactness for FPMC}). If the convergence is smooth, then $\Sigma$ has a non-trivial Jacobi field related to $\mc A^h$; on the contray, if the convergence is not smooth, $\Sigma$ has strictly less index than that of $\Sigma_j$ for sufficiently large $j$. We refer to Lemma \ref{lem:countable} for more details.
	
	\vspace{1em}
         To construct self-shrinkers of large entropy, we use approximation methods. More precisely, we focus on balls with larger and larger radius in $\R^{n+1}$, perturb the Gaussian metric on the balls slightly to be generic, and use the theory developed in this article to construct free boundary minimal hypersurfaces of large area. By passing to a limit, these free boundary minimal hypersurfaces will converge to minimal hypersurface in the Gaussian space, or equivalently self-shrinkers. When passing to limits, there may be mass drop, and the limit shrinkers would not have large entropy. To overcome this issue, we prove a new monotonicity formula of minimal surfaces in almost Gaussian spaces, which is more close related to the monotonicity formula of minimal hypersurfaces in the Euclidean spaces compared with that for shrinkers. We also study the spectrum of the Jacobi operator for shrinkers of the form $\Sigma\times\R$ whenever $\Sigma$ is a lower dimensional shrinker, and we prove in Theorem \ref{thm:shrinker finite index} that if $\Sigma$ has finite index, so does $\Sigma\times\R$. This allows us to construct smooth self-shrinkers with finite index and arbitrarily large entropy in high dimensional spaces.

	\subsection*{Outline}
	This paper is organized as follows. In Section \ref{sec:pre}, we first collect some notations and then provide some technical results including the regularity for $\mc A^h$-minimizing hypersurfaces and the compactness of free boundary $h$-hypersurfaces. We mention that Theorem \ref{thm:intro:generic strongly bumpy} is proven in Subsection \ref{subsec:generic properness}. Then in Section \ref{sec:min-max for FPMC}, we prove the relative min-max theory for free boundary $h$-hypersurfaces. The regularity of $h$-almost minimizing varifold with free boundary is proved in Theorem \ref{thm:regularity of h minmax}. Section \ref{sec:multi 1 for fbmh} is devoted to the proof of Theorem \ref{thm:intro:multi 1}. Finally, by applying our main result, we prove the existence of self-shrinkers with arbitrarily large entropy in Section \ref{sec:min-max in Gaussian space}. We also provide some well-known results and technical computations in Appendix \ref{app:sec:removable singularity}--\ref{app:sec:local foliation}.

	\subsection*{Acknowledgments}
	We would like to thank Professor Bill Minicozzi for his interest. Part of this work was carried out when Z.W. was a postdoc at Max-Planck Institute for Mathematics in Bonn and he thanks the institute for its hospitality and financial support. X. Z. is partially supported by NSF grant DMS-1811293, DMS-1945178 and an Alfred P. Sloan Research Fellowship. 
	
	\section{Preliminaries}\label{sec:pre}
	\subsection*{Notations} We collect some notions here.
	
	Let $(M^{n+1},\partial M^{n+1},g)$ be a compact, oriented, smooth Riemannian manifold with smooth boundary, of dimension $3\leq n+1\leq 7$. In this paper, we always isometrically embed $(M,\partial M,g)$ into a closed manifold $(\widetilde M,\widetilde g)$ with dimension $n+1$. Assume that $(\widetilde M,\widetilde g)$ is embedded in some $\mathbb R^L, L\in \mathbb N$.
	\begin{basedescript}{\desclabelstyle{\multilinelabel}\desclabelwidth{10em}}
		\item[$B_r(p)$]  the geodesic ball of $(\wti M,g)$;
		\item[$\mathfrak X(\wti M)$] the space of smooth vector fields on $\wti M$;
		\item[$\mathfrak {\wti X}(M,\Sigma)$] the collection of $X\in\mathfrak X(\wti M)$ with $X(p)\in T_p(\partial M)$ for $p$ in a neighborhood of $\partial \Sigma$ in $\partial M$;
		\item[$\mc H^k$]  $k$-dimensional Hausdorff measure;
		\item[$\Gamma(M)$] the space of smooth Riemannian metrics on $M$;
		\item[$\mathcal R_k(M;\mathbb Z)$]	 the space of integer rectifiable $k$-currents with support in $M$;
		\item [$\mathcal R_k(M;\mathbb Z_2)$] the space of mod 2 rectifiable $k$-currents with support in $M$;
		\item [$G$] the group $\mathbb Z$ or $\mathbb Z_2$;
		\item [$Z_k(M,\partial M;G)$] the space of $T\in\mathcal R_k(M;G)$ and $\mathrm{spt}(\partial T)\subset\partial M$;
		\item [$\mathbf I_k(M;G)$] the space of $T\in\mathcal R_k(M;G)$ and $\partial T\in\mathcal R_{k-1}(M;G)$;
		\item [$\tau$] equivalent class of $Z_k(M,\partial M;G)$, i.e. $T_1\in\tau$ if and only if $T-T_1\in\mc R_k(\partial M;G)$ for any $T\in \tau$;
		\item [$\mathcal Z_k(M,\partial M;G)$] the space of equivalent classes in $Z_k(M,\partial M;G)$;
		\item [$|T|$] the integer rectifiable varifold associated with $T\in \mc R_k(M;\mb Z)$;
		\item [$\|T\|$] the Radon measure associated with $T\in \mc R_k(M;\mb Z)$;
		\item [$\mf M$ on $\mc R_k(M;G)$] the mass norm on $\mathcal R_k(M;G)$;
		\item [$\mf M$ on $\mc Z_k(M,\partial M;G)$] $\mathbf M(\tau)=\inf\{\mathbf M(T):T\in\tau\}$.
		\item [$\mc F$ on  $\mc R_k(M;G)$]  the flat metric on $\mc R_k(M;G)$;
		\item [$\mc F$ on $\mathcal Z_k(M,\partial M;G)$] $\mathcal F(\tau_1,\tau_2)=\inf\{\mathcal F(T_1,T_2):T_1\in\tau_1,T_2\in \tau_2\}$ for $\tau_1,\tau_2\in \mathcal Z_k(M,\partial M;G)$;
		\item [$\mc C(M)$] the space of sets $\Omega\subset M$ with finite perimeter (Cappccioppoli set);
		\item [$\mathfrak X(M,\partial M)$] the space of smooth vector field in $M$ so that $X(p)\in T_p(\partial M)$ for $p\in \partial M$;
		\item [$\partial \Omega$] the (reduced) boundary of $\llbracket \Omega\rrbracket$ restricted in the interior of $M$, (thus, as an integer rectifiable current, $\partial \Omega\llcorner \partial M=0$);
		\item [$\mathcal V_k(M)$] the space of $k$-dimensional integral varifolds;
		\item [$\mathrm{VarTan}(V,p)$] the tangent space of $V$ at $p$;
		\item [$G_k(M)$] denotes the $k$-dimensional Grassmannian bundle over $M$;
		\item [$\mathbf F$ on $\mc V_k(M)$] $\mathbf F(V, W ) = \sup\{V (f)- W (f ) : f\in C_c(G_k(\mathbb R^L)),
		|f |\leq 1, \mathrm{Lip}(f ) \leq 1\},$
		for $V, W\in\mathcal V_k(M)$;
		\item [$\mf F$ on $\mc C(M)$]  $\mathbf F(\Omega_1,\Omega_2)=\mathcal F(\Omega_1-\Omega_2)+\mathbf F(|\partial\Omega_1|,|\partial \Omega_2| )$ for $\Omega_1,\Omega_2\in \mathcal C(M)$;
		\item [$\llbracket M\rrbracket$] the integral current associated with $M$;
		\item [${[M]}$] the integral varifold associated with $M$;
		\item [$\overline{\mathbf B}^{\mathbf F}_\epsilon(\Omega)$] the collection of $\Omega'\in \mathcal C(M)$ so that $\mathbf F(\Omega,\Omega')\leq \epsilon$, where $\Omega\in\mc C(M)$.
	\end{basedescript}

	For any $\tau\in \mathcal Z_k(M,\partial M;G)$, there is a \emph{canonical representative} $T\in \tau$ where $T\in Z_k(M,\partial M;G)$ and $T\llcorner\partial M=0$.
	
	Given $c > 0$, a varifold $V\in\mc V_k(M)$ is said to have {\em $c$-bounded first variation in an open subset $U\subset M$}, if
	\[|\delta V (X)|\leq c\,\int_M|X|d\mu_V, \text{ for any } X\in\mathfrak X(U,U\cap \partial M),
	\]
where $\mathfrak X(U,U\cap \partial M)$ is defined to be the subset of $\mk X(M,\partial M)$ compactly supported in $U$.
	
	Let $\Sigma$ be an $n$-dimensional manifold with smooth boundary. Recall that an {\em immersion} is a map $\phi:(\Sigma,\partial \Sigma)\hookrightarrow(M,\partial M)$. In this paper, we often use $\Sigma$ to denote both the immersion and the hypersurface when there is no ambiguity.
	
	We are interested in the following weighted area functional defined on $\mathcal C(M)$. Given $h\in C^{\infty}(M)$, define the {\em $\mathcal A^{h}$-functional} on $\mathcal C(M)$ by
	\begin{equation}\label{eq:def of Ah}
		\mathcal A^h(\Omega)=\mathbf M(\partial \Omega)-\int_{\Omega}h\,d\mathcal H^{n+1}.
	\end{equation}
	The {\em first variation formula} for $\mathcal A^{h}$ along $X\in\mathfrak X(M,\partial M)$ is 
	\begin{equation*}
		\delta \A^h\big|_{\Omega}(X)=\int_{\partial \Omega}\mathrm{div}_{\partial \Omega}X\,d\mu_{\partial \Omega}-\int_{\partial\Omega}h\langle X,\nu\rangle\,d\mu_{\partial\Omega},
	\end{equation*}
	where $\nu$ is the outward unit normal vector field of $\partial\Omega$.
	
	When the boundary $\partial\Omega=\Sigma$ has support on a smooth immersed hypersurface, by virtue of the divergence theorem, we have
	\begin{equation}\label{eq:1st variation}
		\delta \A^h\big|_{\Omega}(X)=\int_{\partial\Omega}(H-h)\langle X,\nu\rangle\,d \mu_{\partial\Omega}+\int_{\partial\Sigma}\langle X,\bm\eta\rangle\,d\mu_{\partial\Sigma},
	\end{equation}
	where $\bm\eta$ is the unit outer co-normal vector field of $\partial \Sigma$. If $\Omega$ is a critical point of $\A^h$, then \eqref{eq:1st variation} directly implies that $\Sigma$ must have mean curvature $H=h|_\Sigma$ and $\nu\perp \nu_{\partial M}$ along $\partial \Sigma$, where $\nu_{\partial M}$ is the outward unit normal vector field of $\partial M$.
	
	Recall that an immersed hypersurface is called an {\em $h$-hypersurface} if its mean curvature $H(x)=h(x)$ everywhere.

	\subsection{Regularity for boundaries minimizing the $\mc A^h$-functional}
	In \cite{Mor03}, F. Morgan gives a general way to prove the regularity of isoperimetric hypersurfaces in Riemannian manifolds. His methods can be applied to $\mc A^h$-functional to prove the regularity of $h$-hypersurfaces bounding a domain which minimizes the $\mc A^h$-functional; (see  \cite{ZZ18}*{Theorem 2.2}). In this part, we provide the regularity for $\mc A^h$-minimizers with free boundaries.
	
	Let $B$ be an $(n+1)$-ball in $\mb R^{n+1}$, and $S\subset B$ a compact embedded $n$-ball such that $\partial S\cap B=\emptyset$. Denote by $B^+$ and $B^-$ the two components of $B\setminus S$. 
	Note that the $\mc A^h$-functional in \eqref{eq:def of Ah} is well-defined for $\Omega\in \mc C(B^+)$.

	\begin{theorem}\label{thm:regularity for minimizer}
		Suppose that $\Omega\in\mc C(B^+)$ minimizes the $\mc A^h$-functional with free boundary on $S$: that is, for any other $\Lambda\in\mc C(B^+)$ with $\spt\Vert\Lambda-\Omega\Vert \subset B$, we have $\mc A^h(\Lambda)\geq \mc A^h(\Omega)$. Then except for a set of Hausdorff dimension at most $(n-7)$, $\partial\Omega$ is smooth hypersurface embedded properly (under relative topology) in $B^+$ and meets $S\cap B^+$ orthogonally.
	\end{theorem}
	
	We recall a few notations. Without loss of generality, we assume that $S\cap \spt\Vert \partial \Omega\Vert\neq \emptyset$. For $x\in B$, let $\xi(x)$ be the closest point to $x$ among $S$. Clearly, $\xi$ is well-defined in a small neighborhood of $S$. By possible rescaling, we assume that $B$ is contained in such a neighborhood. Then we define the reflection map $\sigma$ across $S$ by
	\[\sigma(x)=2\xi(x)-x.\]
	We have $\sigma^2=\mathrm{id}$ and define
	\[\wti\Omega=\Omega-\sigma_\#\Omega.\]
	Thus $\wti \Omega\in\mc C(V)$, where $V=B^+\cup(S\cap B)\cap \sigma(B^+)$. By possible rescaling, we may assume that $B_{3R}\subset B$. Set
	\[\Omega'=\wti\Omega\llcorner B_{R}.\]
	
	We also denote by $\wti B_r(p)=\sigma(B_r(p))$. By shrinking $r$ if it is necesssary, we can assume that there exists $\kappa>1$ so that
	\begin{equation}\label{eq:lip sigma}
		\mathrm{Lip}(\sigma|_{\wti B_{2r}(p)})\leq 1+\kappa r.
	\end{equation}
	
	\begin{proof}[Proof of Theorem \ref{thm:regularity for minimizer}]
		
		Take $p\in \spt(\partial\Omega')$. Then for $r<(R-|p|)/3$ and $\Lambda\in \mc C(V)$ so that $\spt\Vert \Lambda-\Omega'\Vert\subset B_r(p)$, we have
		\begin{align}
			\label{eq:partial Omega'}\mu_{\partial \Omega'}(B_r(p))
			&=\mu_{\partial\Omega'}(B_r(p)\cap B^+_R)+\mu_{\partial \Omega'}(B_r(p)\cap \sigma(B^+_R))\\
			&=\mu_{\partial \Omega}(B_r(p))+\mu_{\sigma_{\#}(\partial\Omega)}\big(\sigma(\wti B_r(p)\cap B_R^+)\big)\nonumber\nonumber\\
			&\leq\mu_{\partial \Omega}(B_r(p))+(1+\kappa r)^n\mu_{\partial\Omega}(\wti B_r(p))\nonumber\\
			&\leq\mu_{\partial \Omega}(B_r(p))+\mu_{\partial\Omega}(\wti B_r(p))+\beta(n)\kappa r \mu_{\partial\Omega}(\wti B_r(p))\nonumber\\
			&\leq\mu_{\partial \Omega}(B_r(p))+\mu_{\partial\Omega}(\wti B_r(p))+(1+\kappa r)^n\beta(n)\kappa r \mu_{\sigma_\#(\partial\Omega)}( B_r(p)).\nonumber
		\end{align}
		Here \eqref{eq:lip sigma} is used in the first and the last inequalities. The constant $\beta(n)$ is allowed to  change from line to line. On the other hand,
		\begin{align}
			\label{eq:partial Lambda}	\mu_{\partial\Lambda}(B_r(p)) 
			&\geq\mu_{\partial \Lambda}(B_r(p)\cap B_R^+)+\mu_{\partial \Lambda}(B_r(p)\cap\sigma(B_R^+))\\
			&\geq\mu_{\partial \Lambda}(B_r(p)\cap B_R^+)+(1+\kappa r)^{-n}\mu_{\sigma_\#(\partial \Lambda)}\big(\wti B_r(p)\cap B_R^+\big)\nonumber\\
			&\geq\mu_{\partial \Lambda}(B_r(p)\cap B_R^+)+\mu_{\sigma_\#(\partial \Lambda)}(\wti B_r(p)\cap B_R^+)-\beta(n)\kappa r\mu_{\sigma_\#(\partial \Lambda)}(\wti B_r(p)\cap B_R^+)\nonumber\\
			&\geq\mu_{\partial \Lambda}(B_r(p)\cap B_R^+)+\mu_{\sigma_\#(\partial \Lambda)}(\wti B_r(p)\cap B_R^+)-\beta(n)\kappa r\mu_{\partial \Lambda}(B_r(p)).\nonumber
		\end{align}
		Here the third inequality is from \eqref{eq:lip sigma}. 

		Recall that $\Omega$ is $\mc A^h$-minimizing. Assume that $|h|\leq c$. Hence we have
		\begin{align*}
			\mu_{\partial\Lambda}(B_r(p)\cap B_R^+)-\mu_{\partial\Omega}(B_r(p))\geq-c\mc H^{n+1}((\Lambda\triangle\Omega)\cap B_R^+)\geq -c\mc H^{n+1}(\Lambda\triangle\Omega'),
		\end{align*}
		and
		\begin{align*}
			\mu_{\sigma_\#(\partial\Lambda)}(\wti B_r(p)\cap B_R^+)-\mu_{\partial\Omega}(\wti B_r(p)) & \geq-c\mc H^{n+1}((\sigma_\#\Lambda\triangle\Omega)\cap \wti B_r(p)\cap  B_R^+) \\
			&\geq -c\mc H^{n+1}(\Lambda\triangle\Omega'),
		\end{align*}
		where $\triangle$ is the symmetric difference of two Cappccioppoli sets. The above two inequalities, together with \eqref{eq:partial Omega'} and \eqref{eq:partial Lambda} and the isoperimetric inequalities, imply that
		\begin{equation}\label{eq:estimate of minor perturbation}
			\mu_{\partial \Lambda}(B_r(p))-\mu_{\partial\Omega'}(B_r(p))\geq -\beta'(n,c)r[\mu_{\partial \Lambda}(B_r(p))+\mu_{\partial\Omega'}(B_r(p))].
		\end{equation}
		Here $\beta'(n,c)$ is a constant depends only on $c$ and $n$.
		
		
		Applying (\ref{eq:estimate of minor perturbation}) to \cite{Mor03}*{Corollary 3.7, 3.8}, $\partial \Omega'\llcorner B_{r}(p)$ is $C^{1,1/2}$. In particular, $\partial \Omega$ is a properly embedded free boundary $C^{1,1/2}$ hypersurface. Since $\partial \Omega$ is an $h$-hypersurface, the classical PDE argument implies the higher regularity.
	\end{proof}
	
	Note that such process works for any Riemannian manifold.  Let $(M^{n+1}, \partial M, g)$ be compact Riemannian manifold with boundary which is isometrically embedded into a closed manifold $(\wti M, g)$. Recall that $B_r(p)$ is denoted by the geodesic ball of $\wti M$ with radius $r$ and center at $p$.  We then have the following regularity theorem:
	\begin{theorem}\label{thm:regularity of h-minimizer}
		Given $\Omega\in\mc C(M)$, $p\in\spt\Vert\partial \Omega\Vert$, and some small $s > 0$, suppose that $\Omega\llcorner B_s(p)$ minimizes the $\mc A^h$-functional: that is, for any other $\Lambda\in\mc C(M)$ with $\spt\Vert\Lambda - \Omega\Vert\subset B_s(p)\cap M$, we have $\mc A^h(\Lambda)\geq \mc A^h(\Omega)$. Then except for a set of Hausdorff dimension at most $(n-7)$, $\partial\Omega\llcorner B_s (p)$ is a properly embedded hypersurface with free boundary on $\partial M$, and is real analytic if the ambient metric on $M$ is real analytic.
	\end{theorem}

	\subsection{Estimates of touching sets}\label{subsec:estimates of touching sets}
	An immersed hypersurface is {\em almost embedded} if it locally decomposes into smooth embedded sheets that touch other sheets or $\partial M$ but do not cross.  A hypersurface is {\em properly embedded} if it is almost embedded and no sheets touch other sheets or $\partial M$. 
	
	Let $\mc S =\mc S(g)$ be the collection of all Morse functions $h$ such that the zero set $\Sigma_0 = \{h = 0\}$ is a compact, smoothly embedded hypersurface so that 
	\begin{itemize}
		\item $\Sigma_0$ is transverse to $\partial M$ and the mean curvature of $\Sigma_0$ vanishes to at most finite order;
		\item $\{x\in\partial M:H_{\partial M}(x)=h(x)\text{ or } H_{\partial M}(x)=-h(x)\}$ is contained in an $(n-1)$-dimensional submanifold of $\partial M$. 
	\end{itemize}

	\begin{lemma}\label{lem:Sg contains open dense}
		$\mc S(g)$ contains an open and dense subset in $C^\infty(M)$.
	\end{lemma}
	\begin{proof}
		For any $h_0\in C^\infty(M)$ and a neighborhood $\mc U_0$ of $h_0$ in $C^\infty(M)$, we are going to find an open subset of $\mc U_1\subset \mc U_0$ so that $\mc U_1\subset \mc S(g)$. First, $h_0$ can be extended to a function $\wti h_0\in C^\infty(\wti M)$. We can also extend $H_{\partial M}(x)$ to a smooth function $\wti H\in C^{\infty}(\wti M)$.  Second, we take a neighborhood $\wti {\mc U}$ of $\wti h_0$ in $C^\infty(\wti M)$ so that $\wti h|_M\in \mc U_0$ if $\wti h\in \wti{\mc U}$.

		By \cite{ZZ18}*{Proposition 3.8}, there exists a Morse function $\wti h_1\in\wti{\mc U}$ so that $\{\wti h_1=0\}$ is a closed, embedded hypersurface with mean curvature vanishing at most finite order. By perturbing $\wti h_1$ slightly, according to the generic existence of Morse functions, we can find $\wti h_2\in \wti{\mc U}$ so that
		\begin{itemize}
			\item $\wti h_2$, $\wti H\pm \wti h_2$ are all Morse functions on $\wti M$;
			\item $\{\wti h_2=0\}$ is a closed, embedded hypersurface which is transverse to $\partial M$ and has mean curvature vanishing at most finite order;
			\item $\{\wti H(x)=\wti h_2(x)\}$ and $\{\wti H(x)=-\wti h_2(x)\}$ are smooth, embedded hypersurfaces and transverse to $\partial M$. 
		\end{itemize}
		Denote by $h=\wti h_2|_M$. Then the last item implies that $\{x\in\partial M:H_{\partial M}(x)=h(x)\text{ or }-h(x)\}$ is contained in an $(n-1)$-dimensional submanifold of $\partial M$. Thus $h\in\mc S(g)$. Moreover, by the choice of $\wti h_2$, we have that $h$ and $\wti H|_M\pm h$ are Morse functions on $M$ with $0$ as regular value. Thus we conclude that for any $h'$ in a $C^\infty$ neighborhood $\mc U_1$ of $h$ in $C^\infty(M)$, $h'$ and $\wti H|_M\pm h'$ are still Morse functions with $0$ as regular value. $\mc U_1$ is the desired open set and Lemma \ref{lem:Sg contains open dense} is proved.
	\end{proof}

	\medskip
	In the next, we prove that for each $h\in \mc S(g)$, the touching set of an almost embedded free boundary $h$-hypersurface has dimension less than or equal to $(n-1)$.
	\begin{proposition}\label{prop:estimate of touching set}
		Let $h\in \mc S(g)$ and $\Sigma_1,\Sigma_2$ are two different connected, embedded, free boundary $h$-hypersurfaces in a connected open set $U$. Then $\Sigma_1\cap \Sigma_2$ and $\Sigma_1\cap\partial M$ are contained in a countable union of connected, smoothly embedded, $(n-1)$-dimensional submanifolds.
	\end{proposition}
	\begin{proof}
		Recall that the argument in \cite{ZZ18}*{Theorem 3.11} implies that $\Sigma_1\cap\Sigma_2\cap \mathrm{Int}(M)$ is contained in a countable union of connected, smoothly embedded, $(n-1)$-dimensional submanifolds. Note that $\Sigma_1\cap\Sigma_2\subset(\Sigma_1\cap \partial M)\cup (\Sigma_1\cap\Sigma_2\cap \mathrm{Int}(M))$. Thus it suffices to prove that $\Sigma_1\cap \partial M$ is contained in the submanifolds as described.
		
		To do this, we first take $p\in\mathrm{Int}\Sigma_1\cap\partial M$ so that $h(p)\pm H(p)\neq 0$. Then there exists a neighborhood $U(p)\subset M$ of $p$ so that $\Sigma_1\cap U(p)$ can be written as a graph over $\partial M\cap U(p)$. Denote by $u_1$ the graph function. By shrinking the neighborhood, we can also assume that
		\begin{equation}\label{eq:H-h nozero}
			H_{\partial M}(x)+ h(x,u_1(x))\neq 0 \  \text{ and } \  H_{\partial M}(x)-h(x,u_1(x))\neq 0
		\end{equation}
		for all $x\in \partial M\cap U(p)$. Note that such a function satisfies an inhomogeneous linear elliptic PDE of the form
		\[ Lu_1=H_{\partial M}(x)+h(x,u_1(x)) \  \text{ or } \  Lu_1=H_{\partial M}(x)-h(x,u_1(x)).\]
		Together with \eqref{eq:H-h nozero}, then the Hessian of $u_1$ at $p$ has rank at least 1. The implicit function theorem then implies that, on a possibly smaller neighborhood of $p$, the touching set $\{u_1 = \nabla u_1 = 0\}$ is contained in an $(n- 1)$-dimensional submanifold; see \cite{ZZ17}*{Lemma 2.8} for more details. Therefore, we conclude that 
		\begin{equation}\label{eq:touching A}
			\{x\in\partial M\cap \mathrm{Int}\Sigma_1:H_{\partial M}\neq h(x) \text{ and } H_{\partial M}\neq -h(x)\}
		\end{equation}
		is contained in the submanifolds as described in the proposition. Recall that 
		\begin{equation}\label{eq:touching B}
			\{x\in\partial M:H_{\partial M}(x)=h(x)\text{ or } H_{\partial M}(x)=-h(x)\}
		\end{equation}
		is contained in an $(n-1)$-dimensional submanifold of $\partial M$. Then Proposition \ref{prop:estimate of touching set} follows from the fact that $\Sigma_1\cap \partial M$ is contained in the union of \eqref{eq:touching A}, \eqref{eq:touching B} and $\partial \Sigma_1$.
	\end{proof}

	\subsection{Compactness of free boundary $h$-hypersurfaces}
	Recall that an almost embedded hypersurface $\Sigma$ is called a {\em free boundary $h$-hypersurface} if $H_\Sigma=h|_\Sigma$ and $\Sigma$ meets $\partial M$ orthogonally along $\partial\Sigma$. Given $h\in \mc S(g)$, denote by $\mc P^h$ the collection of free boundary $h$-hypersurfaces such that $\llbracket\Sigma\rrbracket=\partial \Omega$ for some open set $\Omega\subset M$.

	Note that when $h\in\mc S(g)$, the min-max free boundary $h$-hypersurfaces produced in Theorem \ref{thm:regularity of h minmax} satisfy the above requirements. Indeed, such $\llbracket\Sigma\rrbracket=\partial\Omega$ is a critical point of the weighted $\A^h$ functional:
	\[\A^h(\Omega)=\Area(\Sigma)-\int_{\Omega} h \,d\mc H^{n+1}.\]
	The second variation formula for $\A^h$ along normal vector field $X=\varphi \nu$ is given by
	\begin{align}
		\label{eq:2nd variation of Ah}&\delta^2\A^h|_\Omega(X,X)=\mathrm{II}_{\Sigma}(\varphi,\varphi)\\
		=&\int_\Sigma(|\nabla\varphi|^2-(\Ric(\nu,\nu)+|A|^2+\partial_\nu h)\varphi^2)\,d\mu_\Sigma-\int_{\partial\Sigma}h^{\partial M}(\nu,\nu)\varphi^2\,d\mu_{\partial\Sigma}.\nonumber
	\end{align}
	In the above formula, $\nabla\varphi$ is the gradient of $\varphi$ on $\Sigma$; $\mathrm{Ric}$ is the Ricci curvature of $M$; $A$ and $h^{\partial M}$ are the second fundamental forms of $\Sigma$ and $\partial M$ with normal vector fields $\nu$ and $\nu_{\partial M}$, respectively. 
	
We remark that in \eqref{eq:2nd variation of Ah}, $\Pi_\Sigma(\cdot,\cdot)$ can also be defined for any immersed free boundary $h$-hypersurfaces and it is a quadratic form on the space of $C^\infty$-functions on $\Sigma$ (not $\phi(\Sigma)$).
		
		The {\em Jacobi field of $\Sigma$} is defined to be a smooth function $\varphi$ on $\Sigma$ (not $\phi(\Sigma)$) so that
	\begin{equation}
		\left\{\begin{aligned}
			&(\Delta+|A|^2+\Ric(\nu,\nu)+\partial_\nu h)\varphi=0 \ \ \  \text{on } \Sigma,\\
			&\frac{\partial\varphi}{\partial \bm\eta}=h^{\partial M}(\nu,\nu)\varphi \ \ \ \ \ \ \ \ \ \ \ \ \ \ \ \ \ \ \ \ \ \ \  \text{ on }  \partial \Sigma.
		\end{aligned}\right.
	\end{equation}

	The classical Morse index for $\Sigma$ is defined to be the number of negative eigenvalues of the above quadratic form. However, since $\Sigma$ may touch itself and the boundary of $M$, a weaker version of the index is needed. Such a concept was introduced by Zhou \cite[Definition 2.1, 2.3]{Zhou19} for closed $h$-hypersurfaces based on Marques-Neves \cite{MN16}*{Definition 4.1}.
	
	\begin{definition}\label{def:k-unstable}
		Given $\Sigma\in \mc P^h$ with $\Sigma=\partial \Omega$, $k\in\mb N$ and $\epsilon\geq 0$, we say that $\Sigma$ is {\em$k$-unstable in an $\epsilon$-neighborhood} if there exist $0<c_0<1$ and a smooth family $\{F_v\}_{v\in\oB^k}\subset \mathrm{Diff}(M)$ with $F_0=\Id, F_{-v}=F_v^{-1}$ for all $v\in\oB^k$ (the standard $k$-dimensional unit ball in $\mb R^k$) such that, for any $\Omega'\in\B^\F_{2\epsilon}(\Omega)$, the smooth function:
		\[\A_{\Omega'}^h:\oB^k\rightarrow[0,+\infty),\ \ \ \ \A^h_{\Omega'}(v)=\A^h(F_v(\Omega'))\]
		satisfies
		\begin{itemize}
			\item $\A^h_{\Omega'}$ has a unique maximum at $m(\Omega')\in B_{c_0/\sqrt {10}}^k(0)$;
			\item $-\frac{1}{c_0}\Id\leq D^2\A^h_{\Omega'}(u)\leq -c_0\Id$ for all $u\in\oB^k$.
		\end{itemize} 
	\end{definition}
	Since $\Sigma$ is a critical point of $\A^h$, necessarily $m(\Omega)=0$.
	
When $h\equiv 0$, this reduces to the $k$-unstable notion for free boundary minimal hypersurfaces defined in \cite[Definition 5.5]{GLWZ19}.
	
	\begin{definition}\label{def:Morse index}
		Assume that $\Sigma\in\mc P^h$ or $\Sigma$ is a free boundary minimal hypersurface. Given $k\in\mb N$, we say that {\em the weak Morse index of $\Sigma$ is bounded (from above) by $k$}, denoted as 
		\[\mathrm{index}_w{(\Sigma)}\leq k,\]
		if $\Sigma$ is not $j$-unstable in 0-neighborhood for any $j\geq k+1$. $\Sigma$ is said to be {\em weakly stable} if $\mathrm{index}_w(\Sigma)=0$.
	\end{definition}
	
	\begin{remark}\label{rmk:weak index}
		We make several remarks:
		\begin{itemize}
			\item If $\Sigma\in \mc P^h$ is $k$-unstable in a 0-neighborhood, then it is $k$-unstable in an $\epsilon$-neighborhood for some $\epsilon>0$;
			\item All the concepts can be localized to an open subset $U\subset M$ by using $\mathrm{Diff}(U)$ in place of $\mathrm{Diff}(M)$;
			\item If $\Sigma\in \mc P^h$ is properly embedded, then $\Sigma$ is $k$-unstable if and only if its classical Morse index is $\geq k$.
		\end{itemize}
	\end{remark}
	
	We also have the following curvature estimates.
	\begin{theorem}[Curvature estimates for weakly stable free boundary $h$-hypersurfaces]\label{thm:curvature estimates}
		Let $3\leq (n+1)\leq 7$ and $V, U\subset M$ be two relatively open subset so that $\overline V\subset U$. Let $\Sigma\in \mc P^h$ be weakly stable in $U$ with $\Area(\Sigma)\leq C$, then there exists $C_1=C_1(n, M,U,V,\Vert h\Vert_{C^3}, C)$, such that 
		\[|A|(x)\leq C_1\text{\ for all } x\in \Sigma\cap V.\]
	\end{theorem}
	The proof is the same as the free boundary minimal cases \cite{GWZ18}*{Theorem 3.2}. 
	
	\bigskip
	Given $h\in\mc S(g),\Lambda>0$ and $I\in\mb N$, let 
	\begin{equation}\label{eq:def of Ph}
		\mc P^h(\Lambda, I):=\{\Sigma\in\mc P^h:\Area(\Sigma)\leq \Lambda,\mathrm{index}_w(\Sigma)\leq I\}.
	\end{equation}
	The main purpose is to prove this theorem:
	\begin{theorem}[Compactness for free boundary $h$-hypersurfaces]\label{thm:compactness for FPMC}
		Let $(M^{n+1},\partial M,g)$ be a compact Riemannian manifold with boundary of dimension $3\leq (n+1)\leq 7$. Assume that $\{h_k\}_{k\in\mb N}$ is a sequence of smooth functions in $\mc S(g)$ such that $\lim_{k\rightarrow\infty}h_k=h_\infty$ in smooth topology, where $h_\infty\in\mc S(g)$ or $h_\infty=0$. Let $\{\Sigma_k\}_{k\in\mb N}$ be a sequence of hypersurfaces such that $\Sigma_k\in\mc P^{h_k}(\Lambda,I)$ for some fixed $\Lambda>0$ and $I\in \mb N$. Then,
		\begin{enumerate}[label=(\roman*)]
			\item \label{compactness thm:smooth limit} up to a subsequence, there exists a smooth, compact, almost embedded free boundary $h_\infty$-hypersurface $\Sigma_{\infty}$ such that $\Sigma_k\rightarrow\Sigma_\infty$ (possibly with integer multiplicity) in the varifold sense, and hence also in the Hausdorff distance by monotonicity formula;
			\item \label{compactness thm:locally smoothly convergence} there exists a finite set of points $\mc Y\subset \Sigma_\infty$ with $\sharp(\mc Y)\leq I$, such that the convergence of $\Sigma_k\rightarrow\Sigma_\infty$ is locally smooth and graphical on $\Sigma_\infty\setminus \mc Y$;
			\item \label{compactness thm:generic multiplicity one convergence} if $h_\infty\in\mc S(g)$, then the multiplicity of $\Sigma_\infty$ is 1, and $\Sigma_\infty\in\mc P^{h_\infty}(\Lambda,I)$;
			\item \label{compactness thm:smooth and proper} assuming $\Sigma_k\neq \Sigma_\infty$ eventually and $h_k=h_\infty=h\in\mc S(g)$ for all $k$ and $\Sigma_k$ smoothly converges to $\Sigma_\infty$, then $\mc Y=\emptyset$, and $\Sigma_\infty$ has a non-trivial Jacobi field; 
			\item \label{compactness thm:index decreasing} if $h_k=h_\infty=h\in\mc S(g)$ and the convergence is not smooth, then $\mc Y$ is not empty and $\Sigma_\infty$ has strictly smaller weak Morse index than $\Sigma_k$ for all sufficiently large $k$;
			\item \label{compactness thm:limit index bound} if $h_\infty\equiv0$ and $\Sigma_\infty$ is properly embedded, then the classical Morse index of $\Sigma_\infty$ satisfies $\mathrm{index}(\Sigma_\infty)\leq I$ (without counting multiplicity)
		\end{enumerate}
	\end{theorem}
	
	\begin{proof}
		The proof follows essentially the same way as \cite{Sharp17}*{Theorem 2.3} and \cite{GWZ18}*{Theorem 4.1}; we will only provide necessary modifications.
		
		\bigskip	
		{\noindent\bf{Part 1:}}
		The same argument in \cite{Zhou19}*{Theorem 2.6}, by replacing \cite{Zhou19}*{Theorem 2.5} with Theorem \ref{thm:curvature estimates}, implies that $\Sigma_k$ converges locally smoothly and graphically to an almost embedded free boundary $h_\infty$-hypersurface $\Sigma_\infty$ (possibly with integer multiplicity) away from at most $I$ points, which we denote by $\mc Y$.	Now we prove that $\mc Y$ are all removable.
		
		\medskip
		{\noindent\em Case 1: Suppose $p$ is in the closure of $\Sigma$ and $p\notin \partial M$}.
		
		\medskip	
		Then the argument is the same as that in \cite{Zhou19}*{Theorem 2.7, Part 1}.
		
		\medskip
		{\noindent\em Case 2: Suppose that $p\in \partial M$.}
		
		\medskip
		Using the boundary removable singularity result Theorem \ref{thm:removable touching set}, we see that $p$ is also a removable singularity.

		Up to here, we have finished proving \ref{compactness thm:smooth limit} and \ref{compactness thm:locally smoothly convergence}. The argument in \cite{Zhou19}*{Theorem 2.6(iii)(v)} also works for \ref{compactness thm:generic multiplicity one convergence}\ref{compactness thm:limit index bound} here.

		\bigskip
		{\noindent\bf Part 2:} We now prove \ref{compactness thm:smooth and proper}. It suffices to produce a Jacobi field for the second variation $\delta^2\A^h$ along $\Sigma_\infty$. Recall that the Jacobi fields associated with $\delta^2\A^h$ along a free boundary $h$-hypersurface $\Sigma$ satisfy
		\begin{equation}\label{equation:Jacobi field}
			\left\{\begin{aligned}
				& L^h_\Sigma\varphi=0, \text{\ \ on $\Sigma$},\\
				&\frac{\partial \varphi}{\partial \bm\eta}=h^{\partial M}(\nu,\nu)\varphi, \text{\ \  along $\partial\Sigma,$}
			\end{aligned}\right.
		\end{equation}
		where $L^h_\Sigma=\Delta_\Sigma+\Ric(\nu,\nu)+|A|+\partial_\nu h$ and $\bm\eta$ is the outward co-normal of $\Sigma$. 
		
		Recall that $(M,\partial M,g)$ can be isometrically embedded into a closed Riemannian manifold $(\wti M^{n+1},\wti g)$. Let $\mathfrak {\wti X}(M,\Sigma)$ be the space of vector fields $X\in \mathfrak X(\wti M)$ so that $X(p)\in T_p(\partial M)$ for $p$ in a small neighborhood of $\partial \Sigma$ in $\partial M$.
		
		Denote by $\phi_\infty:(\Sigma_\infty,\partial \Sigma_\infty)\hookrightarrow(M,\partial M)$ the immersion. Now we fix a relative open set $V$ of $\Sigma_\infty$ so that $\phi_\infty|_V$ is an embedding.

		Now let $X\in\mathfrak {\wti X}(M, \Sigma_\infty)$ be an extension of the unit normal vector field of $\phi_\infty(V)$. Let $\Phi(x, t)$ be the one-parameter family of diffeomorphisms of $\wti M$ associated with $X$, so that $\frac{\partial \Phi}{\partial t}(x, t) = X(\Phi(x, t))$. Let $V_\delta$ be the $\delta$-thickening of $V$ with respect to $\Phi$ so that
		\[V_\delta:= \{\Phi(x, t) : x\in V \text{ and } |t| < \delta\}.\]	
		Then for sufficiently large $k$, there exists $\varphi_k(x)\in C^\infty(\Sigma_\infty)$ so that 
		\begin{gather*}
			\phi_k(\Sigma_k)\cap V_\delta=\{\Phi(x,\varphi_k(x)):x\in \phi_\infty(V)\},
		\end{gather*}
		where $\phi_k:(\Sigma_k,\partial\Sigma_k)\hookrightarrow(M,\partial M,g)$ is the immersion map. In the following, we often omit $\phi_k$ and $\phi_\infty$ for simplicity when there is no ambiguity.
		
		Denote by $\Phi^k(x,t)=\Phi(x,t\varphi_k(x))$ and 
		\begin{equation*}
			V_k(t)=\{\Phi^k(x, t):x\in V\}, \ \ \ \  R_k(t)=\{\Phi^k(x, t):x\in V\cap\partial \Sigma_\infty\}.
		\end{equation*}
		Then we have $V_k(0)=V$, $V_k(1)=\Sigma_k\cap V_\delta$.
		
		Now consider any vector field $Z\in\mathfrak {\wti X}(M,\Sigma_\infty)$ such that $|Z|+|\nabla Z|\leq 1$ and $Z|_{\wti M\setminus V_\delta}=0$. Let $(\Psi_t)_t$ be the associated one-parameter family of diffeomorphisms of $\wti M$. 
		From the fact that $\Sigma_k$ and $\Sigma_\infty$ are $h$-hypersurfaces, we have 
		\[\int_0^1\frac{d}{dt}\Big[\int_{V_k(t)}\dv Z-h\langle Z,\nu\rangle\,d\mu_{V_k(t)}\Big]\,dt=0.\]
		
		The computation in Appendix \ref{section:2nd variation} gives that 
		\begin{align*}
			0=&\int_0^1\frac{d}{dt}\Big[\int_{V_k(t)}\dv Z-h\langle Z,\nu\rangle\,d\mu_{V_k(t)}\Big]\,dt\\
			=&\int_0^1\Big[\int_{V_k(t)}\Big(\langle\nabla^\perp(X^\perp),\nabla^\perp(Z^{\perp})\rangle-\Ric(X^\perp,Z^\perp)-|A|^2\langle X^\perp,Z^\perp\rangle\\
			&-(\partial_\nu h)\langle X^\perp,Z^\perp\rangle\Big)\,d\mc H^n+ \int_{R_k(t)}\langle\nabla_{X^\perp}Z^\perp,\nu_{\partial M}\rangle\,d\mc H^{n-1}+\\
			&+\int_{V_k(t)}\wti \Xi_1(X,Z,\mf H)\, d\mc H^{n}+\int_{R_k(t)}\wti\Xi_2(X,Z,\mf H,\bm\eta,\nu_{\partial M})\, d\mc H^{n-1}\Big]\,dt.
		\end{align*}
		Here we used the assumption of $Z|_{\wti M\setminus V_\delta}=0$ in the last equality. By pulling everything back to $\Sigma_\infty$ and letting $Z|_{\Sigma_\infty}=\zeta\nu$, we obtain
		\begin{align*}
			0=&\int_0^1\Big[\int_{V}\Big(\langle\nabla\varphi_k,\nabla\zeta\rangle- \Ric(\nu,\nu)\varphi_k\zeta-|A^{\Sigma_\infty}|^2\varphi_k\zeta-(\partial_\nu h)\varphi_k\zeta\Big)\,d\mc H^n\\
			&-\int_{V\cap \partial \Sigma_\infty}h^{\partial M}(\nu,\nu)\varphi_k\zeta\,d\mc H^{n-1}+\int_{V}\wti W_k(t)(\varphi_k,\zeta)\, d\mc H^{n}+\\
			&+\int_{V\cap\partial\Sigma_\infty}\wti w_k(t)(\varphi_k,\zeta)\, d\mc H^{n-1}\Big]\,dt,
		\end{align*}
		where 
		\begin{gather*}
			\wti W_k(t)(\varphi_k,\zeta)\leq \epsilon_k(\Sigma_\infty,\wti M)|\varphi_k|, \ \ \ \wti w_k(t)(\varphi_k,\zeta)\leq \epsilon_k(\Sigma_\infty,\wti M)|\varphi_k|.
		\end{gather*}
		Here $\epsilon_k\rightarrow 0$ uniformly as $k\rightarrow\infty$ and we used that $|Z|+|\nabla Z|\leq 1$. Now letting $W_k=\int_0^1\wti W_k(t)dt$ and $w_k=\int_0^1 \wti w_k(t)dt$, by Fubini theorem, we have
		\begin{align}
			\label{eq:good in embedded piece}	0=&\int_{\Sigma_\infty}\Big(\langle\nabla\varphi_k,\nabla\zeta\rangle- \Ric(\nu,\nu)\varphi_k\zeta-|A^{\Sigma_\infty}|^2\varphi_k\zeta-(\partial_\nu h)\varphi_k\zeta\Big)\,d\mc H^n\\
			&-\int_{\partial \Sigma_\infty}h^{\partial M}(\nu,\nu)\varphi_k\zeta\,d\mc H^{n-1}+\int_{\Sigma_\infty} W_k(\varphi_k,\zeta)\, d\mc H^{n}+\int_{\partial\Sigma_\infty} w_k(\varphi_k,\zeta)\, d\mc H^{n-1}.\nonumber
		\end{align}
		Here we used the fact that $\zeta=0$ on $\Sigma_\infty\setminus V$.
		
		Now we take finitely many relatively open subsets $\{V_\alpha\}$ of $\Sigma_\infty$ so that $\phi_\infty|_{V_\alpha}$ is an embedding for each $j$ and 
		\[ \bigcup_\alpha V_\alpha=\Sigma_\infty.\] 
		Then we can also find finitely many non-negative cut-off functions $\zeta_\alpha\in C^\infty(\Sigma_\infty)$ so that $\zeta_\alpha=0$ on $\Sigma_\infty\setminus V_\alpha$ and $\sum \zeta_\alpha=1$. Then by \eqref{eq:good in embedded piece}, for each $\alpha$ and each $u\in C^\infty(\Sigma_\infty)$,
		\begin{align*}
			0=&\int_{\Sigma_\infty}\Big(\langle\nabla\varphi_k,\nabla(\zeta_\alpha u)\rangle- \Ric(\nu,\nu)\varphi_k\zeta_\alpha u-|A^{\Sigma_\infty}|^2\varphi_k\zeta_\alpha u-(\partial_\nu h)\varphi_k\zeta_\alpha u\Big)\,d\mc H^n\\
			&-\int_{\partial \Sigma_\infty}h^{\partial M}(\nu,\nu)\varphi_k\zeta_\alpha u\,d\mc H^{n-1}+\int_{\Sigma_\infty} W_k(\varphi_k,\zeta_\alpha u)\, d\mc H^{n}+\int_{\partial\Sigma_\infty} w_k(\varphi_k,\zeta_\alpha u)\, d\mc H^{n-1}.\nonumber
		\end{align*}
		Adding all of them together, we have 
		\begin{align*}
			0=&\int_{\Sigma_\infty}\Big(\langle\nabla\varphi_k,\nabla u\rangle- \Ric(\nu,\nu)\varphi_ku-|A^{\Sigma_\infty}|^2\varphi_ku-(\partial_\nu h)\varphi_ku\Big)\,d\mc H^n\\
			&-\int_{\partial \Sigma_\infty}h^{\partial M}(\nu,\nu)\varphi_ku\,d\mc H^{n-1}+\sum_{\alpha }\Big[\int_{\Sigma_\infty} W_k(\varphi_k,\zeta_\alpha u)\, d\mc H^{n}+\int_{\partial\Sigma_\infty} w_k(\varphi_k,\zeta_\alpha u)\, d\mc H^{n-1}\Big].\nonumber
		\end{align*}

		Let $\wti \varphi_k=\varphi_k/\Vert \varphi_k\Vert_{L^2(\Sigma_\infty)}$. Then the standard PDE theory implies that $\wti\varphi_k$ converges smoothly to a nontrivial $\varphi\in C^\infty(\Sigma_\infty)$ satisfying equation (\ref{equation:Jacobi field}), so we finish proving \ref{compactness thm:smooth and proper}.

		\bigskip
		{\noindent\bf Part 3:}
		In this part, we prove \ref{compactness thm:index decreasing}. The process is similar to \cite{GLWZ19}*{Proposition 5.2}. 
		
		Assume $\Sigma_\infty$ has weak index $I$. If $\Sigma_k$ does not smoothly converge to $\Sigma_\infty$, then there exists a finite set $\mc Y\subset\Sigma_\infty$ so that for any $r>0$, $B_r(p)\cap \Sigma_k$ does not smoothly converge to $B_r(p)\cap\Sigma_\infty$ and $\Sigma_k$ smoothly converges to $\Sigma_\infty$ outside $B_r(\mc Y):=\cup_{p\in \mc Y}B_r(p)$. We now prove that $\mathrm{index}_w(\Sigma_k)\geq I+1$ for large $k$.
		
		By the Definition \ref{def:Morse index} of weak index, there exists $0<c_0<1$ and smooth family $\{F_v\}_{v\in \oB^I}\subset \mathrm{Diff}(M)$ with $F_0=\Id, F_{-v}=F_v^{-1}$ for all $v\in \oB^I$ such that for any $\Omega'\in \B^\F_{2\epsilon}(\Omega_\infty)$ the smooth function:
		\[\A_{\Omega'}^h:\oB^I\rightarrow[0,+\infty),\ \ \ \ \A^h_{\Omega'}(v)=\A^h(F_v(\Omega'))\]
		satisfies
		\begin{itemize}
			\item $\A^h_{\Omega'}$ has a unique maximum at $m(\Omega')\in B_{c_0/\sqrt {10}}^I(0)$;
			\item $-\frac{1}{c_0}\Id\leq D^2\A^h_{\Omega'}(u)\leq -c_0\Id$ for all $u\in\oB^I$.
		\end{itemize} 
		Denote by $X_j(x)=\frac{d}{dt}|_{t=0}F_{te_j}(x)$. Then $X_j\in \mathfrak X(M,\partial M)$.

We can shrink $r$ so that $\{X_j|_{\Sigma_\infty\setminus B_r(\mc Y)}\}_{j=1}^I$ is still linearly independent. Let $\xi_r$ be a cut-off function satisfying $0\leq \xi_r\leq 1$ and $\xi|_{B_r(\mc Y)}=0$ and $\int_{\Sigma_\infty}|\nabla\xi_r|^2\rightarrow 0$ and $\xi_r\rightarrow 1$ as $r\rightarrow 0$. By Appendix \ref{section:cut-off}, we can shrink $r$ so that there exists $\epsilon>0$ satisfying
		\begin{equation}\label{eq:negative definite on Sigma infty outside single set}
			-(2/c_0)\sum_{j}a_j^2<\mathrm{II}_{\Sigma_\infty}(\sum_ja_j\xi_rX_j^\perp,\sum_ja_j\xi_rX_j^\perp)<-(c_0/2)\sum_ja_j^2,\ \ 1\leq j\leq I,
		\end{equation}
		for $\sum_ja_j^2\neq 0$. Recall that $\Sigma_k\setminus B_r(\mc Y)$ smoothly converges to $\Sigma_\infty\setminus B_r(\mc Y)$. Hence for sufficiently large $k$,
		\[-(2/c_0)\sum_{j}a_j^2<\mathrm{II}_{\Sigma_k}(\sum_ja_j\xi_rX_j^\perp,\sum_ja_j\xi_rX_j^\perp)<-(c_0/2)\sum_ja_j^2,\ \ 1\leq j\leq I.\]
		Here we used the fact that $\xi_r=0$ on $B_r(\mc Y)$.
		
		By assumptions, $B_r(p)\cap \Sigma_k$ does not smoothly converge to $B_r(p)\cap\Sigma_\infty$ for $p\in \mc Y$. Hence $\Sigma_k$ is not weak stable in $B_r(p)$. This implies that for each $k$, there exists $\{F'_v\}_{v\in[-1,1]}\subset \mathrm{Diff}(M)$ such that 
		\[ F'_v|_{M\setminus B_r(p)}=\Id,   \]
		and the smooth function $\mc A^h(\Omega')$ satisfying, for some $c_0'>0$,
		\begin{equation}\label{eq:new unstable vector}
			-\frac{1}{c_0'}\Id\leq D^2\A^h_{\Omega'}(u)\leq -c_0'\Id , 
		\end{equation}
		for all $u\in(-1,1)$ and $\Omega'\in \B^{\mf F}_{2\epsilon}(\Omega_k)$.

		For any $X\in \mathfrak X(M,\partial M)$, denote by $\Phi^X_t$ the flow of $X$. Now we define $\wti F:\oB^{I+1}\rightarrow \mathrm{Diff}(M)$ by 
		\[  \wti F_{v}=F_{v_0}'\circ \Phi_1^{\sum_{j=1}^Iv_j\xi_rX_j}.\]
		Here $v=(v_0,\cdots,v_I)$. Then for any $v\in\oB^{I+1}$ we have
		\begin{align*}
			\frac{d^2}{dt^2}\Big|_{t=0}\mc A^h(\wti F_{tv}(\Omega_k))&=\frac{d^2}{dt^2}\Big|_{t=0}\mc A^h(F'_{tv_0}(\Omega_k))+ \frac{d^2}{dt^2}\Big|_{t=0}\mc A^h(\wti F_{t(0,v_1,\cdots,v_I)}(\Omega_k))\\
			&=\frac{d^2}{dt^2}\Big|_{t=0}\mc A^h(F'_{tv_0}(\Omega_k))+\mathrm{II}_{\Sigma_k}(\sum_{j=1}^Iv_j\xi_rX_j^\perp,\sum_{j=1}^Iv_j\xi_rX_j^\perp).
		\end{align*}
		Together with \eqref{eq:negative definite on Sigma infty outside single set} and \eqref{eq:new unstable vector}, we conclude that $\mathrm{index}(\Sigma_k)\geq I+1$. This finishes the proof of \ref{compactness thm:index decreasing}.

	\end{proof}

	There is also a theorem analogous to the above one in the setting of changing ambient metrics on $(M,\partial M)$. The proof proceeds the same way when one realizes that the constant $C_1$ in Theorem 2.5 depends only on the $\Vert g\Vert_{C^4}$ when $g$ is allowed to change.
	\begin{theorem}\label{thm:compactness with changing metrics}
		Let $(M^{n+1},\partial M)$ be a closed manifold of dimension $3\leq (n + 1)\leq 7$, and $\{g_k \}_{k\in\mb N}$ be a sequence of metrics on $(M,\partial M)$ that converges smoothly to some limit metric $g$. Let $\{h_k\}_{k\in\mb N}$ be a sequence
		of smooth functions with $h_k\in\mc S(g_k)$ that converges smoothly to some limit $h_\infty \in C^\infty (M)$, where $h_\infty\in \mc S(g)$ or $h_\infty=0$. Let $\{\Sigma_k\}_{k\in\mb N}$ be a sequence of hypersurfaces with $\Sigma_k\in\mc P^{h_k}(\Lambda,I;g_k)$ for some fixed $\lambda > 0$ and $I\in\mb N$. Then there exists a smooth, compact, almost embedded free boundary $h_\infty$-hypersurface $\Sigma_\infty$, such that Theorem \ref{thm:compactness for FPMC}\ref{compactness thm:smooth limit}\ref{compactness thm:locally smoothly convergence}\ref{compactness thm:generic multiplicity one convergence} are satisfied.
	\end{theorem}
	
	\subsection{Generic existence of good pairs}\label{subsec:good pairs}
	Let $(\wti M^{n+1},g)$ be a closed Riemannian manifold of dimension $3\leq (n+1)\leq 7$. Let $M$ be a compact domain of $\wti M$ with smooth boundary. A pair $(\wti g,\wti h)$ consisting of a Riemannian metric $\wti g$ and smooth function $\wti h\in C^\infty(\wti M)$ is called a {\em good pair related to $M$}, if
	\begin{enumerate}
		\item $\wti h$ is a Morse function;
		\item\label{enu:good pair:finite order} the zero set $\{\wti h=0\}$ is a smoothly embedded closed hypersurface in $\wti M$, which is transverse to $\partial M$ and has mean curvature $\wti H$ vanishing to at most finite order;
		\item\label{enu:good pair:h pm H}  $\{x\in\partial M:\wti h|_{\partial M}=\pm H_{\partial M}\}$ has dimension less than or equal to $(n-1)$;
		\item \label{enu:good pair:h-bumpy}$\wti g$ is {\em bumpy for $\mc A^h$}, i.e., every almost embedded prescribed mean curvature hypersurface in $\wti M$ with free boundary on $\partial M$ is non-degenerate.
	\end{enumerate}
	Clearly, if $(\wti g,\wti h)$ is a good pair related to $M$, then $\wti h|_M\in \mc S(\wti g|_M)$; (see Section \ref{subsec:estimates of touching sets}). In this subsection, we are going to prove the generic existence of good pairs related to $M$.
	
	\medskip
	Denote by $\wti {\mc S}_0$ the set of smooth functions $\wti h\in C^\infty(\wti M)$ such that 
	\begin{itemize}
		\item $\wti h$ is a Morse function;
		\item $\{\wti h=0\}$ is an embedded closed hypersurface in $\wti M$, which is transverse to $\partial M$.
	\end{itemize}
	$\wti {\mc S}_0$ is open and dense in $C^\infty(\wti M)$, and is independent of the choice of a metric.
	
	The following lemma is a generalization of \cite{Zhou19}*{Lemma 3.5} by the last author.
	\begin{lemma}\label{lem:bummpy metric}
		Given $\wti h\in\wti{\mc S}_0$, the set of Riemannian metrics $\wti g$ on $\wti M$ with $(\wti g,\wti h)$ being a good pair related to $M$ is generic in the Baire sense.
	\end{lemma}
	\begin{proof}
		Firstly, we prove that \eqref{enu:good pair:h pm H} is generic. Let $\wti{\bm \Gamma}_0$ be the set of metrics $\wti g$ so that $\{ x\in \partial M:\wti h|_{\partial M}=\pm \wti H_{\partial M}\}$ has dimension less than or equal to $(n-1)$. Denote by $\wti{\bm \Gamma}_1$ the set of metrics on $\wti M$ so that
		\[ \{x\in \partial M:H_{\partial M}(x)\pm h(x)=0\}\subset \{  \nabla^{\partial M}(H_{\partial M}\pm h|_{\partial M})\neq 0 \}.\]
		Clearly, $\wti{\bm \Gamma}_1\subset \wti{\bm\Gamma}_0$ and $\wti{\bm\Gamma}_1$ is open. We claim that $\wti{\bm\Gamma}_1$ is also dense, which would imply that \eqref{enu:good pair:h pm H} is generic.
		
		Indeed, given any $\wti g$ and an open neighborhood $\mc U$ of $\wti g$ in the space of metrics on $\wti M$, we can construct $\wti g'\in \mc U\cap \wti{\bm \Gamma}_1$ as follows. Let $r$ be the signed distance function to $\partial M$ under the metric $\wti g$. Let $\delta>0$ be small enough so that $\partial M$ has a tubular neighborhood $W$ and the exponential map from $\partial M\times(-\delta,\delta)$ to $W$ is a diffeomorphism. Fix a cut-off function $\eta:(-1,1)\rightarrow [0,1]$ satisfying 
		\[\text{$\eta(z)=1$ for $z\in[-\delta/2,\delta/2]$ and $\eta(z)=0$ for $|z|\geq\delta$}.\] Now let $\varphi\in C^{\infty}(\partial M)$ be small enough so that
		\begin{equation}\label{eq:zeros non-degenerate}
			\{x\in \partial M:n\varphi(x)+H_{\partial M}(x)\pm h(x)=0\}\subset \{  \nabla^{\partial M}(n\varphi+H_{\partial M}\pm h|_{\partial M})\neq 0 \}.
		\end{equation}
		This can be satisfied because the set of smooth Morse functions $u$ on $\partial M$ with empty singular set $\{u=\nabla u=0:x\in \partial M\}=\emptyset$ is open and dense.
		
		Define $\phi(x)=-r(x)\eta(r(x))\varphi(\pi(x))$, where $\pi$ is the projection of $W$ to its closest point in $\partial M$. Let $\wti g'=e^{2\phi}\wti g$. By taking $\varphi$ in a small enough neighborhood of $0$, we have that $\wti g'\in \mc U$. Then the second fundamental form of $\partial M$ with respect to $\wti g'$ is given by (see \citelist{\cite{Bes87}*{Section 1.163}\cite{IMN17}*{Proposition 2.3}})
		\[ A_{\partial M,\wti g'}=A_{\partial M,\wti g}+\wti g\cdot (\nabla \phi)^{\perp}.  \]    
		Recall that $\wti g|_{\partial M}=\wti g'|_{\partial M}$. Hence the mean curvature with respect to $\wti g'$ is
		\[ H_{\partial M,\wti g'}=  H_{\partial M,\wti g}+n\varphi. \]
		By \eqref{eq:zeros non-degenerate}, $\wti g'\in \wti{\bm\Gamma}_1$. Thus $\wti{\bm\Gamma}_1$ is dense.
		
		\medskip
		In the next, we prove that the set of metrics $\wti g$ under which $\{\wti h = 0\}$ has mean curvature vanishing to at most finite order is an open and sense subset. Clearly, it is open and it suffices to prove denseness. Let $\mc U$ be an open neighborhood of $\wti g$. Then from the above argument, for any $\varphi \in C^\infty( \{\wti h=0\} )$ small enough, we can find a metric $\wti g'\in \mc U$ so that $\wti g'|_{\{\wti h=0\}}=\wti g|_{\{\wti h=0\}}$ and  $\{\wti h=0\}$ has mean curvature $H_0+\varphi$, where $H_0$ is the mean curvature of $\{\wti h=0\}$ under $\wti g$. We can choose $\varphi$ so that $H_0+\varphi$ is a Morse function on $\{\wti h=0\}$, hence $H_0+\varphi$ vanishes to at most finite order. 
		This gives the denseness and we conclude that \eqref{enu:good pair:finite order} is generic.

		\medskip	
		It remains to prove that \eqref{enu:good pair:h-bumpy} is generic. The proof is the same with the Bumpy metric theorem \cite{Whi91}*{Theorem 2.2}. See also \cite{ACS17}*{Theorem 9} and an alternative version \cite{GWZ18}*{Theorem 2.8}.
	\end{proof}
	
	\medskip
	In the end of this subsection, we prove that the space of almost embedded free boundary $h$-hypersurfaces is countable for a good pair related to $M$.
	
	Given $h\in\mc S(g),\Lambda>0$ and $I\in\mb N$, recall that $\mc P^h(\Lambda,I)$ the set of $\Sigma\in\mc P^h$ satisfying $\Area(\Sigma)\leq \Lambda$ and its weak Morse index is bounded by $I$ from above.

	\begin{lemma}\label{lem:countable}
		Let $(\wti g,\wti h)$ be a good pair related to $M$. Denote by $h=\wti h|_M$ and $g=\wti g|_M$. Then $\mc P^h(\Lambda, I)$ is countable, and hence $\bigcup_{I\geq 0}\mc P^h(\Lambda, I)$ is countable.
	\end{lemma}
	\begin{proof}
		We prove it by an inductive method. By Theorem \ref{thm:compactness for FPMC}\ref{compactness thm:index decreasing}\ref{compactness thm:generic multiplicity one convergence}, there are only finitely many elements in $\mc P^h(\Lambda, 0)$. Hence $\mc P^h(\Lambda,0)$ contains only finitely many hypersurfaces.
		
		Assuming that $\mc P^h(\Lambda, I)$ is countable for some $I\geq 0$, we are going to prove that $\mc P^h(\Lambda, I+1)$ is also countable. Using Theorem \ref{thm:compactness for FPMC}\ref{compactness thm:index decreasing}\ref{compactness thm:generic multiplicity one convergence} again, we know that for any $r>0$, $\mc P^h(\Lambda, I+1)\setminus\bigcup_{\Sigma\in\mc P^h(\Lambda, I)}\B_r^\mf F(\Sigma)$ contains only finitely many elements. Therefore, $\mc P^h(\Lambda, I+1)$ is countable.
		
		By induction, we have proved that $\mc P^h(\Lambda,I)$ is countable for all $I\geq 0$.
	\end{proof}

	\subsection{Generic properness for free boundary minimal hypersurfaces}\label{subsec:generic properness}
	In this part, we review White's Generically Transversality Theorem and prove an adapted version for free boundary minimal hypersurfaces, i.e.  Theorem \ref{thm:intro:generic strongly bumpy}.
	
	We focus on the co-dimension one and embedded case. We first consider a closed manifold  $\wti M$  and a two-sided, closed, embedded hypersurface $N\subset \wti M$. (Later we will embed $M$ into $\wti M$ and let $N=\partial M$.) Let $\wti {\mathbf\Gamma}$ be the set of smooth Riemannian metrics on $\wti M$. Let $\mc {\wti M}$ be the space of all pairs $(\wti\gamma,\wti \Sigma)$ such that $\wti\gamma\in\wti {\mathbf\Gamma}$ and $(\wti\Sigma,\partial \wti\Sigma)$ is an embedded minimal hypersurface in $(\wti M,\wti\gamma)$ with (possibly empty) free boundary $\partial \wti \Sigma\subset N$. Recall that the projection $\wti\Pi:\mc{\wti M}\rightarrow\wti{\mathbf\Gamma}$ is defined as
	\[\wti\Pi(\wti\gamma,\wti\Sigma)=\wti\gamma.\]
	
	Denote by $\mc {\wti M}_{reg}$ the set of $(\wti \gamma,\wti \Sigma)\in\mc{\wti M}$ such that $\wti\Sigma$ is non-degenerate (with no nontrivial Jacobi field). Then by the work of White \citelist{\cite{Whi91}\cite{Whi17}} and Ambrozio-Carlotto-Sharp \cite{ACS17}, together with the fact that $\mc {\wti M}$ is second countable, we know that $\mc{\wti M}_{reg}$ is a countable union of open sets $U_j$ such that $\wti\Pi$ maps each $U_j$ homeomorphically onto an open subset of $\wti {\mathbf\Gamma}$. 

	Denote by $\mc {\wti M}_0$ the set of $(\wti\gamma,\wti\Sigma)\in\mc {\wti M}_{reg}$ such that $\wti \Sigma$ is transverse to $N$. 
	\begin{lemma}\label{lem:generic for ambient metrics}
		$\wti \Pi(\mc{\wti M}\setminus\mc{\wti M}_0)$ is meager in $\wti{\mathbf\Gamma}$. As a corollary, for generic metrics on $\wti M$, each embedded minimal hypersurface with free boundary on $N$ is transverse to $N$.
	\end{lemma}
	\begin{proof}
		The proof here is similar to \cite{Whi19}*{Corollary 5}. Recall that $\mc {\wti M}_{reg}$ is a countable union of open sets $U_j$ so that $\wti\Pi|_{U_j}$ is a homeomorphism. Now 
		\[\wti\Pi(\mc {\wti M}\setminus\mc {\wti M}_0)\subset \bigcup_{j=1}^\infty\Pi((\mc {\wti M}\setminus\mc {\wti M}_0)\cap  U_j)\cup\wti\Pi(\wti {\mc M}\setminus\wti{\mc M}_{reg}).\] 
		Now for any $(\wti \gamma,\wti\Sigma)\in(\wti{\mc M}\setminus\wti{\mc M}_0)\cap U_j$, we can always take a smooth vector field $X\in\mathfrak X(\wti M)$ so that $X|_N = f\n$ satisfying, (here $\n$ is the unit normal vector field of $N$),
		\begin{itemize}
			\item $X=0$ in a neighborhood $U_{\partial \wti \Sigma}$ of $\partial\wti \Sigma$ in $\wti M$;
			\item For any $x\in\mathrm{Int}(\wti\Sigma)\cap N$, $|f|= 1$ and $\nabla f=0$; (e.g. letting $f\equiv 1$ in a neighborhood of $\mathrm{Int}(\wti\Sigma)\cap N$).
		\end{itemize}
                 Denote by $(F_t)_{-1<t<1}$ the 1-parameter family of diffeomorphisms of $\wti M$ associated with $X$.
		\begin{claim}\label{claim:transverse by perturb}
			There exists $\delta>0$ so that for all $0<t<\delta$, $\wti\Sigma$ is transverse to $F_{t}(N)$. 
		\end{claim}
		\begin{proof}
		By the choice of $f$, we can take an open set $V\subset N$ containing $\mathrm{Int}(\wti\Sigma)\cap N$ such that $X(x)\neq 0$ for any $x\in V$. By taking $\delta_0>0$ small, denote by $U$ the set of 
			\[  \{F_t(x):x\in V, t\in (-\delta_0,\delta_0)\} ,\]
			which is diffeomorphic to $V\times (-\delta_0,\delta_0)$ by the map $F(x,t):=F_t(x)$. Define a continuous function $Z:U\rightarrow \mb R$ by
			\[ Z(F_t(x))=d^2(F_t(x),\wti \Sigma)+1-\langle \n_{F_t(x)},\nabla d\rangle^2.\]
			Here $\n_{F_t(x)}$ is the unit normal vector field of $F_t(N)$, and $d(\cdot)$ is the signed distance function to $\wti \Sigma$. It follows that $Z(F_t(x))\geq 0$. Moreover, if $Z(F_0(x_0))=0$, then $Z$ is smooth in a small neighborhood of $(x_0,0)$ and $\nabla ^2Z\big|_{F_0(x_0)}\geq 0$. Let $\{x^j\}_{j=1}^n$ be a normal coordinate system of $N$ at $F_0(x_0)$, and write $e_j=\partial/\partial x^j$.
			By a direct computation, we have
			\[  \nabla Z\big|_{F_0(x_0)}=0,\ \ \  \nabla^2Z\big|_{F_0(x_0)}(e_i,\partial_t)=0, \]
			where $\partial_t=dF(\frac{\partial }{\partial t})$. Moreover, 
			\begin{align*}
				\nabla^2Z(e_i,e_j)\Big|_{F_0(x_0)}&=-2\nabla_{e_i}\nabla _{e_j}\langle \n_{F_t(x)},\nabla d\rangle \Big|_{F_0(x_0)}\\
				&=2\langle \nabla_{e_i}\n-\nabla_{e_i}\nabla d,\nabla_{e_j}\n-\nabla_{e_j}\nabla d\rangle,
			\end{align*}
			and
			\begin{align*}
				&\ \frac{d^2}{dt^2}\Big|_{F_0(x_0)}Z(x,t)\\
				=&\ 2\Big[\frac{d}{dt}\Big|_{(x,t)=(x_0,0)}d(F_t(x),\wti \Sigma)\Big]^2-2\langle \n_{F_0(x_0)},\nabla d\rangle \cdot \frac{d^2}{dt^2}\Big|_{(x,t)=(x_0,0)}\langle \n_{F_t(x)},\nabla d\rangle\nonumber\\
				=&\ 2f^2=2.    \nonumber 
			\end{align*}

			Therefore, we conclude that in a small neighborhood of $(x_0,0)$,
			\begin{equation}\label{eq:greater than t2}
				Z(x,t)\geq t^2/2.
			\end{equation}
			
			Now we assume on the contrary that there exists a sequence of positive $t_i\rightarrow 0$ so that $\wti \Sigma$ is not transverse to $F_t(N)$. Then there exists $(x_i,t_i)$ so that $F_{t_i}(x_i)\in \wti \Sigma$ and $F_{t_i}(N)$ is tangent to $\wti \Sigma$. By the definition of $Z$, we have $Z(x_i,t_i)=0$. Assume that $\lim_{i\rightarrow\infty}x_i=x_0$, then $Z(x_0, 0)=0$. However, by \eqref{eq:greater than t2}, for sufficiently large $i$,
			\[   Z(x_i,t_i)\geq t_i^2/2. \] 
			This leads to a contradiction and hence Claim \ref{claim:transverse by perturb} is proved.
		\end{proof}
		
		Claim \ref{claim:transverse by perturb} gives that for $t>0$ small enough, $F_t^{-1}(\wti\Sigma)$ is transverse to $N$, hence $(F_{t}^*\wti\gamma,F_{t}^{-1}\wti\Sigma)\in\wti{\mc M}_0$. Note that $\wti{\mc M}_0$ is open. This is to say that $(\wti{\mc M}\setminus\wti{\mc M}_0)\cap U_j$ is nowhere dense. Because $\wti\Pi|_{U_j}$ is a homeomorphism, $\wti\Pi(\wti{\mc M}\setminus\wti{\mc M}_0)\cap U_j$ is nowhere dense. Hence $\wti \Pi(\mc{\wti M}\setminus\mc{\wti M}_0)$ is meager in $\wti{\mathbf\Gamma}$.
	\end{proof}
	
	Now we are ready to prove Theorem \ref{thm:intro:generic strongly bumpy}. Let $(M^{n+1},\partial M)$ be a smooth compact manifold with boundary and $3\leq(n+1)\leq 7$. Denote by $\mathbf\Gamma$ the space of smooth Riemannian metrics on $(M,\partial M)$. Let $\mc { M}$ be the space of all pairs $(\gamma,\Sigma)$ such that $\gamma\in\mathbf\Gamma$ and $(\Sigma,\partial \Sigma)$ is an almost properly embedded, free boundary minimal hypersurface in $( M,\partial M,\gamma)$. Recall that the projection $\Pi:\mc{ M}\rightarrow\mathbf\Gamma$ is defined as
	\[\Pi(\gamma,\Sigma)=\gamma.\]
	Denote by $\mc {M}_0$ the set of $(\gamma,\Sigma)\in\mc { M}$ such that $\Sigma$ is non-degenerate and proper.

	\begin{proof}[Proof of Theorem \ref{thm:intro:generic strongly bumpy}]
		We first prove that $\mathbf\Gamma\setminus\Pi(\mc M\setminus \mc M_0)$ is dense in $\mathbf\Gamma$. Indeed, for any $\gamma\in\mathbf\Gamma$, $(M^{n+1},\partial M,\gamma)$ can be isometrically embedded into a closed manifold $(\wti M^{n+1},\wti\gamma)$. By Lemma \ref{lem:generic for ambient metrics}, there exists a sequence of $\wti\gamma_j\rightarrow\wti \gamma$ such that every embedded minimal hypersurface in $(\wti M,\wti \gamma_j)$ with free boundary on $\partial M$ is non-degenerate and transverse to $\partial M$. When restricting to $(M,\partial M)$, $\wti\gamma_j|_{M}\rightarrow\gamma$ and $\wti\gamma_j|_M\in \mathbf\Gamma\setminus\Pi(\mc M\setminus \mc M_0)$ for each $j$.
		
		Now for any $C>0$, denote by $\mathbf\Gamma(C)$ the set of $\gamma\in\mathbf\Gamma$ so that every free boundary minimal hypersurface in $M$ with $\Area\leq C$ and $\mathrm{index}\leq C$ is non-degenerate and proper. Then the above argument gives that $\mathbf\Gamma(C)$ contains a dense set $\mathbf\Gamma\setminus\Pi(\mc M\setminus \mc M_0)$. 
		\begin{claim}\label{claim:contain open set}
			For any $C>0$, $\mathbf\Gamma(C)$ is open.
		\end{claim}
		\begin{proof}[Proof of Claim \ref{claim:contain open set}]
	Let $\gamma\in\mf \Gamma(C)$. Suppose for the contrary that $\{\gamma_j\}$ is a sequence of metrics on $(M,\partial M)$ and $\{\Sigma_j\}$ a sequence of free boundary minimal hypersurfaces in $(M,\partial M,\gamma_j)$ with $\gamma_j\rightarrow\gamma$, $\Area(\Sigma_j)\leq C$ and $\mathrm{Index}(\Sigma_j)\leq C$, but $\Sigma_j$ is either degenerate or improper. Then by the compactness theorems in \citelist{\cite{ACS17}\cite{GWZ18}}, up to a subsequence, $\Sigma_j$ locally smoothly converges to a free boundary minimal hypersurface $\Sigma$ away from a finite set in $(M,\partial M,\gamma)$ with $\Area\leq C$ and $\mathrm{Index}\leq  C $. Since $\gamma\in\mathbf\Gamma(C)$, $\Sigma$ is non-degenerate and proper. This implies that the convergence is smooth, and hence $\Sigma_j$ is also non-degenerate and proper for large $j$. This is a contradiction and hence Claim \ref{claim:contain open set} is proved.
		\end{proof}
		We thus have that $\mathbf\Gamma(C)$ contains an open dense set of $\mathbf\Gamma$. Recall that
		\[\mf \Gamma\setminus\Pi(\mc M\setminus \mc M_0)=\bigcap_{k=1}^{\infty}\mathbf\Gamma(k).\]
		Obviously, this is a generic subset of $\mathbf\Gamma$.
	\end{proof}
	
	\section{Relative Min-max theory for free boundary $h$-hypersurfaces}\label{sec:min-max for FPMC}
	Here we generalize multi-parameter min-max theory for prescribed mean curvature hypersurface to compact manifolds with boundary.

	\subsection{Notations for Min-max construction}\label{subsec:relative min-max}
	In this part, we describe the setup for min-max theory for free boundary $h$-hypersurfaces associated with multiple-parameter families in $\mathcal C(M)$. All the setups here are the same as those in \cite{Zhou19}*{\S 1.1}. First, we list some notations for cubical complex.
	\begin{basedescript}{\desclabelstyle{\multilinelabel}\desclabelwidth{6em}}
		\item [$I(1,j)$] the cubical complex on $I=[0,1]$ with 0-cells $\{ [k\cdot 3^{-j}]\}$ and 1-cells $\{ [k\cdot 3^{-j},(k+1)\cdot 3^{-j}]\}$;
		\item [$I(m,j)$] the cubical complex $I(1,j)\otimes I(1,j)\otimes\cdots\otimes I(1,j)$ ($m$ times);
		\item [$\n(k,j)$] the map from $I(k)_0$ to $I(j)_0$ defined as the unique element such that\\ $d(x,\n(k,j)(x)) = \inf \{d(x,y):y\in I(j)_0\}$, where $k\geq j$;
		\item [$X^k$] a cubical subcomplex of dimension $k\in \mathbb N$ in some $I(m,j)$;
		\item [$X(l)$]  the union of all cells of $I(m, l)$ whose support is contained in some cell of $X$;
		\item [$X(l)_q$] the set of all $q$-cells in $X(l)$.	
	\end{basedescript}
	
	Let $X^k$ be a cubical complex of dimension $k\in \mathbb N$ in some $I(m,j)$ and $Z\subset X$ be a cubical subcomplex. 
	
	Let $\Phi_0:X\rightarrow (\C(M),\F)$ be a continuous map (under the $\F$-topology on $\C(M)$). Let $\Pi$ be the collection of all sequences of continuous maps $\{\Phi_i:X\rightarrow\C(M)\}_{i\in\mb N}$ such that
	\begin{enumerate}
		\item each $\Phi_i$ is homotopic to $\Phi_0$ in the flat topology on $\C(M)$;
		\item there exist homotopy maps $\{\Psi_i:[0,1]\times X\rightarrow\C(M)\}_{i\in\mb N}$ which are continuous in the flat topology, $\Psi_i(0,\cdot)=\Phi_i$, $\Psi_i(1,\cdot)=\Phi_0$, and satisfy
		\[\limsup_{i\rightarrow\infty}\sup\{\F(\Psi_i(t,x),\Phi_0(x)):t\in[0,1],x\in Z\}=0.\]
	\end{enumerate}
	Given a pair $(X,Z)$ and $\Phi_0$ as above, $\{\Phi_i\}_{i\in \mb N}$ is called a {\em $(X,Z)$-homotopy sequence of mappings into $\C(M)$}, and $\Pi$ is called the {\em $(X,Z)$-homotopy class of $\Phi_0$}. Then we define the $h$-width by
	\[\mf L^h=\mf L^h(\Pi):=\inf_{\{\Phi_i\}\in \Pi}\limsup_{i\rightarrow\infty}\sup_{x\in X}\{\A^h(\Phi_i(x))\}.\]
	
	A sequence $\{\Phi_i\}_{i\in\mb N}\in \Pi$ is called a {\em minimizing sequence} if $\mf L^h(\{\Phi_i\})=\mf L^h(\Pi)$, where 
	\[\mf L^h(\{\Phi_i\}):=\limsup_{i\rightarrow\infty}\sup_{x\in X}\{\A^h(\Phi_i(x))\}.\]
	
	Given $\Phi_0$ and $\Pi$, by the same argument as \cite{Zhou19}*{Lemma 1.5}, there exists a minimizing sequence.
	
	\begin{definition}\label{def:critical set}
	If $\{\Phi_i\}_{i\in\mb N}$ is a minimizing sequence in $\Pi$, the {\em critical set} of $\{\Phi_i\}$ is defined by 
	\[\mf C(\{\Phi_i\})=\{V=\lim_{j\rightarrow\infty}|\partial \Phi_{i_j}(x_j)| \text{\ as varifolds : with\ }\lim_{j\rightarrow\infty}\A^h(\Phi_{i_j}(x_j))=\mf L^h(\Pi)\}.\]
\end{definition}

	\subsection{Discretization and Interpolation}
	We record several discretization and interpolation results developed by Marques-Neves \citelist{\cite{MN14}\cite{MN17}} (for closed manifolds), and later by Li and the last author \cite{LZ16} (for compact manifolds with boundary). Though these results were proven for sweepouts in $\mc Z_n (M, \mb Z)$ or $\mc Z_n (M, \mb Z_2)$, they work well for sweepouts in $\mc C(M)$ like in \cite{Zhou19}*{\S 1.3}. We will point out necessary modifications.
	
	We refer to \cite{ZZ18}*{Section 4} for the notion of discrete sweepouts. Though all definitions therein were made when $X = [0, 1]$, there is no change for discrete sweepouts on $X$.
	
	Recall that given a map $\phi: X(k)_0\rightarrow \mc C(M)$, the {\em fineness} of $\phi$ is defined as
	\[
	\mf f(\phi)=\sup\{\mc F(\phi(x)-\phi(y)) +\mf M(\partial \phi(x)-\partial \phi(y)): x, y \text{ are adjacent vertices in } X(k)_0\}.
	\]
	Here two vertices $x,y \in X(k)_0$ are {\em adjacent} if they belong to a common cell in $X(k)_1$ .
	
	\begin{definition}[\cite{MN17}*{\S 3.7}]\label{def:no concentration of mass}
		Given a continuous (in the flat topology) map $\Phi: X\rightarrow\mc C(M)$, we say that $\Phi$ has {\em no concentration of mass} if
		\[
		\lim_{r\rightarrow 0}\sup\{\Vert\partial \Phi(x)\Vert(B_r (p)): p\in M, x\in X\} = 0.
		\]
	\end{definition}
	
	The purpose of the next theorem is to construct discrete maps out of a continuous map in the flat topology.
	\begin{theorem}[Discretization, \cite{Zhou19}*{Theorem 1.11}]\label{thm:discretization}
		Let $\Phi: X\rightarrow\mc C(M)$ be a continuous map in the flat topology that has no concentration of mass, and $\sup_{x\in X}\mf M(\partial \Phi (x)) < +\infty$. Assume that $\Phi|_Z$ is continuous under the $\mf F$-topology. Then there exist a sequence of maps
		\[\phi_i: X(k_i)_0\rightarrow\mc C(M),\]
		and a sequence of homotopy maps:
		\[\psi_i: I(k_i)_0\times X(k_i)_0\rightarrow\mc C(M),\]
		with $k_i<k_{i+1}, \psi_i (0,\cdot) = \phi_{i-1}\circ\mf n(k_i, k_{i-1} )$,
		$\psi_i (1,\cdot) = \phi_i$, and a sequence of numbers $\{\delta_i\}_{i\in\mb N}\rightarrow 0$ such that
		\begin{enumerate}[label=(\roman*)]
			\item \label{item:dis:fineness} the fineness $\mf f(\psi_i)<\delta_i$;
			\item \label{item:dis:flat close} \[\sup\{\mc F(\psi_i(t,x)-\Phi (x)): t\in I(k_i)_0, x\in X(k_i)_0\}\leq\delta_i;\]
			\item \label{item:dis:M bound} for some sequence $l_i\rightarrow\infty$, with $l_i<k_i$,
			\[\mf M(\partial\psi_i(t, x))\leq\sup\{\mf M(\partial\Phi(y)) : x, y \in\alpha, \text{ for some } \alpha\in X(l_i)\} + \delta_i;\]
			and this directly implies that
			\[\sup\{\mf M(\partial\phi_i(x)): x\in X(k_0)_0 \} \leq\sup\{\mf M(\partial\Phi(x)): x\in X\} + \delta_i.\]
		
		As $\Phi|_Z$ is continuous in the $\mf F$-topology, we have from \ref{item:dis:M bound} that for all $t\in I(k_i)_0$ and $x\in Z(k_i)_0$,
		\[\mf M(\partial\psi_i(t, x))\leq\mf M(\partial \Phi(x)) + \eta_i\]
		with $\eta_i\rightarrow 0$ as $i\rightarrow\infty$. Applying \cite{LZ16}*{Lemma 3.13}, we get by \ref{item:dis:flat close} that
		
			\item 
			\[
			\sup\{\mf F(\psi_i (t, x), \Phi(x)) : t \in I(k_i)_0 , x\in Z(k_i)_0\}\rightarrow0, \text{ as } i \rightarrow\infty.\]
			
		Now given $h\in C^\infty (M)$, denoting $c = \sup_M |h|$, then we have from \ref{item:dis:flat close}\ref{item:dis:M bound} that
		
			\item \label{item:discretizaiton:bound of Ah} \[\mc A^h (\phi_i(x))\leq\sup\{\mc A^h(\Phi(y)):\alpha\in X(l_i), x, y \in\alpha\}+(1+c)\delta_i;\]
			and hence
			\[\sup\{\mc A^h(\phi_i(x)):x\in X(k_i)_0\}\leq\sup\{\mc A^h(\Phi(x)):x\in X \}+(1+c)\delta_i.\]
		\end{enumerate} 
	\end{theorem}
	\begin{proof}
		The last named author \citelist{\cite{Zhou17}*{Theorem 5.1}\cite{Zhou19}*{Theorem 1.11}} proved this for closed manifolds. The argument works well here by using the isoperimetric lemmas given in \cite{LZ16}*{\S 3.2}.

	\end{proof}

	Before stating the next result, we first recall the notion of homotopic equivalence between discrete sweepouts. Let $Y$ be a cubical subcomplex of $I(m, k)$. Given two discrete maps $\phi_i : Y(l_i)_0\rightarrow\mc C(M)$, we say that {\em $\phi_1$ is homotopic to $\phi_2$ with fineness less than $\eta$}, if there exist $l\in\mb N$, $l> l_1, l_2$ and a map
	\[\psi: I(1, k + l)_0\times Y(l)_0\rightarrow\mc C(M)\]
	with fineness $\mf f(\psi) < \eta$ and such that
	\[\psi([i-1], y) =\phi_i (\mf n(k+l, k + l_i)(y)), i = 1, 2, y \in Y(l)_0.\]

	The purpose of the next theorem is to construct a continuous map in the $\mf F$-topology from a discrete map with small fineness,  which is called an {\em Almgren extension}. Moreover, the Almgren extensions from two homotopic maps are also homotopic to each other.
	\begin{theorem}[Interpolation, \cite{Zhou19}*{Theorem 1.12 and Proposition 1.14}]\label{thm:Almgren extension}
		There exist some positive constants $C_0= C_0(M, m)$ and $\delta_0 = \delta_0(M, m)$ so that if $Y$ is a cubical subcomplex of $I(m, k)$ and
		\[\phi: Y_0\rightarrow\mc C(M)\]
		has $\mf f(\phi) < \delta_0$, then there exists a map
		\[\Phi:Y\rightarrow\mc C(M)\]
		continuous in the $\mf F$-topology and satisfying
		\begin{enumerate}[label=(\roman*)]
			\item  $\Phi(x) = \phi(x)$ for all $x \in Y_0$;
			\item  if $\alpha$ is some $j$-cell in $Y$, then $\Phi$ restricted to $\alpha$ depends only on the values of $\phi$ restricted on the vertices of $\alpha$;
			\item  \[\sup\{\mf F(\Phi(x),\Phi(y)) : x, y \text{ lie in a common cell of } Y \}\leq C_0\mf f (\phi).\]
		\end{enumerate}
		Moreover, if $\phi_i:Y(l_i)_0\rightarrow \mc C(M)$ ($i=1,2$) is homotopic to each other with fineness  $\eta< \delta_0 (M, m)$, 
		\[\Phi_1 , \Phi_2 : Y\rightarrow\mc C(M )\]
		of $\phi_1, \phi_2$, respectively, are homotopic to each other in the $\mf F$-topology.
	\end{theorem}

	Now for a continuous map $\Phi$ in Theorem \ref{thm:discretization}, there exists a sequence of discretized maps $\{\phi_i\}$. Applying Theorem \ref{thm:Almgren extension} to each $\phi_i$, we obtain $\Phi_i$ continuous in the $\mf F$-topology. Then the next proposition says that $\Phi_i$ is homotopic to $\Phi$.
	\begin{proposition}[\cite{Zhou19}*{Proposition 1.15}]
		Let $\{\phi_i\}_{i\in\mb N}$ and $\{\psi_i\}_{i\in\mb N}$ be given by Theorem \ref{thm:discretization} applied to some $\Phi$ therein. Assume that $\Phi$ is continuous in the $\mf F$-topology on $X$. Then the Almgren extension $\Phi_i$ is homotopic to $\Phi$ in the $\mf F$-topology for sufficiently large $i$.
		
		In particular, for $i$ large enough, there exist homotopy maps $\Psi_i : [0,1]\times X\rightarrow\mc C(M)$ continuous in the $\mf F$-topology, $\Psi_i (0, \cdot) =\Phi_i$, $\Psi_i(1,\cdot) =\Phi$, and
		\[
		\limsup_{i\rightarrow\infty}\sup_{t\in[0,1],x\in X}\mf F(\Psi_i(t, x),\Phi(x)) \rightarrow 0.\]
		Therefore, for given $h\in C^\infty (M)$, we have
		\[\limsup_{i\rightarrow\infty}\sup_{x\in X} \mc A^h(\Phi_i(x))\leq \sup_{x\in X}\mc A^h(\Phi(x)).\]
	\end{proposition}

	\subsection{Regularity of $h$-almost minimizing varifolds with free boundary}
	One key ingredient in the Almgren-Pitts theory to prove regularity of min-max varifold is to introduce the ``$h$-almost minimizing with free boundary'' concept. 
	\begin{definition}[$h$-almost minimizing varifolds with free boundary]
		Let $\bm \nu$ be the $\mc F$ or $\mf M$-norm, or the $\mf F$-metric. For any given $\epsilon,\delta> 0$ and a relative open subset $U\subset M$, we define $\ms A^h(U;\epsilon,\delta;\bm\nu)$ to be the set of all $\Omega\in\mc C(M)$ such that if $\Omega= \Omega_0, \Omega_1, \Omega_2,...,\Omega_m \in\mc C(M)$ is a sequence with:
		\begin{enumerate}[label=(\roman*)]
			\item $\spt(\Omega_i-\Omega)\subset U$;
			\item  $\bm\nu(\partial\Omega_{i+1},\partial \Omega_i)\leq\delta$;
			\item $\mc A^h(\Omega_i)\leq \mc A^h(\Omega)+\delta$, for $i=1,...m$, then 
			\[\mc A^h(\Omega_m)\geq \mc A^h(\Omega)-\epsilon.\]  
		\end{enumerate}
		We say that a varifold $V\in\mc V_n(M)$ is {\em $h$-almost minimizing in $U$ with free boundary} if there exist sequences $\epsilon_i\rightarrow 0$, $\delta_i\rightarrow 0$, and $\Omega_i\in\ms A^h(U;\epsilon_i,\delta_i;\mc F)$ such that $\mf F(|\partial \Omega_i|,V)\leq \epsilon_i.$
	\end{definition}

	For each $p\in\partial M$, as defined in \cite{LZ16}*{Definition A.4}, the {\em Fermi half-ball and half-sphere of radius $r$ centered at $p$} are
	\[ \wti {\mc B}^+_r(p)=\{ q\in M: \wti r_p(q)<r \},\ \ \ \ \ \wti {\mc S}^+_r(p)=\{q\in M:\wti r_p(q)=r\} , \]
	where $\wti r_p(\cdot)$ is the {\em Fermi distance function to $p$}; (see \cite{LZ16}*{Definition A.1}). For $p\in M\setminus \partial M$, $\wti{\mc B}^+_r(p)$ and $\wti{\mc S}^+_r(p)$ are also used to denote the geodesic ball and sphere of radius $r$ at $p$.
	
	\begin{definition}\label{def:almost minimizing}
		A varifold $V\in\mc V_n (M)$ is said to be {\em $h$-almost minimizing in small annuli with free boundary} if for each $p\in M$, there exists $r_{am}(p) > 0$ such that $V$ is $h$-almost minimizing with free boundary in $\mc A_{s,r}(p)$ for all $0<s<r\leq r_{am}(p)$, where $\mc A_{s,r}(p)=\wti{\mc B}^+_r(p)\setminus \wti{\mc B}^+_s(p)$.
	\end{definition}

	Before stating the regularity of $h$-almost minimizing varifolds with free boundary, we provide the regularity of $h$-replacements, which follows from Theorem \ref{thm:regularity of h-minimizer}.
	\begin{proposition}[Replacements \cite{ZZ18}*{Proposition 6.8}]\label{prop:replacement}
		Let $V\in\mc V_n(M)$ be $h$-almost minimizing with free boundary in a relative open set $U\subset M$ and $K\subset U$ be a compact subset. Then there exists $V^*\in\mc V_n(M)$, called an $h$-replacement of $V$ in $K$ such that, with $c=\sup |h|$,
		\begin{enumerate}[label=(\roman*)]
			\item \label{item:prop:only change in K} $V\llcorner(M\setminus K)=V^*\llcorner(M\setminus K)$;
			\item \label{item:prop:area change bounded} $-c\cdot\mathrm{Vol}(K)\leq|V|(M)-|V^*|(M)\leq c\cdot\mathrm{Vol}(K)$;
			\item \label{item:prop:almost minimizing} $V^*$ is $h$-almost minimizing in $U$ with free boundary ;
			\item \label{item:prop:locally minimizing} $V^*=\lim_{i\rightarrow\infty}|\partial\Omega^*_i|$ as varifolds for some $\Omega^*_i\in\mc C(M) $ such that 
			\[
			\Omega^*_i\in\ms A(U;\epsilon_i,\delta_i;\mc F) \text{ with $\epsilon_i,\delta_i\rightarrow 0$};
			\]
			furthermore $\Omega^*_i$ locally minimizes $\mc A^h$ in $\mathrm{int}(K)$ (relative to $\partial M$);
			\item\label{item:prop:bounded first variation} if $V$ has $c$-bounded first variation in $M$, then so does $V^*$.
		\end{enumerate}
	\end{proposition}
	\begin{proof}
		The proof here is same as \cite{ZZ18}*{Proposition 6.8}. We only sketch the steps and point out the difference here.
		
		By Definition \ref{def:almost minimizing}, there exist sequences $\epsilon_i\rightarrow 0$, $\delta_i\rightarrow 0$ and $\Omega_i\in\ms A^h(U;\epsilon_i,\delta_i;\mc F)$ such that $\mf F(|\partial \Omega_i|,V)\leq \epsilon_i$. Then for each $\Omega_i$, denote by $\Omega_i^*$ the current by solving a constrained minimization problem; see \cite{ZZ18}*{Lemma 6.7}. Then by Theorem \ref{thm:regularity of h-minimizer}, $\partial \Omega^*_i$ is a properly embedded $h$-hypersurface with free boundary. One can check that $V^*:=\lim_{i\rightarrow\infty}|\partial\Omega_i^*|$ (as varifolds) is the desired replacement; (see \cite{ZZ17}*{Proposition 5.8} for details).
	\end{proof}

	Making use of Theorem \ref{prop:replacement} and the monotonicity formula for varifolds with bounded first variation, we can classify the tangent varifolds.
	\begin{lemma}[\cite{LZ16}*{Proposition 5.10}]\label{lem:classification of tangent cones}
		Let $2\leq n\leq 6$. Suppose that $V\in\mc V_n(M)$ has $c$-bounded first
		variation in $M$ and is $h$-almost minimizing in small annuli with free boundary. For any tangent varifold $C\in \mathrm{VarTan}(V,p)$ with $p\in \spt|V|\cap \partial M$, we have either
		\begin{enumerate}[label=(\roman*)]
			\item \label{item:tangent cone:entire plane} $C=\Theta^n(|V|,p)|T_p(\partial M)|$ where $\Theta^n(|V|,p)\in \mb N$ or
			\item \label{item:tangent cone:half plane} $C=2\Theta^n(|V|,p)|S\cap T_pM|$ for some $n$-plane $S\in \mf G(L,n)$ such that $S\subset T_p\wti M$ and $S\perp T_p(\partial M)$ and $2\Theta^n(|V|,p)\in\mb N$.
		\end{enumerate}
		Moreover, for $|V|$-a.e. $p\in\spt |V|\cap \partial M$, the tangent varifold of $V$ at $p$ is unique, and the set of $p\in\partial M$ in which case (ii) occurs as its unique tangent cones has $|V|$-measure 0; hence $V$ is rectifiable.
	\end{lemma}
	\begin{proof}
		The first step is to prove that $C$ is a stationary rectifiable cone in $T_pM$ with free boundary. Such a result follows from the monotonicity formula, which also holds true for varifolds with bounded first variation (see \cite{GLZ16}) together with the argument in \cite{LZ16}*{Lemma 5.8}.
		
		Then applying Proposition \ref{prop:replacement}, the proof of \cite{LZ16}*{Proposition 5.10} gives the desired classification. The only necessary modification is to replace the condition $\|V\|(M)=\|V^*\|(M)$ by Proposition \ref{prop:replacement} \ref{item:prop:area change bounded} so that we can still get the volume growth for the blow up limit.  
	\end{proof}

	Now we are ready to prove the main regularity theorem for varifolds which is $h$-almost minimizing with free boundary and has $c$-bounded first variation.
	\begin{theorem}[Main regularity]\label{thm:regularity of h minmax}
		Let $2\leq n\leq 6$, and $(M^{n+1},\partial M , g)$ be an $(n +1)$-dimensional smooth, compact Riemannian manifold with boundary. Further let $h\in\mc S(g)$ and set $c=\sup |h|$. Suppose $V\in\mc V_n(M)$ is a varifold which
		\begin{enumerate}
			\item has $c$-bounded first variation in $M$, and
			\item is $h$-almost minimizing in small annuli with free boundary,
		\end{enumerate}
		then $V$ is induced by $\Sigma$, where $\Sigma$ is a compact, almost embedded $h$-hypersurface with free boundary (possibly disconnected).
	\end{theorem}
	
	\begin{proof}
		We only need to prove the regularity of $V$ near an arbitrary point $p\in\spt |V|\cap \partial M$. Fix a $p\in\spt|V|\cap\partial M$, then there exists $0< r_0 < r_{am}(p)$ such that for any $r < r_0$, the mean curvature $H$ of $\partial \wti{\mc B}^+_r (p)\cap M$ in $M$ is greater than $c$. Here $r_{am}(p)$ is as in Definition \ref{def:almost minimizing}. In particular, if $r < r_0$ and $W\in\mc V_n (M)$ has $c$-bounded first variation in $M$ and $W\neq 0$ in $\mc{\wti B}_r^+(p)$, then 
		\begin{equation}\label{eq:main regularity:from outside}
			\emptyset\neq \spt W \cap \mc{\wti S}_r^+(p) = \mathrm{Clos}[\spt |W|\setminus\mathrm{Clos}(\mc{\wti B}_r^+(p))]\cap \mc{\wti S}_r^+(p).
		\end{equation}
		Here $\mathrm{Clos}(\cdot)$ stands for the closure of some set.

		We will show that $V\llcorner \wti{\mc B}^+_{r_0} (p)$ is an almost embedded free boundary $h$-hypersurface with density equal to $2$ along its self-touching set.
		
		The argument consists of six steps:
		
		\medskip	
		{\noindent\bf Step 1}: {\em Constructing successive $h$-replacements $V^*$ and $V^{**}$ on two overlapping concentric annuli.}
		
		\medskip	
		{\noindent\bf Step 2}: {\em Gluing the $h$-replacements smoothly (as immersed hypersurfaces) on the overlap.}
		
		\medskip
		{\noindent\bf Step 3}: {\em Extending the $h$-replacements to the point $p$ to get an `$h$-replacement' $\wti V$ on the punctured ball.}
		
		\medskip	
		{\noindent\bf Step 4}: {\em Showing that the singularity of $\wti V$ at $p$ is removable, so that $\wti V$ is regular.}
		
		\medskip
		{\noindent\bf Step 5}: {\em Showing that $\spt V\cap \wti{\mc B}^+_\alpha(p)$ is not contained in $\partial M$ for all $\alpha>0$.}

		\medskip	
		{\noindent\bf Step 6}: {\em Proving that $V$ coincides with the almost embedded hypersurface $\wti V$  on a small neighborhood of $p$.}
		
		\medskip	
		We now proceed to the proof.
		
		\bigskip
		{\noindent\bf Step 1}. Fix any $0 < s < t < r$. Since $V$ is $h$-almost minimizing on small annuli with free boundary, we can apply \cite{ZZ18}*{Lemma 6.7} by replacing \cite{ZZ18}*{Theorem 2.2} with Theorem \ref{thm:regularity of h-minimizer} to obtain a first replacement $V^*$ of $V$ on $K = \mathrm{Clos}(\mc A_{s,t}(p))$. By Theorem \ref{thm:regularity of h-minimizer} and Proposition \ref{prop:replacement} \ref{item:prop:locally minimizing}, if
		\[\Sigma_1:= \spt |V^*|\llcorner \mc A_{s,t}(p),\]
		then $(\Sigma_1,\partial\Sigma_1 )\subset (M, \partial M )$ is an almost embedded stable free boundary $h$-hypersurface with some unit normal $\nu_1$; when the multiplicity is $1$, $\Sigma_1$ is locally a boundary so we can choose $\nu_1$ to be the outer normal.
		
		Note that all the touching set $\mc S(\Sigma_1)$ is contained in a countable union of $(n-1)$-dimensional connected submanifolds $\bigcup S_1^{(k)}$. Since a countable union of sets of measure zero still has measure zero, by Sard’s theorem we can choose $s_2 \in(s, t)$ such that $\wti{\mc S}_{s_2}^+(p)$ intersects $\Sigma_1$ and all the $S^{(k)}_1$ transversally (even at $\partial\Sigma_1$). Then given any $s_1\in(0,s)$, following the argument in \cite{ZZ17}*{Theorem 6.1, Step 1}, we can construct $V^{**}$, which is a replacement of $V^*$ on $\mathrm{Clos}(\mc A_{s_1,s_2})$. Denote by
		\[\Sigma_2:=\spt|V^{**}|\llcorner\mc A_{s_1,s_2}(p).\]
		
		\bigskip
		{\noindent\bf Step 2}.  We now show that $\Sigma_1$ and $\Sigma_2$ glue smoothly (as immersed hypersurfaces) across $\wti{\mc S}_{s_2}^+(p)$. Indeed, define the intersection set
		\begin{equation}
			\Gamma=\mathrm{Clos}(\Sigma_1)\cap \wti{\mc S}_{s_2}^+(p).
		\end{equation}
		Then by transversality, $\Gamma$ is an almost embedded hypersurface in $\wti{\mc S}_{s_2}^+(p)$. Particularly, $\Gamma$ is not contained in $\partial M$ since $\wti{\mc S}_{s_2}^+$ intersect $\partial\Sigma_1$ and $\mc S(\Sigma_1)$ transversally. For $x\in\Gamma\setminus \partial M$, following from the interior argument \cite{ZZ18}*{Theorem 7.1, Step 3}, $\Sigma_1$ coincides with $\Sigma_2$ (with matching normal) in a neighborhood of $x$. Using the unique continuation of $h$-hypersurfaces, we conclude that
		\[\Sigma_1\llcorner\mc A_{s,s_2}(p)=\Sigma_2\llcorner\mc A_{s,s_2}(p).\]
		Then we finish the proof of Step 2. 
		
		\bigskip
		{\noindent\bf Step 3}. Then we extend the replacements, via the unique continuation from Step 2, all the way to $p$. In fact, we use $V_{s_1}^{**}$ to denote the second replacement that we constructed in Step 1 with inner radius $s_1$. Step 2 shows that this construction does not depend on $s_2$. Then we define $\wti V$ to be the limit of $V_{s_1}^{**}$ as $s_1 \to 0$.
		
		See \cite{ZZ17}*{Theorem 6.1, Step 3} for details. 
		
		\bigskip
		{\noindent\bf Step 4}. We now determine the regularity of $\wti V$ at $p$.
		
		Firstly, observing that $\wti V$ is still $h$-almost minimizing in small annuli with free boundary and that $\wti V$ is the varifold limit of a sequence $V_{s_1}^{**}$, which all have $c$-bounded first variation, we know that $\wti V$ also has $c$-bounded first variation. This implies that the classification of tangent cones in Lemma \ref{lem:classification of tangent cones} also holds true. Secondly, $\wti V$, when restricted to any small annulus $\mc A_{\alpha,\beta}(p)$ ($0<\alpha<\beta<s$), already coincides with a smooth, almost embedded, weakly stable $h$-boundary $\Sigma$ with free boundary. Using these two ingredients, by Theorem \ref{thm:removable touching set}, $\Sigma$ extends smoothly across $p$ as an almost embedded hypersurface in $\wti{\mc B}^+_s(p)$. Thus we complete Step 4.

		\bigskip
		{\noindent\bf Step 5}. We argue by contradiction and assume that $\spt V\cap \wti{\mc B}^+_{\alpha_0}(p)\subset \partial M$ for some $\alpha_0<s$. Here we also use the chosen constants $s$, $t$, $\alpha$, $s_2$ in the previous steps. We first recall that by the Constancy Theorem \cite{Si}*{Theorem 41.1}, 
		\[ \spt V\cap \wti{\mc B}^+_{\alpha_0}(p)= \partial M\cap \wti {\mc B}^+_{\alpha_0}(p).   \] 
		Now we take the first replacement $V^*$ of $V$ on $K=\mathrm{Clos}(\mc A_{s,t}(p))$. In the next paragraph, we are going to prove that for any $\alpha<\alpha_0$, $ \partial \wti{\mc B}^+_{\alpha}(p)\cap\partial M\subset\spt \wti V $. As a result, for any $\alpha<\alpha_0$, $\wti{\mc B}^+_{\alpha}(p)\cap \partial M\subset \spt \wti V$, which leads to a contradiction to Proposition \ref{prop:estimate of touching set}.

		To conclude this step, we consider the second replacement $V^{**}_{\alpha}$ of $V^*$ on $\mathrm{Clos}(\mc A_{\alpha,s_2}(p))$. By the assumption $\wti{\mc B}^+_\alpha(p)\cap \partial M\subset \spt V$ and $V\llcorner \wti {\mc B}^+_\alpha(p)= {V^{**}_{\alpha}}\llcorner \wti{\mc B}^+_\alpha(p)$, we have $\partial \wti{\mc B}^+_\alpha(p)\cap \partial M\subset\spt V^{**}_{\alpha}$. On the other hand, $V^{**}_{\alpha}\llcorner \mc A_{\alpha,s_2}(p)=\wti V\llcorner \mc A_{\alpha,s_2}(p)$. Together with the classification of tangent cones in Lemma \ref{lem:classification of tangent cones}, we conclude that $\partial \wti {\mc B}^+_\alpha(p)\cap \partial M\subset \spt\wti V$. This concludes Step 5.  
		
		\bigskip
		{\noindent\bf Step 6}. It remains to show that $V$ coincides with $\wti V$ in $\wti {\mc B}^+_{s}(p)$. Recall that by \cite{ZZ18}*{Theorem 7.1}, $V\llcorner\mathrm{Int}(M)$ is an almost embedded $h$-hypersurface. Denote by $\mc S(V)$ the self-touching set. Then by Proposition \ref{prop:estimate of touching set}, $\mc S(V)$ is contained in a countable union of smoothly embedded $(n-1)$-dimensional submanifolds. Hence we can take $s'<s$ so that $\spt V\cap \mathrm{Int} M\cap \partial \wti{\mc B}^+_{s'}(p)$ is not contained in $\mc S(V)$. Recall that $V^{**}_{s'}$ is the second replacement of $V^*$ in $\mc A_{s',s_2}$. Take 
		\[z\in \spt V\cap \mathrm{Int} M\cap \partial \wti{\mc B}^+_{s'}(p)\setminus \mc S(V).\]
		Such a set is non-empty by Step 5. Then $V$ coincides with $V^{**}_{s'}$ in a small neighborhood of $z$ by the construction of the second replacement. Then the unique continuation principle gives that $V=\wti V$ in $\wti {\mc B}^+_{s}(p)$. This completes the proof of Theorem \ref{thm:regularity of h minmax}.
	\end{proof}

	\subsection{Relative Min-max theory for free boundary $h$-hypersurfaces}

	The existence of almost minimizing varifolds follows from a combinatorial argument of Pitts \cite{Pi}*{page 165-page 174} inspired by early work of Almgren \cite{Alm65}. Pitts's argument works well in the construction of min-max $h$-hypersurfaces; see \cite{ZZ18}*{Theorem 6.4}. Marques-Neves has generalized Pitts's combinatorial argument to a more general form in \cite{MN17}*{\S 2.12}, and we can adapt their result to the free boundary $h$-hypersurface setting with no change.

	Recall that a minimizing sequence $\{\Phi_i\}_{i\in\mb N}\in\Pi$ such that every element of $\mf C(\{\Phi_i\})$ (see Definition \ref{def:critical set}) has $c$-bounded variation or belongs to $|\partial \Phi_0|(Z)$ is called a {\em pull-tight}.

	The purpose of this part is to establish min-max theory for free boundary $h$-hypersurfaces. Recall that the Morse index of an almost embedded free boundary $h$-hypersurface $\Sigma$ is given in Definition \ref{def:Morse index}.
	\begin{theorem}\label{thm:index bound for all g}
		Let $(M^{n+1},\partial M,g)$ be a compact Riemannian manifold of dimension $3\leq (n+1)\leq 7$, and $h\in \mc S(g)$ which satisfies $\int_Mh\geq 0$. Given a $k$-dimensional cubical complex $X$ and a subcomplex $Z\subset X$, let $\Phi_0:X\rightarrow\C(M)$ be a map continuous in the $\mf F$-topology, and $\Pi$ be the associated $(X,Z)$-homotopy class of $\Phi_0$. Suppose 
		\begin{equation}\label{eq:relative constraint}
			\mf L^h(\Pi)>\max\big\{\max_{x\in Z}\A^h(\Phi_0(x)),0\big\}.
		\end{equation}
		Then there exists a nontrivial, smooth, compact, almost embedded hypersurface with free boundary $(\Sigma^n,\partial\Sigma)\subset (M,\partial M)$ , such that
		\begin{itemize}
			\item $\llbracket\Sigma\rrbracket=\partial \Omega$ for some $\Omega\in\C(M)$, where the mean curvature of $\Sigma$ with respect to the unit outer normal of $\Omega$ is $h$, i.e.
			\[H|_\Sigma = h|_\Sigma ;\]
			\item $\A^h(\Omega) =\mf L^h(\Pi)$;
			\item  $\mathrm{index}_w(\Sigma) \leq  k$.
		\end{itemize}
	\end{theorem}
	
	\begin{proof}[Proof of Theorem \ref{thm:index bound for all g}]
		The proof can be divided into five steps. In the first four steps, we always assume that $(M,\partial M,g)$ is isometrically embedded into a closed manifold $(\wti M^{n+1},\wti g)$ and $(\wti g,\wti h)$ is a good pair (e.g. Section \ref{subsec:good pairs}) related to $M$ so that $\wti h|_M=h$. 
		
		\medskip
		{\noindent\bf Step A:} We construct a pulled-tight minimizing sequence $\{\Phi_i\}_{i\in \mb N}\in\Pi$ so that every element of $\mf C(\{\Phi_i\})$ either has $c$-bounded first variation, or belongs to $|\partial \Phi_0|(Z)$.
		
		\medskip
		{\noindent\bf Step B:} There exists $V\in\mf C(\{\Phi_i\})$ so that $V$ is $h$-almost minimizing in small annuli with free boundary.
		
		\medskip
		{\noindent\bf Step C:} $V$ has $c$-bounded first variation, and hence $V$ is supported on an almost embedded free boundary $h$-hypersurface $\Sigma$ satisfying $\llbracket\Sigma\rrbracket=\partial \Omega$ and $\mc A^h(\Omega)=\mf L^h(\Pi)$.
		
		\medskip
		{\noindent\bf Step D:} The $\{\Phi_i\}$ and $V$ in Step B can be chosen so that the support of $V$ has weak Morse index less than or equal to $k$. 
		
		\medskip
		{\noindent\bf Step E:} We provide the proof for general $h\in\mc S(g)$.
		
		\begin{proof}[Proof of Step A]
			Let $c=\sup_M|h|$ and $L^c=2\mf L^h+c\Vol (M)$. Set
			\[A^c_{\infty}=\{V\in \mc V _n(M):\Vert V\Vert (M)\leq L^c, V \text{ has $c$-bounded first variation}\}\cup|\partial \Phi_0|(Z) , \]
			where $|\partial \Phi_0|(Z)=\{|\partial \Phi_0(x)|:x\in Z\}$. Then we can follow \cite{LZ16}*{Proposition 4.17} to construct a continuous map:
			\[H:[0,1]\times (\C(M),\F)\cap\{\M(\partial \Omega)\leq L^c\}\rightarrow(\C(M),\F)\cap\{\M(\partial \Omega)\leq L^c\}\]
			such that:
			\begin{enumerate}[label=(\roman*)]
				\item  $H(0,\Omega)=\Omega$ for all $\Omega\in \C(M)$;
				\item \label{pull tight:fix Z} $H(t,\Omega)=\Omega$ if $|\partial \Omega|\in A^c_\infty$;
				\item  if $|\partial \Omega|\notin A^c_\infty$,
				\[\A^h(H(1,\Omega))-\A^h(\Omega)\leq -L(\F(|\partial\Omega|,A^c_\infty))<0,\]
				where $L:[0,+\infty)\rightarrow[0,+\infty)$ is a continuous function with $L(0)=0$, $L(t)>0$ when $t>0$;
				\item  for every $\epsilon >0$, there exists $\delta>0$ such that 
				\[x\in Z, \F(\Omega,\Phi_0(x))<\delta\Rightarrow\F(H(t,\Omega),\Phi_0(x))<\epsilon, \text{ for all } t\in [0,1];\]
				this is a direct consequence of \ref{pull tight:fix Z} since $|\partial \Phi_0|(Z)\subset A^c_\infty$.
			\end{enumerate}
			
			Given a minimizing sequence $\{\Phi_i^*\}_{i\in\mb N}\in \Pi$, we define $\Phi_i(x)=H(1,\Phi^*_i(x))$ for every $x\in X$. Then $\{\Phi_i\}_{i\in\mb N}$ is also a minimizing sequence in $\Pi$. Moreover, $\mf C(\{\Phi_i\})\subset\mf C(\{\Phi_i^*\})$ and every element of $\mf C(\{\Phi_i\})$ either has $c$-bounded first variation, or belongs to $|\partial \Phi_0|(Z)$. We refer to \cite{Zhou19}*{Lemma 1.8} for the details of verification. This finishes proving Step A.
		\end{proof}

		\begin{proof}[Proof of Step B]
			The proof here is parallel to \cite{Zhou19}*{Theorem 1.7} and we just sketch the idea for completeness.
			
			Let $\{\Phi_i\}_{i\in \mb N}\in \Pi$ be a pulled-tight minimizing sequence. For each $\Phi_i$, Theorem \ref{thm:discretization} gives a sequence of maps:
			\[\phi_i^j:X(k_i^j)_0\rightarrow \mc C(M)\]
			with $k_i^j<k_i^{j+1}$ and a sequence of positive $\delta_i^j\rightarrow 0$ (as $j\rightarrow \infty$), satisfying Theorem \ref{thm:discretization}. Then for each $i$, take sufficiently large $j(i)$ and let $\varphi_i=\phi_i^{j(i)}$. Denote by $S=\{\varphi_i\}$. Then we have $\mf L^h(S)=\mf L^h(\{\Phi_i\})$ and $\mf C(S)=\mf C(\{\Phi_i\})$, where
			\begin{gather*}
				\mf L^h(S) = \limsup_{i\rightarrow\infty}\sup\{\mc A^h(\varphi_i(y)): y\in(X_i^{j(i)})_0\};\\
				\mf C(S)=\{V = \lim_{j\rightarrow\infty}|\partial \varphi_{i_j} (y_j)| \text{ as varifolds: with } \lim_{j\rightarrow\infty}\mc A^h(\varphi_{i_j}(y_j))=\mf L^h (S)\}.
			\end{gather*}  
			We now prove that there exists $V\in \mf C(S)$ so that it is $h$-almost minimizing in small annuli with free boundary. For a further reason, we need a stronger result:
			\begin{claim}\label{claim:almost minimizng with free boundary}
				There exist a varifold $V$ satisfying the following: for any $p\in M$ and any small enough annulus $\mc A_{r_1,r_2}$ centered at $p$ with radii $0<r_1<r_2$, there exist two sequences of positive real numbers $\epsilon_j\rightarrow 0, \delta_j\rightarrow 0$, a subsequence $\{i_j\}\subset\{i\}$ and $y_j\in\mathrm{dmn} \varphi_{i_j}$ (the domain of $\varphi_{i_j}$) so that 
				\begin{itemize}
					\item $\lim_{j\rightarrow\infty}\mc A^h(\varphi_{i_j}(y_j))=\mf L^h (S)$;
					\item $\varphi_{i_j}(y_j)\in \ms A^h(\mc A_{r_1,r_2};\epsilon_j,\delta_j;\mf M)$; and
					\item $\lim_{j\rightarrow\infty}|\partial \varphi_{i_j} (y_j)|=V$.
				\end{itemize}
			\end{claim}
			If Claim \ref{claim:almost minimizng with free boundary} were not true, then using the argument in \cite{Zhou19}*{Theorem 1.16}, we can find a sequence $\wti S=\{\wti\varphi_i\}$ so that $\wti\varphi_i$ is homotopic to $\varphi_i$ with fineness converging to zero as $i\rightarrow \infty$ and $\mf L^h(\wti S)<\mf L^h(S)$. The key point here is that $\mf C(S)$ is compact in the sense of varifolds.
			
			Then by Theorem \ref{thm:Almgren extension}, the Almgren extensions of $\varphi_i,\wti \varphi_i$:
			\[  \Phi_i^{j(i)},\wti \Phi_i:X\rightarrow \mc C(M),\]
			respectively, are homotopic to each other in the $\mf F$-topology for large $i$ and 
			\[  \limsup_{i\rightarrow\infty}\sup\{\mc A^h(\wti\Phi_i(x)):x\in X\}\leq \mf L^h(\wti S)<\mf L^h(S).\]
			This leads to a contradiction and we have finished the proof of Claim \ref{claim:almost minimizng with free boundary}. Obviously, such a varifold $V$ is almost minimizing with free boundary in small annuli by Definition \ref{def:almost minimizing}. Therefore, Step B is also completed.
		\end{proof}
		
		\begin{proof}[Proof of Step C]
			We first prove that $V$ has $c$-bounded first variation. Indeed, from Step A, $V$ either has $c$-bounded first variation or belongs to $|\partial \Phi_0|(Z)$. Recall that being $h$-almost minimizing in small annuli with free boundary always implies $V$ has $c$-bounded first variation away from finitely many points. Then by a cut-off trick, we only need to prove that $\|V\|$ has at most $r^{n-\frac{1}{2}}$-volume growth near these bad points . This essentially follows from \cite{HL75}*{Theorem 4.1} and we provide more details here.
			
			Let $p$ be a bad point. Then we can take $\epsilon>0$ small enough so that $\wti{\mc B}^+_\epsilon(p)\setminus\{p\}$ has no bad point. For any $0<t<\epsilon<u<1$, let $\eta_u,\phi_t:\mb R\rightarrow \mb R$ be two cut-off functions so that 
			\begin{gather*}
				\eta_u'\leq 0; \  \ \eta_u(x)=1 \text{ for } x\leq u; \  \ \eta_u(x)=0 \text{ for } x\geq 1;\\
				\phi_t'\geq 0;\ \  \phi_t(x)=0 \text{ for } x\leq t/2;\ \  \phi_t(x)=1 \text{ for } x\geq t.
			\end{gather*} 
			We only need to consider $p\in \partial M$. Denote by $\wti r$ the \text{Fermi distance function} to $p$ in \cite{LZ16}*{Appendix A}. Then  there exist $C$ and $\epsilon$ so that for $x\in M$ with $\wti r(x)\leq \epsilon$,
			\begin{equation}\label{eq:almost distance}
				\big|{|\nabla \wti r (x)|-1}\big | \leq C\epsilon\ \  \text { and } \ \ |\nabla^2 \wti r^2/2(x)-g|\leq C\epsilon. 
			\end{equation}
			Being $c$-bounded first variation in $\{0<\wti r<\epsilon\}$ gives that for any $\rho<\epsilon$, 
			\begin{align}
				\label{eq:from c bounded 1st variation} \int \dv_S\big(\eta_u(\wti r/\rho)\phi_t(\wti r)\nabla \wti r^2/2 \big)\,dV(x,S)
				& \leq c\int \eta_u(\wti r/\rho)\phi_t(\wti r) |\nabla \wti r^2/2| \, dV(x,S)\\
				& \leq (1+C\epsilon)c \int \eta_u(\wti r/\rho) \cdot \wti r \, dV .\nonumber
			\end{align}
			By direct computation of the left hand side,
			\begin{align}\label{eq:from direct computation}
				&\ \int \dv_S\big(\eta_u(\wti r/\rho)\phi_t(\wti r)\nabla \wti r^2/2\big)\,dV(x,S)\\
				\geq &\ \int \eta_u'(\wti r/\rho)\cdot \rho^{-1}\cdot\phi_t(\wti r) \wti r \langle p_S(\nabla \wti r), \nabla \wti r\rangle  +\eta_u(\wti r/\rho)\phi_t(\wti r)\dv_S (\nabla\wti r^2/2) \, dV(x,S)\nonumber\\
				\geq &\ \int \eta_u'(\wti r/\rho)\cdot \rho^{-1}\cdot\phi_t(\wti r) \wti r\cdot (1+C\epsilon) +\eta_u(\wti r/\rho)\phi_t(\wti r)\cdot (1-C\epsilon)n\, dV.\nonumber
			\end{align}
			Here $p_S(\cdot)$ is the projection to the hyperplane $S$ and\eqref{eq:almost distance} is used in the last inequality. Note that $V$ either has $c$-bounded first variation or belongs to $|\partial \Phi_0|(Z)$. Hence $\|V\|(\{p\})=0$. Combining \eqref{eq:from c bounded 1st variation} with \eqref{eq:from direct computation}, letting $t\rightarrow 0$, 
			and by shrinking $\epsilon$ if necessary, we have 
			\[(n-1/2)I(\rho)\leq -\int \eta_u'(\wti r/\rho) \wti r\rho^{-1}=\rho\frac{d}{d\rho}\int \eta_u(\wti r/\rho)\, dV=\rho I'(\rho),\]
			where $I(\rho)=\int \eta_u(\wti r/\rho)\, dV$. This implies the monotonicity of $I(\rho)\cdot \rho^{-n+1/2}$. Hence for $\rho<\epsilon$, we have
			\[ I(\rho)\leq I(\epsilon)\epsilon^{-n+1/2}\rho^{n-1/2} .\]
			Then we conclude that 
			\[ \|V\|(\{x\in M:\dist_M(x,p)<\rho/2\})\leq I(\rho) \leq C\rho^{n-1/2}.  \]
			This proves that $V$ has $c$-bounded first variation. Then by Theorem \ref{thm:regularity of h minmax}, $V$ is supported on an almost embedded $h$-hypersurface $\Sigma$ with free boundary.
			
			\medskip
			We now prove that $\Sigma$ is a free boundary $h$-hypersurface satisfying $\llbracket\Sigma\rrbracket=\partial \Omega$ and $\mc A^h(\Omega)=\mf L^h(\Pi)$ for some $\Omega\in \mc C(M)$. Recall that Claim \ref{claim:almost minimizng with free boundary} gives that $V=\lim_{j\rightarrow\infty}|\partial \varphi_{i_j} (y_j)|$. Denote by $\Omega_j=\varphi_{i_j} (y_j)$. Then it suffices to prove that $\partial \Omega_j$ subsequently converges to $\Sigma$ in the $\mc F$ metric. Let $\Omega$ be a limit of $\Omega_j$ in the $\mc F$ topology. Then $\spt(\partial \Omega)\subset \Sigma$. Now for any $p\in \mathrm{Int}(\Sigma)\setminus \mathrm{Touch}(\Sigma)$ and $r_p>0$ small enough, the Constancy Theorem \cite[Theorem 26.27]{Si} implies that $\partial \Omega\llcorner B_{r_p}(p)=\llbracket \Sigma\rrbracket\llcorner B_{r_p}(p)$ or $0$. Here $\mathrm{Touch}(\Sigma)$ is the touching set inside $\Sigma$.
			
			In the next, we prove that $\partial \Omega\llcorner B_{r_p}(p)=\llbracket\Sigma\rrbracket \llcorner B_{r_p}(p)$ for $p\in\mathrm{Int}(\Sigma)\setminus \mathrm{Touch}(\Sigma)$.
			Recall that first $h$-replacement $V^*$ in $\mc A_{r_1,r_2}(p)$ coincides with $V$ in the Step 5 of Theorem \ref{thm:regularity of h minmax}. On the other hand, the argument in \cite{ZZ18}*{Proposition 7.3} gives that $V^*\llcorner \mc A_{r_1,r_2}(p)=|\partial \Omega|\llcorner \mc A_{r_1,r_2}(p)$. Hence $\partial \Omega\llcorner B_{r_p}(p)=\llbracket\Sigma\rrbracket \llcorner B_{r_p}(p)$.
			
			Recall that by Proposition \ref{prop:estimate of touching set}, the self-touching set and $\Sigma\cap \partial M$ are contained in countable $(n-1)$-dimensional submanifolds. Therefore, we conclude that $V=|\partial \Omega|$. Hence Step C is finished. 
		\end{proof}	
		
		\begin{proof}[Proof of Step D]
			Recall that by Lemme \ref{lem:countable}, the set of almost embedded free boundary $h$-hypersurface is countable. Then the proof here is parallel to \cite{Zhou19}*{Theorem 3.6}, which is in fact a generalization of \cite{MN16}.	
			
			Denote by $\mc U$ the set of all $(\Sigma,\Omega)\in\mc P^h\times\C(M)$ so that $\llbracket\Sigma\rrbracket=\partial \Omega$. Set
			\begin{gather*}
				\mc W=\{\Omega\in \mc C(M):(\Sigma,\Omega)\in\mc U, \A^h(\Omega)=\mf L^h \},\\
				\mc W(r)=\{\Omega\in \mc W: (\Sigma,\Omega)\in\mc U, \mf F([\Sigma]
				,\mf C(\{\Phi_i\}_{i\in \mb N}))\geq r\},\\
				\mc W^{k+1}=\{\Omega\in \mc W:(\Sigma,\Omega)\in\mc U, \mathrm{index}_w(\Sigma)\geq k+1\}.
			\end{gather*}
			
			It suffices to show that for every $r>0$, $\mc W\setminus(\mc W^{k+1}\cup\mc W(r))$ is not empty.
			
			By the same proof with \cite{Zhou19}*{Lemma 3.7}, there exist $i_0\in\mb N$ and $\bar\epsilon_0>0$ such that $\mf F(\Phi_i(X),\mc W(r))>\bar\epsilon_0$ for all $i>i_0$.
			
			As $(\wti g,\wti h)$ is a good pair related to $M$, $\mc W^{k+1}$ is countable by Lemma \ref{lem:countable}, and we can assume that
			\[\mc W^{k+1}\setminus\B_{\bar{\epsilon}_0}^\F(\mc W(r))=\{\Omega_1,\Omega_2,\cdots\},\]
			and $\partial \Omega_j=\llbracket\Sigma_j\rrbracket$ for $ j\geq 1$. Then by taking $\epsilon_1>0$ small enough, we can make sure $\B_{\epsilon_1}^\mf F(\Omega_1)\cap\B_{\bar{\epsilon}_0}^\F(\mc W(r))=\emptyset$. Using the Deformation Theorem (\citelist{\cite{Zhou19}*{Theorem 3.4}\cite{GLWZ19}*{Theorem 5.8}}), by shrinking $\epsilon_1>0$ if necessary, we can find $i_1\in \mb{N}$, and $\{\Phi_i^1\}_{i\in \mb{N}}$ so that 
			\begin{itemize}
				\item $\Phi_i^1$ is homotopic to $\Phi_i$ in the $\mf{F}$-topology for all $i\in \mb{N}$;
				\item $\mf{L}^h(\{\Phi_i^1\}_i)\leq \mf L^h(\Pi)$;
				\item $\mf{F}(\Phi_i^1(X), \overline{\mf{B}}_{\epsilon_1}^{\mf{F}}(\Omega_1)\cup \overline{\mf{B}}_{\bar\epsilon_0}^{\mf{F}}(\mc{W}(r)))>0$ for all $i\geq i_1$;
				\item  no $\Omega_j$ belongs to $\partial \overline{\mf{B}}_{\epsilon_1}^{\mf{F}}(\Omega_1)$; (this can be easily satisfied since $\{\Omega_1,\Omega_2,\cdots\}$ is a countable set.)
			\end{itemize}
			
			Inductively, there are two possibilities. The first case is that we can find for all $l\in \mb{N}$, there exist a sequence $\{\Phi_i^l\}_{i\in \mb{N}}$, $\epsilon_l$, $i_l\in \mb{N}$, and $\Omega_{j_l}\in \mc{W}^{k+1}\setminus  \overline{\mf{B}}_{\epsilon_0}^{\mf{F}}(\mc{W}(r))$ for some $j_l\in \mb{N}$ so that 
			\begin{itemize}
				\item $\Phi_i^l$ is homotopic to $\Phi_i$ in the $\mf{F}$-topology for all $i\in \mb{N}$;
				\item $\mf{L}^h(\{\Phi_i^l\}_i)\leq\mf L^h(\Pi)$;
				\item $\mf{F}(\Phi_i^l(X), \cup_{q=1}^l \overline{\mf{B}}_{\epsilon_q}^{\mf{F}}(\Omega_{j_q})\cup  \overline{\mf{B}}_{\bar\epsilon_0}^{\mf{F}}(\mc{W}(r)))>0$ for all $i\geq i_l$;  
				\item $\{\Omega_1,\cdots,\Omega_l\}\subset \cup_{q=1}^l \overline{\mf{B}}_{\epsilon_q}^{\mf{F}}(\Omega_{j_q}) $;
				\item  no $\Omega_j$ belongs to $\partial \overline{\mf{B}}_{\epsilon_1}^{\mf{F}}(\Omega_1)\cup \cdots\cup  \partial \overline{\mf{B}}_{\epsilon_l}^{\mf{F}}(\Omega_{j_l})$. 
			\end{itemize}
			The second case is that the process stops in finitely many steps. This means that we can find some $m\in \mb{N}$, a sequence $\{\Phi_i^m\}_{i\in \mb{N}}$, $\epsilon_1,\cdots,\epsilon_m>0$, $i_m\in \mb{N}$, and $\Omega_{j_1},\cdots,\Omega_{j_m}\in \mc{W}^{k+1}\setminus  \overline{\mf{B}}_{\bar\epsilon_0}^{\mf{F}}(\mc{W}(r))$ so that 
			\begin{itemize}
				\item $\Phi_i^m$ is homotopic to $\Phi_i$ in the $\mf{F}$-topology for all $i\in \mb{N}$;
				\item $\mf{L}^h(\{\Phi_i^m\}_i)\leq \mf L^h(\Pi)$;
				\item $\mf{F}(\Phi_i^m(X), \cup_{q=1}^m \overline{\mf{B}}_{\epsilon_q}^{\mf{F}}(\Omega_{j_q})\cup  \overline{\mf{B}}_{\bar\epsilon_0}^{\mf{F}}(\mc{W}(r)))>0$ for all $i\geq i_m$;  
				\item $\{\Omega_j:j\geq 1\}\subset \cup_{q=1}^m \overline{\mf{B}}_{\epsilon_q}^{\mf{F}}(\Omega_{j_q}) $.
			\end{itemize}
			
			For the first case, we can choose a diagonal sequence $\{\Phi_{p_l}^l\}_{l\in \mb{N}}$ and set $\Psi_l=\Phi_{p_l}^l$, where $\{p_l\}_{l\in \mb{N}}$ is an increasing sequence such that $p_l\geq i_l$  and \[\sup_{x\in X}\mc A^h(\Phi_{p_l}^l (x))\leq \mf L^h(\Pi)+\frac{1}{l}.\]
			For the second case, we simply choose the last sequence $\{\Phi_l^m\}_l$  and set $p_l=l$ and $\Psi_l=\Phi_l^m$. Now it is easy to see that for both cases, the new sequence $\{\Psi_l\}_{l\in \mb{N}}$ satisfies the following conditions:
			\begin{enumerate}
				\item $\Psi_l$ is homotopic to $\Phi_{p_l}$ in the $\mf{F}$-topology for all $l\in \mb{N}$;
				\item $\mf{L}^h(\{\Psi_l\}_l)\leq\mf L^h(\Pi)$;
				\item given any subsequence $\{l_j\}\subset\{l\}$, $x_j\in X$, if $\lim_{j\rightarrow\infty}\mc A^h(\Psi_{l_j}(x_j))=\mf L^h(\Pi)$, then $\{\Psi_{l_j}(x_j)\}_{j\in\mb N}$ does not converge in the $\mf F$-topology to any element in  $\mc{W}^{k+1}\cup \mc{W}(r)$.
			\end{enumerate}
			Then by Steps B and C, there exists $V\in \mf C(\{\Psi_l\})$ so that its support has weak Morse index less than or equal to $k$. This finishes the proof of Step D.
		\end{proof}	
		
		\begin{proof}[Proof of Step E]
			Assume that $(M^{n+1},\partial M,g)$ is isometrically embedded into $(\wti M^{n+1},\wti g)$. Recall that $\wti {\mc S}_0$ is the set of smooth Morse functions so that the zero set is a properly embedded closed hypersurface in $\wti M$, and is transverse to $\partial M$. 
			
			Given $h\in \mc S(g)$, we can take an extension $\wti h$ of $h$ so that $\wti h|_M=h$ and $\wti h\in \wti{\mc S}_0$. By Lemma \ref{lem:bummpy metric}, there exists a sequence of smooth metrics $\wti g_j$ on $\wti M$ so that $\wti g_j|_M\rightarrow g$ smoothly and $(\wti g_j,\wti h)$ is a good pair related $M$ for each $j$. Then by Steps A-D, for each $j$, there exists an almost embedded free boundary $h$-hypersurface $\wti \Sigma_j$ with $\partial{\wti \Omega_j}=\llbracket\wti \Sigma_j\rrbracket$ for some $\wti \Omega_j\in\C(M)$ and $\wti \Sigma_j$ has weak Morse index less than or equal to $k$ (with respect to $\wti g_j|_M$). Recall that $h\in\mc S(g)$. Let $\Sigma_\infty$ (with $\partial\Omega_\infty=\llbracket\Sigma_\infty\rrbracket$) be the limit of $\{\wti \Sigma_j\}_{j\in \mb N}$ given in Theorem \ref{thm:compactness with changing metrics}, then the multiplicity one (see Theorem \ref{thm:compactness for FPMC} \ref{compactness thm:generic multiplicity one convergence}) and locally smoothly convergence away from a finite set imply that $\A^h(\Omega_\infty) = \mf L^h(\Pi, g)$ and $\mathrm{index}_w(\Sigma_\infty)\leq k$.
		\end{proof}
		
		By putting all above together, Theorem \ref{thm:index bound for all g} is finished.
	\end{proof}

	\section{Multiplicity one for free boundary minimal hypersurfaces}\label{sec:multi 1 for fbmh}
	\subsection{Multiplicity one for relative sweepouts}\label{subsec:multi one for relative sweepout}
	In this subsection, we approximate the area functional by the weighted functionals $\A^{\epsilon h}$ for some prescribing function when $\epsilon\rightarrow 0$. By the index estimates and the multiplicity one result for $\A^{\epsilon h}$, we prove that the limit free boundary minimal hypersurfaces also have multiplicity one.
	
	Recall that a Riemannian metric $g$ is said to be {\em bumpy} if every finite cover of a smooth immersed free boundary minimal hypersurface is non-degenerate. A bumpy metric $g$ is said to be {\em strongly bumpy} if the touching set of every immersed free boundary minimal hypersurface is empty. By Theorem \ref{thm:intro:generic strongly bumpy}, the set of strongly bumpy metrics is generic in the Baire sense.
	
	\begin{theorem}\label{thm:multi one for relative sweepout}
		Let $(M^{n+1},\partial M,g)$ be a compact Riemannian manifold with boundary of dimension $3 \leq (n + 1)\leq 7$. Let X be a $k$-dimensional cubical subcomplex of $I(m,j)$ and $Z\subset X$ be a subcomplex, and $\Phi_0:X\rightarrow\C(M)$ be a map continuous in the $\mf F$-topology. Let $\Pi$ be the associated $(X,Z)$-homotopy class of $\Phi_0$. Assume that
		\[ \mf L(\Pi) > \max\big(\max_{x\in Z}\mf M(\partial\Phi_0(x)),0\big),\]
		where we let $h\equiv 0$ in Section \ref{subsec:relative min-max}.
		
		If $g$ is a strongly bumpy metric, then there exists a disjoint collection of smooth, connected, compact, two-sided, properly embedded, free boundary minimal hypersurfaces $\Sigma=\cup_{i=1}^N\Sigma_i$ so that
		\[\mf L(\Pi)=\sum_{i=1}^N\Area(\Sigma_i), \text{ and } \mathrm{index}(\Sigma)=\sum_{i=1}^N\mathrm{index}(\Sigma_i)\leq k.\]
		In particular, each component of $\Sigma$ is two-sided and has exactly multiplicity one.
	\end{theorem}
	
	\begin{proof}
		Fix a sequence of positive numbers $\epsilon_k\rightarrow 0$ as $k\rightarrow\infty$. Recall that $\mc S(g)$ is open and dense in $C^\infty(M)$. Thus
		\[ \bigcap_{j=1}^{\infty} \{f:\epsilon_jf\in\mc S(g)\} \]
		is generic in the Baire sense. Pick an $h$ in this generic set with $\int_Mh\geq 0$ (to be fixed at the end, in Part 5), and $\epsilon>0$ small enough so that 
		\[\mf L(\Pi)-\max_{x\in Z}\mf M (\partial \Phi_0(x))>2\epsilon\sup_M|h|\cdot \Vol(M).\]
		Note that $\epsilon_j\cdot h\in\mc S(g)$ for each $j$.
		
		Then we can follow the argument in \cite{Zhou19}*{Theorem 4.1}, by replacing \cite{Zhou19}*{3.1} with Theorem \ref{thm:index bound for all g}, to produce a non-trivial, smooth, compact, almost embedded, free boundary $\epsilon_j\cdot h$-hypersurface $\Sigma_{\epsilon_j}=\partial\Omega_{\epsilon_j}$; moreover, $\A^h(\Omega_{\epsilon_j})=\mf L^{\epsilon_j h}=\mf L^{\epsilon_j h}(\Pi)$ and $\mathrm{index}_w(\Sigma_{\epsilon_j})\leq k$. We also have $\mf L^{\epsilon_j h}\rightarrow\mf L(\Pi)$ as $j\rightarrow \infty$. 
		
		Now applying Theorem \ref{thm:compactness for FPMC}, by the strong bumpiness of $g$, there exists a subsequence (still denoted by $\{\epsilon_j\}\rightarrow 0$) such that $\Sigma_j=\Sigma_{\epsilon_j}$ converges to a smooth, compact, embedded, free boundary minimal hypersurface $\Sigma$ (with integer multiplicity). We denote by $\mc Y$ the set of points where the convergence fails to be smooth. By Theorem \ref{thm:compactness for FPMC}, 
		\[\mf M(\Sigma)=\mf L(\Pi), \text{ and } \mathrm{index}(\Sigma)\leq k.\]

		We now prove that every component of $\Sigma_\infty$ is two-sided and has multiplicity one. Without loss of generality, we may assume that $\Sigma_\infty$ has only one connected component with multiplicity $m\in\mathbb N$. We will prove by contradiction.
		
		\medskip
		{\noindent\bf Part 1:}  We assume that $\Sigma_\infty$ is 2-sided; otherwise, we would consider the double cover of $\Sigma_\infty$ just like \cite{Zhou19}*{Proof of Theorem 4.1, Part 8}. Let $\nu$ be the global unit normal of $\Sigma$ and $X\in \wti {\mathfrak X}(M,\Sigma)$ be an extension of $\nu$. Suppose that $\phi_t$ is a one-parameter family of diffeomorphisms generated by $X$. For any domain $U\subset \Sigma$ and  small $\delta>0$, $\phi_t$ produces a neighborhood $U_\delta$ of $U$ with thickness $\delta$, i.e., $U_\delta=\{\phi_t(x)\,|\, x\in U, |t|\leq \delta\}$. We will use $(t, x)$ as coordinates on $U_\delta$. If $U$ is in the interior of $\Sigma$, then  for $\delta$ small $U_\delta$ is the same as $U\times [-\delta,\delta]$ in the geodesic normal coordinates of $\Sigma$. Now fix a domain $U\subset \subset \Sigma\setminus \mc{Y}$, by the convergence $\Sigma_j\to m\Sigma$, we know that for $j$ sufficiently large, $\Sigma_j\cap U_\delta$ can be decomposed to $m$ graphs over $U$ which can be ordered by height
		\[u_j^1\leq u_j^2\leq \cdots\leq u_j^m, \text{ and $u_j^i\rightarrow 0$, in smooth topology as $j\rightarrow\infty$.} .\]
		
		Since $\Sigma_j$ is the boundary of some set $\Omega_j$, we know that the unit outer normal $\nu_j$ of $\Omega_j$ will alternate orientations along these graphs; see \cite{Zhou19}*{Proof of Theorem 4.1, Part 3}.
		
		\medskip
		{\noindent\bf Part 2:} We first deal with an easier case: $m$ is an odd number. Hence $m\geq3$. Let $f$ be any smooth function on $\Sigma$ and $Z\in\mathfrak {\wti X}(M,\Sigma)$ be an extension of $f\nu$. Construct a family of hypersurfaces 
		\[\Sigma_{j,t}:=\{\phi((1-t)u_j^1+tu_j^m,x)|x\in U\},\]
		where $\phi(t,x):=\phi_t(x)$. Then $\Sigma_{k,0}$ and $\Sigma_{k,1}$ are the bottom and top sheets of $\Sigma_j$. Since $\Sigma_j$ is a free boundary $\epsilon_jh$-hypersurface, we have
		\[\int_0^1\Big[\frac{d}{dt}\int_{\Sigma_{j,t}}(\dv Z-\epsilon_jh\langle Z,\nu\rangle )\,d\mc H^n\Big]\,dt=0.\]
		Then the computation in Appendix \ref{section:2nd variation} gives that
		\begin{align*}
			0=&\int_0^1\Big[\int_{\Sigma_{j,t}}\Big(\langle\nabla^\perp(X^\perp),\nabla^\perp(Z^{\perp})\rangle-\Ric(X^\perp,Z^\perp)-|A|^2\langle X^\perp,Z^\perp\rangle\\
			&-\epsilon_j(\partial_\nu h)\langle X^\perp,Z^\perp\rangle\Big)\,d\mc H^n+ \int_{\partial \Sigma_{j,t}}\langle\nabla_{X^\perp}Z^\perp,\nu_{\partial M}\rangle\,d\mc H^{n-1}+\\
			&+\int_{\Sigma_{j,t}}\wti \Xi_1(X,Z,\mf H)\, d\mc H^{n}+\int_{\partial\Sigma_{j,t}}\wti\Xi_2(X,Z,\mf H,\bm\eta,\nu_{\partial M})\, d\mc H^{n-1}\Big]\,dt,
		\end{align*}
		where $X|_\Sigma=(u_j^m-u_j^1)\nu$, $\bm\eta$ is the co-normal of $\Sigma_{j,t}$, and 
		\[|\wti\Xi_1(X,Z,\mf H)|+|\wti\Xi_2(X,Z,\mf H,\bm\eta,\nu_{\partial M})|
		\leq  C\big(|X|(|\mf H+\epsilon_jh\nu|+| \bm\eta-\nu_{\partial M} |+|Z^\top|)+|X^\top|\big).\]
		Here $\mf H$ is the mean curvature vector. Denote by $\varphi_j=u_j^m-u_j^1$. By pulling everything back to $U$, we obtain 
		\begin{align*}
			0=&\int_0^1\Big[\int_{U}\Big(\langle\nabla \varphi_j,\nabla f\rangle- \Ric(\nu,\nu)\varphi_jf-|A^{\Sigma}|^2\varphi_jf-\epsilon_j(\partial_\nu h)\varphi_jf\Big)\,d\mc H^n\\
			&-\int_{\partial \Sigma\cap U}h^{\partial M}(\nu,\nu)\varphi_jf\,d\mc H^{n-1}+\int_{U}\wti W_j(t)(\varphi_j,f)\, d\mc H^{n}+\int_{\partial\Sigma\cap U}\wti w_j(t)(\varphi_j,f)\, d\mc H^{n-1}\Big]\,dt,
		\end{align*}
		where 
		\begin{equation}\label{eq:wti W and w}
			\wti W_j(t)(\varphi_j,f)\leq \wti \epsilon_j(f,\Sigma,M)|\varphi_j|,\ \wti w_j(t)(\varphi_j,f)\leq \wti \epsilon_j(f,\Sigma,M)|\varphi_j|.
		\end{equation}
		Here $\wti \epsilon_j\rightarrow 0$ uniformly as $j\rightarrow\infty$. Now letting $W_j=\int_0^1\wti W_j(t)dt$ and $w_j=\int_0^1 \wti w_j(t)dt$, by Fubini theorem, we have
		\begin{align*}
			0=&\int_{U}\Big(\langle\nabla\varphi_j,\nabla f\rangle- \Ric(\nu,\nu)\varphi_jf-|A^{\Sigma}|^2\varphi_jf-\epsilon_j(\partial_\nu h)\varphi_jf\Big)\,d\mc H^n\\
			&-\int_{\partial \Sigma\cap U}h^{\partial M}(\nu,\nu)\varphi_jf\,d\mc H^{n-1}+\int_{U} W_j(\varphi_j,f)\, d\mc H^{n}+\int_{\partial\Sigma\cap U} w_j(\varphi_j,f)\, d\mc H^{n-1}.
		\end{align*}
		Take $q\in U$. Let $\wti \varphi_j=\varphi_j/\varphi_j(q)$. Then the standard PDE theory implies that $\wti\varphi_j$ converges smoothly to a positive $\varphi\in C^\infty(U)$ satisfying 
		\begin{equation}
			\left\{\begin{aligned}
				& L_\Sigma\varphi=0, \text{\ \ on $U$},\\
				&\frac{\partial \varphi}{\partial \bm\eta}=h^{\partial M}(\nu,\nu)\varphi, \text{\ \  along $\partial\Sigma\cap U$.}
			\end{aligned}\right.
		\end{equation}
		Note that above argument works for any $U\subset\subset\Sigma\setminus\mc Y$. Taking an exhaustion of $\Sigma\setminus \mc Y$, we can extend $\varphi$ to $\Sigma\setminus \mc Y$ and such that 
		\begin{equation}
			\left\{\begin{aligned}
				& L_\Sigma\varphi=0, \text{\ \ on $\Sigma\setminus \mc Y$},\\
				&\frac{\partial \varphi}{\partial \bm\eta}=h^{\partial M}(\nu,\nu)\varphi, \text{\ \  along $\partial\Sigma\setminus \mc Y$.}
			\end{aligned}\right.
		\end{equation}
		
		\medskip
		{\noindent\bf Part 3:} Next we use White's local foliation argument to prove that $\varphi$ extends smoothly across $\mc Y$, and this will contradict the bumpy assumption of $g$. 
		
		By the work for interior singularity in \cite{Zhou19}*{Proof of Theorem 4.1, Part 5}, it suffices to show the uniform boundedness for $p\in\mc Y\cap \partial M$. Since $\Sigma$ has empty touching set, then $p\in \partial \Sigma$. Using the $h$-hypersurface with free boundary (see Proposition \ref{prop:h foliation with boundary}), we can also prove that $\varphi$ is bounded. Then the classical PDE gives that $\varphi$ is smooth across $\mc Y$. Thus, we conclude that there is a positive Jacobi field on $\Sigma$, which is a contradiction to the fact that $g$ is a strongly bumpy metric.
		
		\medskip
		{\noindent\bf Part 4:} We now take care the case when $m$ is even. Hence $m\geq 2$. Then without loss of generality we may assume that 
		\[
		H|_{\mathrm{Graph}(u_j^m)}(x)=-\epsilon_jh(x,u_j^m(x)), \text{ and } H|_{\mathrm{Graph}(u_j^1)} (x) = \epsilon_j h(x, u_j^1 (x)), \text{ for } x \in  U .\]
		Then by the argument in Part 2, we have
		\begin{align*}
			0=&\int_{U}\Big(\langle\nabla\varphi_j,\nabla f\rangle- \Ric(\nu,\nu)\varphi_jf-|A^{\Sigma}|^2\varphi_jf-\epsilon_j(\partial_\nu h)\varphi_jf +2\epsilon_jhf \Big)\,d\mc H^n\\
			&-\int_{\partial \Sigma\cap U}h^{\partial M}(\nu,\nu)\varphi_jf\,d\mc H^{n-1}+\int_{U} W_j(\varphi_j,f)\, d\mc H^{n}+\int_{\partial\Sigma\cap U} w_j(\varphi_j,f)\, d\mc H^{n-1}.
		\end{align*}
		Here $\varphi_j=u^m_j-u^1_j$; $ W$ and $ w$ are defined as the integral of $\wti W$ and $\wti w$ in \eqref{eq:wti W and w}. Fix a point $q\in \Sigma\setminus \mc Y$.
		
		\medskip
		{\noindent\bf Case 1:} $\limsup_{k\rightarrow\infty} \varphi_j(q)/\epsilon_j=+\infty$. Then the renormalizations $\wti \varphi_j=\varphi_j/\varphi_j(q)$ converges locally smoothly to a nontrivial function $\varphi\geq 0$ on $\Sigma\setminus \mc Y$, and
		by same reasoning as Part 2, we have
		\begin{equation*}
			\left\{\begin{aligned}
				& L_\Sigma\varphi=0, \text{\ \ on $\Sigma\setminus \mc Y$},\\
				&\frac{\partial \varphi}{\partial \bm\eta}=h^{\partial M}(\nu,\nu)\varphi, \text{\ \  along $\partial\Sigma\setminus \mc Y$.}
			\end{aligned}\right.
		\end{equation*}
		
		\medskip
		{\noindent\bf Case 2:} $\limsup_{k\rightarrow\infty} \varphi_j(q)/\epsilon_j< +\infty$. Consider renormalizations $\wti \varphi_j=\varphi_j/\epsilon_j$. Then again by the same reasoning, $\wti\varphi_j$ converges locally smoothly to a nonnegative $\varphi\geq 0$ on $\Sigma \setminus\mc Y$, and such that 
		\begin{equation}\label{eq:Lu equal 2h}
			\left\{\begin{aligned}
				& L_\Sigma\varphi=2h|_\Sigma , \text{\ \ on $\Sigma\setminus \mc Y$},\\
				&\frac{\partial \varphi}{\partial \bm\eta}=h^{\partial M}(\nu,\nu)\varphi, \text{\ \  along $\partial\Sigma\setminus \mc Y$.}
			\end{aligned}\right.
		\end{equation}
		Then by the argument in \cite{Zhou19}*{Proof of Theorem 4.1, Part 7} together with using the foliations by Proposition \ref{prop:h foliation with boundary}, we can also prove that $\varphi$ is smooth across $\mc Y$ in both cases. 
		
		\medskip
		{\noindent\bf Part 5:} Following the step in \cite{Zhou19}*{Proof of Theorem 4.1, Part 9}, one can also show that for a nicely chosen $h\in\mc S(g)$, the (unique) solutions to \eqref{eq:Lu equal 2h} must change sign. Thus there is no 1-sided component, and the multiplicity for 2-sided component must be one. 
		
		We only sketch the proof for two-sided case here. Note that every almost properly embedded free boundary minimal hypersurface has empty touching set since $g$ is strongly bumpy. Recall that by \citelist{\cite{GWZ18}\cite{Wang19}}, there are only finitely many free boundary minimal hypersurfaces with $\Area\leq \mf L(\Pi)$ and $\mathrm{index}\leq k$, denoted by $\{\Sigma_1,\Sigma_2,\cdots,\Sigma_N\}$. Then as in \cite{Zhou19}, we can take disjoint neighborhood $U_j^\pm\subset \Sigma_j$ so that $U_j^\pm\cap \partial M=\emptyset$. Since all the small neighborhoods are disjoint, we can take a smooth function $u$ defined on $\bigcup U_j^\pm$ whose supports are compact and 
		\begin{itemize}
			\item $u|_{U_j^+}$ is nonnegative and is positive at some point;
			\item $u|_{U_j^-}$ is nonpositive and is negative at some point.
		\end{itemize} 
		Then take $h_0\in C^\infty(M)$ be an extension of $Lu/2$ such that $h$ equals to $0$ in a neighborhood of $\partial M$. Recall that 
		\[ \bigcap_{j=1}^\infty\{  f:\epsilon_jf\in\mc S(g) \} \]
		is generic. Take $h$ in this generic set which is close to $h_0$ as we wanted. Then the solution of \eqref{eq:Lu equal 2h} would be close to $u/2$ on each $\Sigma_j$, therefore it must change sign. Thus, such an $h$ is a desired function. 
	\end{proof}
	
	\begin{remark}
		We remark here the theorem is stated only for stongly bumpy metrics. However, we also believe the multiplicity one holds true for metrics which are only bumpy. Such a result may need a highly nontrivial argument for constructing Jacobi fields; see \cite{Wang19}.
	\end{remark}
	
	\subsection{Application to volume spectrum}
	In this part, we will apply the result in Section \ref{subsec:multi one for relative sweepout} to study volume spectrum introduced by Gromov \cite{Gro88}, Guth \cite{Guth09}, and Marques-Neves \cite{MN17}. In particular, we will prove that for a strongly bumpy metric, the volume spectrum can be realized by the area of min-max minimal hypersurfaces with free boundary produced by Theorem \ref{thm:multi one for relative sweepout}.
	
	We first recall the definition of volume spectrum following \cite{MN17}*{Section 4}. Let $(M^{n+1},\partial M, g)$ be a compact Riemannian manifold with boundary. Let $X$ be a cubical subcomplex of $I^m = [0,1]$ for some $m\in\mb N$. Given $k\in\mb N$, a continuous map $\Phi: X\rightarrow\mc Z_n (M,\partial M; \mb Z_2)$ is called a \emph{$k$-sweepout} if
	\[0\neq \Phi^*(\bar{\lambda}^k) \in H^k(X, \mb Z_2),\]
	where $\bar{\lambda}\in H^1(\mc Z_n(M,\partial M;\mb Z_2))=\mb Z_2$ is the generator. A sweepout $\Phi$ is said to be \emph{admissible} if it has no concentration of mass (see Definition \ref{def:no concentration of mass}). Denote by $\mc P_k$ the set of all admissible $k$-sweepouts. Then
	\begin{definition}
		The \emph{$k$-width} of $(M,\partial M, g)$ is
		\[\omega_k (M, g) = \inf_{\Phi\in\mc P_k} \sup\{\mf M(\Phi(x)) : x\in \mathrm{dmn}(\Phi)\},\]
		where $\mathrm{dmn}(\Phi)$ is the domain of $\Phi$.
	\end{definition}
	
	The $k$-width satisfies a Weyl's asymptotic law. This asymptotic behaviour was first conjectured by Gromov in \cite{Gro88} and studied by Guth in \cite{Guth09}. Finally, Liokumovich-Marques-Neves proved the following Weyl law for $k$-width.
	
	\begin{theorem}[\cite{LMN16}*{\S` 1.1}]
		There exists a constant $a(n)>0$ such that, for every compact Riemannian manifold $(M^{n+1},g)$ with (possibly empty) boundary, we have
		\begin{equation}
			\lim_{p\to\infty}\omega_p(M)p^{-\frac{1}{n+1}}=a(n)\mathrm{vol}(M)^{\frac{n}{n+1}}.
		\end{equation}
	\end{theorem}
	
	Assume from now on that 
	$3\leq (n+1)\leq 7$. It was later observed by Marques-Neves in \cite{MN16} (see also \cite{GLWZ19}*{Section 4}) that one can restrict to a subclass of $\mc P_k$ in the definition of $\omega_k (M,g)$. In particular, let $\wti {\mc P}_k$ denote those elements $\Phi\in\mc P_k$ which is continuous under the $\mf F$-topology, and whose domain $X = \mathrm{dmn}(\Phi)$ has dimension $k$ (and is identical to its $k$-skeleton). Then
	\[
	\omega_k (M, g) = \inf_{\Phi\in\wti{\mc P}_k} \sup\{\mf M(\Phi(x)) : x \in\mathrm{dmn}(\Phi)\}.
	\]
	
	Following the idea of Marques-Neves \cite{MN16}, the last two authors together with Q.Guang and M. Li also proved in \cite{GLWZ19} that for each $k\in\mb N$, there exists a disjoint collection of smooth, connected, almost properly embedded, free boundary minimal hypersurfaces $\{\Sigma^k_i : i = 1,\cdots, l_k\}$ with integer multiplicities $\{m^k_i : i =
	1,\cdots, l_k \}\subset \mb N$, such that 
	\[
	\omega_k(M, g) = \sum_{i=1}^{l_k}m^k_i\cdot \Area(\Sigma^k_i),
	\ \ \sum_{i=1}^{l_k}\mathrm{index}(\Sigma^k_i)\leq  k.\]

	Before stating the main theorem, we recall an observation by \cite{MN17}*{Corollary 3.4}. Denote $S^1$ by the unit circle.
	\begin{lemma}[\citelist{\cite{MN17}*{Corollary 3.4}\cite{LMN16}*{Proposition 2.12}}]\label{lem:homotopy trivial}
		Let $\tau\in \mc Z_n(M,\partial M;\mb Z_2)$ so that its support is a properly embedded free boundary minimal hypersurface in $(M,\partial M,g)$. There exists $\bar\epsilon$ sufficiently small, depending on $\tau$ and $M$ so that every map $\Phi: S^1\rightarrow\mc Z_n (M,\partial M;\mb Z_2)$ with
		\[\Phi(S^1)\subset B_{\bar{\epsilon}}^\mc F (\tau) = \{\tau_1 \in\mc Z_n(M,\partial M;\mb Z_2) : \mc F(\tau_1, \tau ) < \bar\epsilon\}.\]
		is homotopically trivial.
	\end{lemma}

	We will also use the following Lemma proved in \cite{GLWZ19}. Such a result follows from the Morse index upper bound estimates for the free boundary min-max theory in \cite{GLWZ19}.
	\begin{lemma}[\cite{GLWZ19}, Theorem 2.1]\label{lem:index bound for fbmh}
		Suppose $g$ is bumpy, then for each $k\in\mb N$, there exist a $k$-dimensional cubical complex $X$ and a map $\Phi_{0,k} : X \rightarrow\mc Z_n (M,\partial M,\mf F;\mb Z_2 )$ continuous in the $\mf F$-topology with $\Phi_{0,k}\in\wti{\mc P}_k$, such that
		\[
		\mf L(\Pi_k ) = \omega_k (M, g),\]
		where $\Pi_k = \Pi(\Phi_{0,k})$ is the class of all maps $\Phi : X \rightarrow\mc Z_n (M,\partial M,\mf F;\mb Z_2)$ continuous in the $\mf F$-topology that are homotopic to $\Phi_{0,k}$ in flat topology.
		
		Moreover, there exists a pulled-tight (see \cite{GLWZ19}*{Theorem 5.8}) minimizing sequence $\{\Phi_i\}_{i\in\mb N}$ of $\Pi_k$ such that if $\Sigma\in\mf C(\{\Phi_i\}_{i\in\mb N})$ has support a compact, smooth, almost properly embedded, free boundary minimal hypersurface, then
		\[\|\Sigma\|(M) = \omega_ k(M,g), \ \text{ and } \ \mathrm{index}_w(\text{support of $\Sigma$})\leq k.\]
	\end{lemma}

	Now we are going to state and prove our main theorem.
	\begin{theorem}[same with Theorem \ref{thm:intro:multi 1}]\label{thm:multiplicity one for volume spectrum}
		If $g$ is a strongly bumpy metric and $3\leq (n +1) \leq 7$, then for each $k\in\mb N$, there
		exists a disjoint collection of compact, smooth, connected, properly embedded, two-sided, free boundary minimal hypersurfaces $\{\Sigma^k_i : i = 1, \cdots , l_k \}$, such that
		\[
		\omega_k(M,g)=\sum_{i=1}^{l_k}\Area(\Sigma^k_i), \ \ \mathrm{index}(\Sigma^k_i)\leq k.\]
		That is to say, the min-max minimal hypersurfaces with free boundary are all two-sided and have multiplicity one for generic metrics.
	\end{theorem}

	\begin{proof}[Proof of Theorem \ref{thm:multiplicity one for volume spectrum}]
		Since $g$ is bumpy, then there are only finitely many compact, almost properly embedded, free boundary minimal hypersurfaces with $\Area \leq\Lambda$ and $\mathrm{index}\leq I$ for given $\Lambda > 0, I\in\mb N$ by \citelist{\cite{GWZ18}\cite{Wang19}}; see \cite{ACS17} for strongly bumpy metrics. 
		
		\medskip
		Now we fix $k\in\mb N$ and omit the sub-index $k$ in the following. Take $\Pi = [\Phi_0 : X \rightarrow\mc Z_n (M,\partial M, \mf F;\mb Z_2 )]$ with $\mf L(\Pi)=\omega_k$.

		We proceed the proof by the following three steps.
		
		\medskip
		{\noindent\bf Step I:} {\em In this step, we show how to find another minimizing sequence, still denoted by $\{\Phi_i\}_{i\in\mb N}$, such that for $i$ sufficiently large, either $|\Phi_i(x)|$ is close to a regular min-max free boundary minimal hypersurface, or the mass $\mf M(\Phi_i(x))$ is strictly less than $\omega_k(M,g)$.} 
		
		\medskip
		We recall the following observation by \cite{MN17}*{Claim 6.2}. Let $\mc S$ be the set of all stationary integral varifolds with $\Area\leq\omega_k$ whose support is a compact, smooth, almost properly embedded, free boundary minimal hypersurface with $\mathrm{index}(\text{support})\leq k$. Consider the set $\mc T$ of all $\tau\in\mc Z_n (M,\partial M;\mb Z_2)$ with $\mf M(\tau)\leq \omega_k$
		and either $\tau = 0$ or the support of $\tau$ is a compact, smooth, properly embedded, free boundary minimal hypersurface with $\mathrm{index}\leq k$. By the bumpy assumption, both sets $\mc S$ and $\mc T$ are finite. Moreover,
		\begin{claim}[\cite{MN17}*{Claim 6.2}]\label{lem:mfF implies mcF}
			For every $\bar{\epsilon} > 0$, there exists $\epsilon > 0$ such that $\tau\in \mc Z_n(M,\partial M;\mb Z_2 )$ with $\mf F(|\tau|, \mc S) \leq  3\epsilon \Rightarrow \mc F(\tau, \mc T ) < \bar{\epsilon}$.
		\end{claim}
		\begin{proof}[Proof of Claim \ref{lem:mfF implies mcF}]
		Here since $g$ is strongly bumpy, then each element in $\mc S$ is supported on a properly embedded hypersurface. This implies that Constancy Theorem \cite[\S 26.27]{Si} can be applied. Then the proof is the same with \cite{MN17}*{Claim 6.2}.
		\end{proof}

		Let $\{\Phi_i\}_{i\in\mb N}$ be chosen as in Lemma \ref{lem:index bound for fbmh}. We choose $\bar{\epsilon}$ as Lemma \ref{lem:homotopy trivial} so that every map $\Phi:S^1\rightarrow \mc Z_n(M,\partial M;\mb Z_2)$ with $\Phi(S^1)\subset B_{\bar\epsilon}^{\mc F}(T)$ is homotopically trivial. According to such $\bar\epsilon$, we then choose $\epsilon>0$ by Claim \ref{lem:mfF implies mcF}. Take a sequence $\{k_i\}_{i\in\mb N}\rightarrow\infty$ so that 
		\[
		\sup\{\mf F(\Phi_i(x), \Phi_i(y)) :\alpha \in X(k_i), x, y\in\alpha\}\leq\epsilon/2.
		\]
		Consider $Z_i$ to be the cubical subcomplex of $X(k_i)$ consisting of all cells $\alpha\in X(k_i)$ such that
		\[\mf F(|\Phi_i(x)|,\mc S)\geq \epsilon, \text{ for every vertex $x$ in $\alpha$}.\]
		Hence $\mf F(|\Phi_i (x)|,\mc S) \geq \epsilon/2$ for all $x\in Z_i$. Consider this sub-coordinating sequence $\{\Phi_i|_{Z_i}\}_{i\in\mb N}$. $\mf L(\{\Phi_i|_{Z_i}\})$ and $\mf C(\{\Phi_i|_{Z_i}\})$ are defined in the same way as in Section \ref{subsec:relative min-max} with $\mc A^h$ replaced by $\mf M$.

		Let $Y_i = \overline{X\setminus Z_i}$. It then follows that
		\[
		\mf F(|\Phi_i(x)|, \mc S)<\frac{3}{2}\epsilon, \text{ for all $x\in Y_i$}.\]
		We also denote $B_i = Y_i\cap Z_i$. In fact, $B_i$ is the topological boundary of $Y_i$ and $Z_i$. For a later purpose, we also consider the set
		\[
		\mf B_i = \text{ the union of all cells $\alpha\in Z_i$ such that $\alpha \cap B_i\neq \emptyset$}.\]
		$\mf B_i$ can be thought of the ``thickening'' of $B_i$ inside $Z_i$.
		
		Let $\lambda=(\Phi_i)^*{\bar{\lambda}}\in H^1(X,\mb Z_2)$. Let $Y_i'= Y_i\cup\mf B_i$ and $Z_i' = \overline{Z_i\setminus \mf B_i}$, and $i_1':Y_i'\rightarrow X$ and $i'_2 : Z_i'\rightarrow X$ be the inclusion maps. Then by Lemma \ref{lem:homotopy trivial}, we have
		\[(i'_1)^*({\lambda}) = 0\in H^1 (Y_i';\mb Z_2), \text{ and } (i'_2)^*({\lambda}^{k-1})\neq 0\in H^{k-1}(Z_i'; \mb Z_2).\]
		Now consider $\wti B_i=Y_i'\cap Z_i'$ and the set
		\[\wti{\mf B}_i = \text{ the union of all cells $\alpha\in Z_i'$ such that $\alpha \cap \wti B_i\neq \emptyset$}.\]
		Let $\wti Y_i=Y_i'\cup \wti{\mf B}_i$ and $\wti Z_i=\overline{Z_i'\setminus \wti{\mf B}_i}$ and $\wti i_1:\wti Y_i\rightarrow X$ and $\wti i_2 : \wti Z_i\rightarrow X$ be the inclusion maps. Then by Lemma \ref{lem:homotopy trivial}, we also have
		\[(\wti i_1)^*({\lambda}) = 0\in H^1 (\wti Y_i;\mb Z_2), \text{ and } (\wti i_2)^*({\lambda}^{k-1})\neq 0\in H^{k-1}(\wti Z_i; \mb Z_2).\]

		\begin{claim}\label{claim:less Mass on Zi}
			$\{\Phi_i\}$ can be deformed so that 
			\[  \mf L(\{\Phi_i|_{Z_i}\}_{i\in\mb N})<\mf L(\Pi)=\omega_k. \]	
		\end{claim}
		\begin{proof}[Proof of Claim \ref{claim:less Mass on Zi}]
			By the work of the last author and M. Li \cite{LZ16} (see also \cite{GLWZ19}*{Theorem 4.5} for the adaptions to $\mb Z_2$-coefficients), together with Lemma \ref{lem:index bound for fbmh}, we also have the following dichotomy:
			\begin{itemize}
				\item no element $V\in\mf C(\{\Phi_i|_{Z_i}\}_{i\in\mb N})$ is $\mb Z_2$-almost minimizing in small annuli with free boundary (see \cite{LZ16}*{Definition 4.19}),
				\item or
				\begin{equation}
					\mf L(\{\Phi_i|_{Z_i}\}_{i\in\mb N})<\mf L(\Pi)=\omega_k.
				\end{equation}
			\end{itemize}
			
			For the latter case, we are done. 
			
			We now assume the first case happens. The argument has been well-known and we only sketch the ideas here. For each $\Phi_i$, we can use the Discretization Theorem \cite{LZ16}*{Theorem 4.12} (or \cite{GLWZ19}*{Theorem 4.5}) to get a sequence of discrete maps $\phi_i^j:X(k_i^j)_0\rightarrow \mc Z_n(M,\partial M;\mb Z_2)$. Taking $j(i)$ sufficiently large and by assumptions, we can deform $\phi_i^{j(i)}$ and then use the Almgren extension \cite{LZ16}*{Theorem 4.14} to get a sequence of maps $\Psi_i:X\rightarrow \mc Z_n(M,\partial M;\mb Z_2)$ which is still a minimizing sequence of $\Pi$ and 
			\[ \mf L(\{\Psi_i|_{Z_i}\}_{i\in\mb N})<\mf L(\Pi)=\omega_k.\]
			This finishes the proof of Claim \ref{claim:less Mass on Zi}. We refer to \cite{Zhou19}*{Proof of Theorem 5.2, Step 2} for more details. 
		\end{proof}
		
		We remark that since $\wti{\mf B}_i\subset Z_i'$, we have
		\[\mf L(\{\Phi_i|_{\wti{\mf B}_i}\}_{i\in\mb N})<\omega_k.\]

		\medskip
		{\noindent\bf Step II:} {\em Now we want to produce sweepouts in $\mc C(M)$ by lifting to the double cover $\partial:\mc C(M )\rightarrow\mc Z_n (M,\partial M;Z_2)$ so as to produce sweepouts satisfying the assumption of Theorem \ref{thm:multi one for relative sweepout}.}
		
		\medskip
		Note that $\Phi_i\circ\wti i_1:\wti Y_i\rightarrow\mc Z_n(M,\partial M;\mb Z_2)$ is homotopically trivial. Then by the lifting criterion \cite{Hat}*{Proposition 1.33}, there exist lifting maps $\wti \Phi_i:\wti Y_i\rightarrow (\mc C(M),\mf F)$ so that $\partial \circ\wti \Phi_i=\Phi_i\circ \wti i_1$.
		\begin{lemma}\label{lem:prevent LPi}
			For $i$ large enough, if $\wti\Pi_i$ is the $(\wti Y_i,\wti{\mf B}_i)$-homotopy class associated with $\wti \Phi_i$, then we have
			\[\mf L(\wti \Pi_i)\geq \mf L(\Pi)>\max_{x\in\wti{\mf B}_i} \mf M(\partial\wti\Phi_i (x)).\]
		\end{lemma}
		\begin{proof}[Proof of Lemma \ref{lem:prevent LPi}]
			Fix $i$ large, so that
			\[\sup_{x\in\wti{\mf B}_i} \mf M(\Phi_i(x)) < \mf L(\Pi),\]
			and we will omit the sub-index $i$ in the following proof.
			
			If the conclusion were not true, then we can find a sequence of maps $\{\wti \Psi_j:\wti Y\rightarrow(\mc C(M),\mf F)\}\subset \wti\Pi$, such that
			\begin{equation}\label{eq:less than wk in Y}
				\limsup_{j\rightarrow\infty} \sup\{\mf M(\partial \wti \Psi_j(x)) :x\in\wti Y\}<\mf L(\Pi),
			\end{equation}
			and homotopy maps $\{\wti H_j:[0,1]\times\wti Y\rightarrow \mc C(M)\}$, which are continuous in the flat topology, $\wti H_j(0,\cdot)=\wti\Phi$, $\wti H_j(1,\cdot)=\wti \Psi_j$, and
			\begin{equation}\label{eq:less than wk in B}
				\limsup_{j\rightarrow\infty} \sup\{\mf F(\wti H_j (t,x), \wti \Phi(x)):t\in [0, 1], x\in\wti{\mf B}\}=0.
			\end{equation}
			
			We now construct new sequences of maps $\{\Psi_j:X\rightarrow \mc Z_n(M,\partial M;\mb Z_2)\}_{j\in\mb N}$ and $\{H_j:[0,1]\times X\rightarrow\mc Z_n(M,\partial M;\mb Z_2)\}$ defined as:
			\begin{itemize}
				\item $H_j(t,x)=\Phi(x)$ if $x\in X\setminus\wti Y$;
				\item $ H_j(t,x)=\partial \circ \wti H_j(t(1-s(x)),x)$ if $x\in\wti Y$, where $s(x)=\min\{1,\dist(x,Y')\}$;
				\item $\Psi_j= \wti H_j(1,\cdot)$.
			\end{itemize}
			Then $\Psi$ is continuous and homotopic to $\Phi$ in the flat topology. Moreover, \eqref{eq:less than wk in Y} and \eqref{eq:less than wk in B} give that 
			\[\limsup_{j\rightarrow\infty}\sup_{x\in \wti Y}\mf M(\Psi_j(x))<\mf L(\Pi).\]
			Recall that $\Psi(x)=\Phi(x)$ for $x\in X\setminus \wti Y$. Then by Claim \ref{claim:less Mass on Zi},
			\[	\mf L(\{\Psi_j|_{X\setminus \wti Y}\}_{i\in\mb N})=\mf L (\{\Phi|_{X\setminus \wti Y}\}_{i\in\mb N})<\mf L(\Pi)=\omega_k.\]
			Thus we conclude that 
			\begin{equation}\label{eq:less than wk in X}
				\limsup_{j\rightarrow\infty}\sup_{x\in X}\mf M(\Psi_j(x))<\mf L(\Pi)=\omega_k.
			\end{equation}
			Note that \eqref{eq:less than wk in B} also gives that $\{\Psi_j\}$ has no concentration of mass. This implies that $\Psi$ is also a $k$-sweepout, so \eqref{eq:less than wk in X} contradicts the definition of $\omega_k$.
		\end{proof}
		
		\medskip
		{\noindent\bf Step III:} {\em Now we are ready to prove Theorem \ref{thm:multiplicity one for volume spectrum}.}
		
		\medskip
		For $i$ large enough in Lemma \ref{lem:prevent LPi}, Theorem \ref{thm:multi one for relative sweepout} applied to $\wti \Pi_i$ gives a disjoint collection of compact, smooth, connected, properly embedded, 2-sided, free boundary minimal hypersurfaces $\Sigma_i =\cup^{N_i}_{j=1}\Sigma_{i,j}$, such that
		\[
		\mf L(\wti \Pi_i)=\sum_{j=1}^{N_i}\Area(\Sigma_{i,j}),\  \text{ and }\ \mathrm{index}(\Sigma_i)\leq k.\]
		Recall that $\mf L(\wti \Pi_i)\leq \mf L(\Phi_i)\rightarrow \mf L(\Pi)=\omega_k$ as $i\rightarrow \infty$.Together with Lemma \ref{lem:prevent LPi}, $\mf L(\wti \Pi_i)\rightarrow \omega_k$ as $i\rightarrow\infty$. Counting the fact that there are only finitely many compact, smooth, properly embedded, free boundary hypersurfaces with $\Area\leq \omega_k +1$ and $\mathrm{index}\leq k$, for $i$ sufficiently large we have
		\[\mf L(\wti\Pi_i) = \mf L( \wti \Pi_{i+1} ) = \cdots = \omega_k .\]
		Hence we finish the proof of Theorem \ref{thm:multiplicity one for volume spectrum}.
	\end{proof}
	
	\section{Existence of self-shrinkers with arbitrarily large entropy}\label{sec:min-max in Gaussian space}
	In this section, we consider the min-max theory in $\mb R^{n+1}$ with the Gaussian metric by virtue of Theorem \ref{thm:multiplicity one for volume spectrum}. We will prove that each $k$-width is realized by a connected, embedded self-shrinker with multiplicity one.
	
	Recall that an embedded hypersurface $\Sigma\subset\mb R^{n+1}$ is called a {\em self-shrinker} if and only if 
	\[  H=\frac{\langle x,\nu\rangle}{2},\]
	where $H$ is the mean curvature with respect to the unit normal vector field $\nu$. It is equivalent to say that $\Sigma$ is minimal under the Gaussian metric $(\mb R^{n+1},(4\pi)^{-1}e^{-\frac{|x|^2}{2n}}\delta_{ij})$. We refer to \cite{CM12_1} for more results concerning about self-shrinkers.
	
	Before we go to details, let us first sketch the idea of the proof. We will construct min-max free boundary minimal hypersurfaces in larger and larger balls in $\R^{n+1}$ with generic metrics near the Gaussian metric. Then by passing to limits we will obtain complete minimial hypersurfaces in $(\mb R^{n+1},(4\pi)^{-1}e^{-\frac{|x|^2}{2n}}\delta_{ij})$.  The limits are non-trivial because the areas of the sequence of the free boundary minimal hypersurfaces outside a large ball are uniformly small (see Claim \ref{claim:no mass loss}). The Weyl law ensures that the limits have arbitrarily large area. Finally an index estimate (Theorem \ref{thm:shrinker finite index}) allows us to say something about the indices of higher dimensional self-shrinkers constructing by multiplying a low dimensional self-shrinker with linear spaces.

	\subsection{Min-max theory in Gaussian metric spaces}
	Denote by $\mc G=(4\pi)^{-1}e^{-\frac{|x|^2}{2n}}\delta_{ij}$. We first give the definition of volume spectrum for $(\mb R^{n+1};\mc G)$.
	\begin{definition}
		Given a sequence of compact domain $K_1\subset\cdots\subset K_j\subset\cdots$ exhausting $\mb R^{n+1}$, then for any positive integer $k$, we define
		\[ \omega_k(\mb R^{n+1};\mc G)=\lim_{j\rightarrow\infty}\omega_k(K_j;\mc G).\]
	\end{definition}
	We remark that such a definition does not depend on the choice of the sequences.
	
	The following definition of index is well-known:
	\begin{definition}\label{def:index of self-shrinkers}
		Let $\Sigma$ be an embedded self-shrinker in $\mb R^{n+1}$. We say that $\Sigma$ has {\em $\mathrm{Index}\geq k$} if there exists a $k$-dimensional subspace $W$ of $C^\infty(\Sigma)$ such that each nonzero $f\in W$ has compact support and 
		\[  \int_\Sigma (|\nabla f|^2-|A|^2f^2-\frac{1}{2}f^2)e^{-\frac{|x|^2}{4}}d\mc H^n<0.\]
		
		$\Sigma$ has index $k$ if and only if $\Sigma$ has $\Index\geq k$ but does not have $\Index\geq (k+1)$.
	\end{definition}
	
	Now we are going to prove the main theorem in this section.
	
	\begin{proof}[Proof of Theorem \ref{thm:intro:min-max in gaussian space}]
		In the following, $B_R(0)$ always denotes the ball in $\mb R^{n+1}$ with radius $R$ under the Euclidean metric. Let $\{R_j\}$ be a sequence of positive numbers with $R_j\rightarrow\infty$. Denote by $\Omega_j=B_{R_j}(0)$. For each $j$, we now take a perturbed metric $g_j$ on $\Omega_j$ so that 
		\begin{itemize}
			\item for any compact domain $\Omega\subset \mb R^{n+1}$, $g_j|_{\Omega}\rightarrow \mc G$ smoothly;
			\item under the metric $g_j$, $B_R(0)$ has mean concave boundary for each $3<R\leq R_j$;
			\item $g_j|_{\Omega_j}$ is strongly bumpy;
			\item $\omega_k(\Omega_j;g_j)\rightarrow \omega_k(\mb R^{n+1};\mc G)$.
		\end{itemize} 	
		The last item can be satisfied because $\omega_k(\Omega;g)$ depends continuously on the metric $g$ by Irie-Marques-Neves \cite{IMN17}*{Lemma 2.1}.  Hence without loss of generality,
		\[ \omega_k(\Omega_j;g_j)<\Lambda_k:=\omega(\mb R^{n+1};\mc G)+1.\]

		Since $g_j$ is strongly bumpy, by Theorem \ref{thm:multiplicity one for volume spectrum}, there exists a free boundary minimal hypersurface 
		\[(\Sigma_j,\partial \Sigma_j)\subset (\Omega_j,\partial\Omega_j;g_j)\]
		so that $\Area(\Sigma_j;g_j)=\omega_k(\Omega_j;g_j)$. By the compactness of minimal surfaces with bounded area and index (see \cite{Sharp17}), $\Sigma_j$ subsequently, locally smoothly converges to a smooth minimal hypersurface $\Sigma$ in $(\mb R^{n+1};\mc G)$ with multiplicity $m\geq 1$ away from a finite set $\mc W$. Then using the fact of non-existence of stable minimal hypersurfaces in $(\mb R^{n+1};\mc G)$, $m$ can only be $1$. Thus $\Sigma_j$ locally smoothly converges to $\Sigma$. We refer to \cite{CM12} for more details about this kind of convergence. Then by the Frankel property for self-shrinkers \cite{CCMS}*{Corollary C.4}, $\Sigma$ is connected. To finish the proof, we prove that no mass is lost in the convergence.
		\begin{claim}\label{claim:no mass loss}
			There exist constants $C,L>0$ depending only on $n$ so that for all $R>L$ and sufficiently large $j$,
			\[ \Area (\Sigma_j\setminus B_R(0);g_j)<C\Lambda_kR^{\frac{1}{2}-n}e^{-\frac{R^2}{4}(1-\frac{1}{2n})}.\]	
		\end{claim}
		\begin{proof}[Proof of Claim \ref{claim:no mass loss}]
			For $s>0$, let $r(s)=\frac{1}{\sqrt {4\pi}}\int_s^{\infty}e^{-t^2/(4n)}dt$. Clearly,
			\begin{equation}\label{eq:growth of r}
				\lim_{s\rightarrow\infty }\sqrt {4\pi}\cdot s\cdot e^{s^2/(4n)}r(s)=2n. 
			\end{equation}
			Hence we can take $L>4n$ large enough so that for all $s>L$,
			\begin{equation}\label{eq:choice of L}
				2n+\frac{1}{10}> \sqrt{4\pi}\cdot s\cdot e^{s^2/8}r(s)>2n-\frac{1}{10}.
			\end{equation}
			For $x\in\mb R^3$, we define $r(x)=r(|x|)$. Then $r(x)$ is the distance to $\infty$ under $(\mb R^3;\mc G)$. Denote by $\nabla$ and $\nabla ^{(j)}$ the Levi-Civita connection associated with $\mc G$ and $g_j$, respectively. Then $|\nabla r|_\mc G=1$ and 
			\begin{equation}\label{eq:compute hessian}
				\mc G(\nabla_{e_1}\nabla r^2/2,e_2)=\mc G(e_1,\nabla r)\mc G(e_2,\nabla r)+r\mc G(\nabla_{e_1}\nabla r,e_2). 
			\end{equation}
			By direct computations, $\nabla_{\nabla r}\nabla r=0$, and for $e_1,e_2\in T_x(\partial B_{|x|}(0))$,
			\begin{equation*}  
				\mc G(\nabla_{e_1}\nabla r,e_2)=\sqrt {4\pi}\cdot e^{\frac{|x|^2}{4n}}\big(\frac{|x|}{2n}-\frac{1}{|x|}\big)\mc G(e_1,e_2),
			\end{equation*}	
			which implies that for $|x|>L>4n$ and $e\in T_x(\partial B_{|x|}(0))$, 
			\[ \mc G(\nabla_{e}\nabla r^2/2,e)\geq\sqrt{4\pi}\cdot  r\cdot e^{|x|^2/8}\cdot \frac{|x|}{2n}(1-\frac{2n}{|x|^2}) \cdot \mc G(e,e)>(1-\frac{1}{4n})\mc G(e,e). \]
			Here we used \eqref{eq:choice of L} in the last inequality.
			Together with \eqref{eq:compute hessian}, we have that for $|x|>L$ and any $e\in T_x\mb R^3$,
			\[  \mc G(\nabla_e\nabla r^2/2,e)>(1-\frac{1}{4n})\mc G(e,e). \]

			Then can assume that for any $e\in T_x\mb R^3$ and $|x|>L$ (by taking $g_j$ close to $\mc G$),
			\begin{equation}\label{eq:hessian bound}
				1+1/j>g_j(\nabla r,\nabla r)> 1-1/j \ \text{ and } \
				g_j(\nabla _{e}^{(j)}\nabla^{(j)}r^2/2,e)>(1-\frac{1}{4n})\cdot g_j(e,e).
			\end{equation}

			Given $0<t<u<1$, we define two sequences of cut-off functions $\phi_{t},\eta_{u}:\mb R\rightarrow\mb R$ so that 
			\begin{gather*}
				\eta_u'\leq 0; \  \ \eta_u(x)=1 \text{ for } x\leq u; \  \ \eta_u(x)=0 \text{ for } x\geq 1;\\
				\phi_t'\geq 0;\ \  \phi_t(x)=0 \text{ for } x\leq t/2;\ \  \phi_t(x)=1 \text{ for } x\geq t.
			\end{gather*} 
			Since $\Sigma_j$ is a minimal hypersurface in $(\Omega_j;g_j)$, then by the divergence theorem, for any $\rho\in( r(R_j),r(L))$,
			\begin{align*}
				0&=\int_{\Sigma_j}\dv^{(j)}_{\Sigma_j}\big[\phi_t(r(x)-r(R_j))\cdot \eta_u(r/\rho)\cdot \nabla^{(j)}r^2/2 \big]\,d\mu_{g_j}\\
				&\geq \int_{\Sigma_j}\phi_t\cdot \eta'_u(r/\rho)\cdot\frac{1}{\rho}\cdot g_j(\nabla^{(j)} r,\nabla^{(j)}r^2/2)+\phi_t\cdot \eta_u(r/\rho)\cdot \dv^{(j)}_{\Sigma_j}(\nabla^{(j)}r^2/2)\, d\mu_{g_j}.\nonumber\\
				&\geq \int_{\Sigma_j}\phi_t \cdot \eta_u'(r/\rho)\cdot \frac{1}{\rho}\cdot (1+1/j)r+\phi_t\cdot \eta_u(r/\rho)\cdot (n-\frac{1}{4})\, d\mu_{g_j}.\nonumber
			\end{align*}
			Here \eqref{eq:hessian bound} is used in the last inequality. Letting $t\rightarrow 0$, we have 
			\begin{align*}
				0\geq\int_{\Sigma_j}-\rho\frac{d}{d\rho}\eta_u(r/\rho)+ \eta_u(r/\rho)\cdot (n-\frac{1}{2})\, d\mu_{g_j}.
			\end{align*}
			Therefore, we conclude that for $\rho \in [r(R_j),r(L)]$ and sufficiently large $j$,
			\[ (n-\frac{1}{2})\cdot I_j(\rho)\leq\rho\frac{d}{d\rho}I_j(\rho), \]
			where $I_j(\rho)=\int_{\Sigma_j}\eta_u(r/\rho)$. Such an inequality implies that $I_j(\rho)\rho^{\frac{1}{2}-n}$ is monotone increasing. Thus, for $\rho\in [r(R_j),r(L)]$,
			\begin{equation}\label{eq:estiamte of I}
				I_j(\rho)\leq I_j(r(L))(r(L))^{\frac{1}{2}-n}\rho^{n-\frac{1}{2}} .
			\end{equation}
			Now letting $u\rightarrow 1$ in $\eta_u$, then 
			\[ I_j(r(R))\rightarrow \Area( \{x\in\Sigma_j:r(|x|)<r(R)\};g_j)=\Area(\Sigma_j\setminus B_{R}(0);g_j).\]
			Together with \eqref{eq:growth of r}, the inequality in \eqref{eq:estiamte of I} becomes 
			\[ \Area(\Sigma_j\setminus B_R;g_j)\leq \Lambda_k\cdot  (r(L))^{\frac{1}{2}-n} R^{\frac{1}{2}-n}e^{-\frac{R^2}{4}(1-\frac{1}{2n})} \]
			for all $R\in (L,R_j)$. The proof of Claim \ref{claim:no mass loss} is finished by taking $C=(r(L))^{\frac{1}{2}-n}$.
		\end{proof}
		
		To proceed the proof of theorem, it follows from Claim \ref{claim:no mass loss} that 
		\[\Area(\Sigma;\mc G)=\lim_{j\rightarrow\infty}\Area(\Sigma_j;g_j)=\lim_{j\rightarrow\infty }\omega_k(\Omega_j;g_j)=\omega_k(\mb R^3;\mc G).\]
		We conclude that Theorem \ref{thm:intro:min-max in gaussian space} is finished.
	\end{proof}

	\subsection{Index estimates}
	In order to prove Corollary \ref{cor:intro:large entropy and finite index}, we provide some equivalent conditions for a self-shrinker to have finite index.
	
	Let $\Sigma$ be an embedded self-shrinker in $\mb R^{n+1}$. Throughout, $L_\Sigma=\Delta_\Sigma
	+|A|^2+\frac{1}{2}-\frac{x}{2}\cdot \nabla$ will be the Jacobi operator from the second variation formula. Such an operator is associated with a bi-linear form:
	\[  \mathfrak B(u,v)=\int_\Sigma(\langle\nabla  u,\nabla v\rangle -|A|^2uv-\frac{1}{2}uv )e^{-\frac{|x|^2}{4}}.\]
	Then bottom of the spectrum $\mu_1$ of $\Sigma$ is defined as
	\[   \mu_1:=\inf_{f}\frac{\mathfrak B(f,f)}{\int_\Sigma f^2e^{-\frac{|x|^2}{4}}}, \]
	where the infimum is taken over smooth functions $f$ with compact support. Since $\Sigma$ may be noncompact, we allow the possibility that $\mu_1 = -\infty$.
	
	In the following we will focus on the index of a self-shrinker of the form $\Sigma\times\R$, where $\Sigma$ is a lower dimensional self-shrinker. We will always use $x$ to denote a point on $\Sigma$ and use $y$ to denote a point in $\R$. We will also use $L_\Sigma$ and $L_{\Sigma\times\R}$ to denote the second variational operator on $\Sigma$ and $\Sigma\times\R$ respectively. Note that
	\begin{equation}
		L_{\Sigma\times\R} =L_{\Sigma} +\partial_y^2 -\frac{1}{2}y\partial_y.
	\end{equation}
	
	The main theorem in this section is the following index estimate, which would help us to obtain the finiteness of index in Corollary \ref{cor:intro:large entropy and finite index}.
	
	\begin{theorem}\label{thm:shrinker finite index}
		Supposing $\Sigma$ is a self-shrinker with finite index, then $\Sigma\times\R$ is a self-shrinker with finite index.
	\end{theorem}
	
	We need the following correspondence between the eigenfunctions and eigenvalues on $\Sigma$ and $\Sigma\times\R$ respectively. In the following the eigenvalues and eigenfunctions are under the Dirichlet boundary conditions.

	Let $\Omega\subset \Sigma$ be a bound domain with smooth boundary. Denote by $\mc W^{1,2}_0(\Omega)$ the closure of 
	\[  \{f\in C^{\infty}(\Sigma):\spt f\subset \mathrm{Int}\ \Omega\}\]
	in the topology of $\mc W^{1,2}(\Sigma)$. Then the self-adjoint operator $L_\Sigma$ has discrete Dirichlet eigenvalues
	\[ \mu_1(\Omega)<\mu_2(\Omega)\leq \mu_3(\Omega)\leq \cdots\rightarrow +\infty,\]
	and associated eigenfunctions $\{f_j\}\subset \mc W^{1,2}_0(\Omega)$. Moreover, $\{f_j\}$ forms a complete basis of the weighted $L^2(\Omega)$; see \cite{Str08}*{\S 11.3 Theorem 2}.
	
	\begin{proposition}\label{prop:eigen-product}
		Suppose $\Omega\subset\Sigma$ is a bounded compact subset with smooth boundary. Suppose $\mu_1(\Omega)<\mu_2(\Omega)\leq\cdots$ are eigenvalues of $L_\Sigma$ on $\Omega$, with associated eigenfunctions $\{\phi_i\}_{i=1}^\infty$; suppose $\nu_1^T<\nu_2^T\leq\cdots$ are eigenvalues of $\partial_y^2-\frac{1}{2}y\partial_y$ on $[-T,T]\subset\R$, with associated eigenfunctions $\{\psi_j\}_{j=1}^\infty$. Then the eigenvalues of $L_{\Sigma\times\R}$ on $\Omega\times[-T,T]$ are $\{\mu_i(\Omega)+\nu_j^T\}_{i,j=1}^{\infty}$, with associated eigenfunctions $\phi_i\psi_j$.
	\end{proposition}
	
	\begin{proof}
		This theorem is just the classical result on the spectrum of the elliptic operators on product manifolds. Here we sketch the proof for completeness. Direct calculation shows that $\phi_i\psi_j$ are eigenfunctions of $L_{\Sigma\times\R}$, with eigenvalue $\mu_i(\Omega)+\nu_j^T$. 
		
		Recall that $\{\phi_i\}_{i=1}^\infty$ is a basis in the weighted $L^2(\Omega)$ sense. Similarly, $\{\psi_i\}_{i=1}^\infty$ is a basis in the weighted $L^2([-T,T])$ sense. Therefore, $\{\phi_i\psi_j\}_{i,j=1}^{\infty}$ is a basis in the weighted $L^2(\Omega\times[-T,T])$ sense. This shows that $\{\phi_i\psi_j\}_{i,j=1}^\infty$ are all possible eigenfunctions up to rescaling.
	\end{proof}
	
	\begin{proof}[Proof of Theorem \ref{thm:shrinker finite index}]
		The proof is divided into two parts.
		
		\medskip
		{\noindent\bf Part I:} {\em We first assume that $\Sigma$ has index $k<\infty$ and we are going to prove that $\mu_1(\Sigma)>-\infty$.} 
		
		By Definition \ref{def:index of self-shrinkers}, there exists a $k$-dimensional subspace $W\subset C^\infty(\Sigma)$ so that each non-zero $f\in W$ has compact support and 
		\[  \mathfrak B(f,f)<0.\]
		Without loss of generality, we may assume that $\spt f \subset B_R(0)\cap \Sigma$ for all $f\in W$. Denote by $\kappa_0=\sup_{\Sigma\cap B_{2R}(0)}|A(x)|+1$. 
		
		We pause to take a cut-off function $\eta: \mb R\rightarrow\mb [0,1]$ and a universal constant $C_0>1$ so that 
		\[ C_0> \eta'(t)\geq 0 \text{ and } |\eta''(t)|<C_0 \text{ for } t\in \mb R;\ \ \ \eta(t)=0 \text{ for } t\leq 1;\ \ \  \eta(t)=1 \text{ for } t\geq 2.\]
		
		Let $\kappa = 2nC_0\kappa^2_0$. We now assume on the contrary that $\mu_1(\Sigma)=-\infty$. Then there exist $f\in C^\infty(\Sigma)$ has compact support and satisfies
		\begin{equation}\label{eq:bad f}
			\int_{\Sigma} (|\nabla f|^2-|A|^2f^2-\frac{1}{2}f^2)e^{-\frac{|x|^2}{4}}<-4\kappa\int_\Sigma f^2e^{-\frac{|x|^2}{4}}.
		\end{equation}
		Define $\phi(x)=\eta(|x|/R)$. Then $\nabla \phi=\frac{1}{R}\cdot \eta'(|x|)\cdot \frac{x^\top}{|x|}$ and 
		\begin{align}\label{eq:compute of Delta phi}
			\Delta \phi &=\dv_\Sigma \big[\eta'(|x|)\cdot \frac{x^\top}{|x|}\cdot \frac{1}{R}\big]\\
			&=\eta''(|x|)\frac{|x^\top|^2}{|x|^2}\cdot\frac{1}{R^2}-\eta'(|x|/R)\frac{|x^\top|^2}{|x|^3}\cdot\frac{1}{R}+\frac{1}{R}(n-2H^2)\eta'(|x|)/|x|\nonumber\\
			&\geq -C_0-C_0/|x|-nC_0\kappa_0^2/|x|\geq -\kappa.\nonumber
		\end{align}
		Here $(\cdot)^\top$ is the projection to $T_x\Sigma$. The second equality used the fact that $\dv_\Sigma (x^\top)=n-2H^2$. The first inequality follows from $|H|^2\leq n|A|^2\leq n\kappa_0^2$ for $|x|\leq 2R$. For $|x|>2R$, $\Delta \phi=0$.
		
		Then by a direct computation, we have
		\begin{align}\label{eq:compute of nabla phi f}
			\int_\Sigma|\nabla (\phi f)|^2e^{-\frac{|x|^2}{4}}&=\int_\Sigma(|\nabla f|^2\phi^2+|\nabla \phi|^2f^2+\frac{1}{2}\langle \nabla \phi^2,\nabla f^2\rangle)e^{-\frac{|x|^2}{4}}\\
			&=\int_{\Sigma}(|\nabla f|^2+|\nabla \phi|^2f^2-\frac{1}{2}f^2\Delta \phi^2+\frac{1}{4}\langle x,\nabla \phi^2\rangle)e^{-\frac{|x|^2}{4}}\nonumber\\
			&\leq \int_\Sigma \Big[|\nabla f|^2+\kappa f^2+\frac{1}{2R}\cdot \phi\eta'(|x|)\cdot \frac{|x^\top|^2}{|x|}\Big]e^{-\frac{|x|^2}{4}}\nonumber\\
			&\leq \int_\Sigma (|\nabla f|^2+2\kappa f^2)e^{-\frac{|x|^2}{4}}.\nonumber
		\end{align}
		Here the first inequality is from the divergence theorem and $\eta'\leq C_0$. \eqref{eq:compute of Delta phi} is used in the second inequality. In the last inequality, we note that $\eta'=0$ for $|x|\geq 2R$.
		
		Now we prove that $\phi f$ has contribution for index. Indeed,
		\begin{align*}
			\mathfrak B(\phi f,\phi f)&=\int_{\Sigma} (|\nabla (\phi f)|^2-|A|^2f^2\phi^2-\frac{1}{2}f^2\phi^2)e^{-\frac{|x|^2}{4}}\\
			&\leq \int_\Sigma (|\nabla f|^2+2\kappa f^2-|A|^2f^2)e^{-\frac{|x|^2}{4}}+\int_{\Sigma\cap B_{2R}(0)} |A|^2f^2e^{-\frac{|x|^2}{4}}\\
			&\leq \int_\Sigma (-2\kappa+1) f^2e^{-\frac{|x|^2}{4}}+\int_{\Sigma\cap B_{2R}(0)} \kappa_0^2f^2e^{-\frac{|x|^2}{4}}\\
			&\leq -\kappa \int_\Sigma f^2\phi^2e^{-\frac{|x|^2}{4}}.
		\end{align*} 
		Here we used \eqref{eq:compute of nabla phi f} in the first inequality. The second one is from \eqref{eq:bad f} and the definition of $\kappa_0$.
		
		Note that $\phi f$ and $\psi\in W$ have disjoint support set. This implies that $\Sigma$ has $\Index\geq k+1$, which lead to a contradiction. This finishes Part I.
		
		\medskip
		{\noindent\bf Part II:} {\em We now prove $\Sigma\times\R$ has finite index.} 
		
		It suffices to prove that for any compact domain $\Omega\times[-T,T]\subset\Sigma\times\R$, the Dirichlet eigenproblem of $L_{\Sigma\times\R}$ has index bounded from above uniformly. Recall that $\mu_1$ is the first eigenvalue of $L_\Sigma$. Suppose $\mu_1(\Omega)<\mu_2(\Omega)\leq\cdots$ are eigenvalues of $L_\Sigma$ on $\Omega$, which correspond to the eigenfunctions $\{\phi_i\}_{i=1}^\infty$; suppose $\nu_1^T<\nu_2^T\leq\cdots$ are Dirichlet eigenvalues of $\partial_y^2-\frac{1}{2}y\partial_y$ ($=:L_0$) on $[-T,T]\subset\R$, which correspond to the eigenfunctions $\{\psi_j\}_{j=1}^\infty$. Recall that the $\nu_j^T$ converges to the eigenvalues of $L_0$ on $\mb R$ as $T\rightarrow \infty$; see \cite{CM2020}*{Lemma 6.1}.
		
		We first note that by definition, $\mu_1\leq \mu_1(\Omega)$ for any compact domain $\Omega$. Secondly, note that the eigenvalues of $L_0$ on $\mb R$ are given by non-negative half-integers with multiplicity one, i.e. $\nu^T_j\rightarrow \frac{j-1}{2}$ as $T\rightarrow \infty$. Thus we can take $T_0>0$ so that \begin{equation}\label{eq:nuj large}
			\text{$\nu_j^T>-\mu_1$ for all $j>1-2\mu_1$ and $T>T_0$.}
		\end{equation}

		To proceed the proof, we take $T>T_0$. By Proposition \ref{prop:eigen-product}, the Dirichlet eigenvalues of $L_{\Sigma\times\R}$ on $\Omega\times[-T,T]$ are $\{\mu_i(\Omega)+\nu_j^T\}_{i,j=1}^\infty$, which correspond to the eigenfunctions $\phi_i\psi_j$.
		
		Since $\Sigma$ has finite index $k$, the index of $\Omega$ is bounded from above by $k$ uniformly. From Part I, $\mu_1\neq -\infty$. This implies that, $\mu_i(\Omega)+\nu_j^T$ must be non-negative if $j>1-2\mu_1$ . So the index of $\Omega\times[-T,T]$ is bounded from above by an uniform constant $k(1-2\mu_1)$. This completes the proof.
	\end{proof}

	\appendix
	\section{Removing singularity for weakly stable free boundary $h$-hypersurfaces}\label{app:sec:removable singularity}
	\begin{theorem}[cf. \cite{ACS17}*{Theorem 27}]\label{thm:removable touching set}
		Let $(M^{n+1},\partial M,g)$ be a compact Riemannian manifold with boundary of dimension $3\leq (n + 1)\leq 7$. Given $h\in C^\infty(M)$ and $\Sigma\subset B_\epsilon (p)\setminus\{p\}$ an almost embedded free boundary $h$-hypersurface with $\partial\Sigma\cap \mathrm{Int}M\cap B_\epsilon(p)\setminus\{p\} = \emptyset$, assume that $\Sigma$ is weakly stable in $B_\epsilon(p)\setminus\{p\}$ as in Remark \ref{rmk:weak index}. If $[\Sigma]$ represents a varifold of bounded first variation in $B_\epsilon(p)$, then $\Sigma$ extends smoothly across $p$ as an almost embedded hypersurface in $B_\epsilon (p)$.
	\end{theorem}
	\begin{proof}

		Given any sequence of positive $\lambda_i \rightarrow 0$, consider the blowups $\{\bm \mu_{p,\lambda_i} (\Sigma) \subset\bm \mu_{p,\lambda_i} (M)\}$, where $\bm {\mu}_{p,\lambda_i} (x) = \frac{x-p}{\lambda_i}$.  Since $\Sigma$ has bounded first variation, $\bm\mu_{p,\lambda_i} (\Sigma)$ converges (up to a subsequence) to a stationary integral rectifiable cone $C$ in a half Euclidean space $\mb R^{n+1}_+ = T_p M$ with free boundary on $\partial \mb R^{n+1}_+=T_p(\partial M)$. By weak stability and Theorem \ref{thm:compactness for FPMC}, the convergence is locally smooth and graphical away from the origin, so $C$ is an integer multiple of some embedded minimal hypercone with free boundary. Hence the reflection of $C$ across $\partial \mb R^{n+1}_+$ is a stable minimal cone in $\mb R^{n+1}\setminus \{0\}$, and hence a plane with integer multiplicities. Therefore, we conclude that $C=m\cdot T_p(\partial M)$ or $2m\cdot P$ for some hyperplane $P\subset \mb R^{n+1}_+$ with $P\perp \partial \mb R^{n+1}_+$. Here $m=\Theta^n(\Sigma,p)$. Note that $P$ may depend on the choice of $\{\lambda_i\}$.
		
		If $C=m\cdot T_p(\partial M)$, then the argument in \cite{Zhou19}*{Theorem B.1} implies the removability of $p$.
		
		If $C=2m\cdot P$ for some half-hyperplane $P$, then by the locally smooth and graphical convergence, there exists $\sigma_0> 0$ small enough, such that for any $0 < \sigma \leq \sigma_0$, $\Sigma$ has an $m$-sheeted, ordered, graphical decomposition in the annulus $A_{\sigma/2,\sigma}(p) =M\cap B_\sigma(p)\setminus B_{\sigma/2} (p)$:
		\[\Sigma \cap A_{\sigma/2,\sigma} (p) = \bigcup_{i=1}^m \Sigma_i(\sigma).\]
		Here each $\Sigma_i (\sigma)$ is a graph over $A_{\sigma/2,\sigma} (p)\cap P$.
		We can continue each $\Sigma_i (\sigma)$ all the way to $B_{\sigma_0}(p)\setminus \{p\}$, and denote the continuation by $\Sigma_i$. Then each $\Sigma_i$ is a free boundary $h$-hypersurface in $M\cap B_{\Sigma_0}(p)\setminus \{p\}$ and $\Theta^n(\Sigma_i,p)=1/2$.  By the Allard type regularity theorem for rectifiable varifolds with free boundary and bounded mean curvature \cite{GJ86}*{Theorem 4.13}, $\Sigma_i$ extends as a $C^{1,\alpha}$ free boundary hypersurface across $p$ for some $\alpha\in(0,1)$. Higher regularity of $\Sigma_i$ follows from the prescribing mean curvature equation and elliptic regularity.
	\end{proof}

	\section{The second variation of $\A^h$ for smooth hypersurfaces}\label{section:2nd variation}
	Our goal is to derive the second variation of $\A^h$ for arbitrary hypersurfaces with boundary in $(M^{n+1},\partial M,g)$. Note that we always assume $(M^{n+1},\partial M,g)$ is isometrically embedded in some closed $(n+1)$-dimensional Riemannian manifold $(\wti M,\wti g)\subset \mb R^L$. Let $(\Sigma^n,\partial\Sigma)\subset (M,\partial M)$ be an embedded hypersurface in $M$ with boundary on $\partial M$. We consider a two-parameter family of ambient variations $\Sigma(t,s)$ of $\Sigma = \Sigma(0, 0)$ defined by
	\[
	\Sigma(t,s)= F_{t,s}(\Sigma) = \Psi_s\circ\Phi_t(\Sigma),
	\]
	where $\Phi_t$ and $\Psi_s$ are the flows generated by compactly supported vector fields $X$ and $Z$ in $\mb R^L$, respectively. We assume both $X,Z\in \wti{\mathfrak X}(M,\Sigma)$. Therefore each $\Sigma(t, s)$ is an embedded hypersurface in $\wti M$ with boundary lying on $\partial M$.
	
	We have
	\[\frac{\partial F_{t,s}}{\partial t}(x)\Big|_{t=s=0}=X(x),\  \frac{\partial F_{t,s}}{\partial s}(x)\Big|_{t=s=0}=Z(x),\ \text{and }  \frac{\partial }{\partial t}\frac{\partial }{\partial s}F_{t,s}(x)\Big|_{t=s=0}=D_XZ(x). \]
	
	The computation in \cite{ACS17}*{Appendix A} gives that 
	\begin{align}\label{eq:2nd variaiton of area}
		&\frac{\partial}{\partial t}\frac{\partial }{\partial s}\Big|_{t=s=0}\mc H^n(\Sigma(t,s))= \frac{\partial }{\partial t}\Big|_{t=0}\int_{\Sigma(t,0)}\dv Z\,d\mc H^n\\
		=&\int_{\Sigma}\Big(\langle\nabla^\perp(X^\perp),\nabla^\perp(Z^{\perp})\rangle-\Ric(X^\perp,Z^\perp)-|A|^2\langle X^\perp,Z^\perp\rangle\Big)\,d\mc H^n+\int_{\partial \Sigma}\langle\nabla_{X^\perp}Z^\perp,\nu_{\partial M}\rangle\,d\mc H^{n-1}+\nonumber\\
		&+\int_\Sigma\Xi_1(X,Z,\mf H)\, d\mc H^{n}+\int_{\partial\Sigma}\Xi_2(X,Z,\mf H,\bm\eta,\nu_{\partial M})\, d\mc H^{n-1},\nonumber
	\end{align}
	where
	\begin{gather*}
		\Xi_1(X,Z,\mf H)=\langle X^\perp,\mf H\rangle\langle Z^\perp,\mf H\rangle-\langle\nabla_{X^\perp}(Z^{\perp}),\mf H\rangle-\langle[X^\perp,Z^\top],\mf H\rangle,
	\end{gather*}
	and 
	\begin{align*}
		\Xi_2(X,Z,\mf H,\bm\eta,\nu_{\partial M})=&\langle\nabla_{X^\perp}Z^\perp,\bm\eta-\nu_{\partial M}\rangle+\langle [X^\perp,Z^\top],\bm\eta\rangle+\mathrm{div}_{\Sigma}(Z^\top)\langle X^\top,\bm\eta\rangle\\
		&-\langle Z^\perp,\mf H\rangle\langle X^\top,\bm\eta\rangle-\langle X^\perp,\mf H\rangle\langle Z^\top,\bm\eta\rangle.
	\end{align*}
	Here $\mf H=-H\nu$ is the mean curvature vector; $\bm\eta$ is the unit co-normal vector field of $\Sigma$; $\nabla$ is the connection on $M$; $\perp$ and $\top$ are the normal and tangential parts on $\Sigma(t,s)$, respectively.
	
	Let $h$ be a smooth function on $\wti M$. Then a direct computation gives that 
	\begin{align}\label{eq:2nd variation of integral of h}
		&\frac{\partial}{\partial t}\Big|_{t=0}\int_{\Sigma(t,0)}h\langle Z,\nu\rangle\,d\mc H^{n+1}\\
		=&\int_\Sigma \langle Z,h\nu\rangle\mathrm{div} X+\nabla_X\langle Z,h\nu\rangle\, d\mc H^{n}\nonumber\\
		=&\int_\Sigma \Big(\langle Z,h\nu\rangle \dv X^\top- \langle Z,h\nu\rangle\langle\mf H,X^\perp\rangle +\nabla_{X^\top}\langle Z,h\nu\rangle+(\nabla_{X^\perp }h)\langle Z,\nu\rangle+h(\nabla_{X^\perp}\langle Z,\nu\rangle)\Big)\,d\mc H^n\nonumber\\
		=&\int_{\Sigma}\Big((\partial_\nu h)\langle X^\perp,Z^\perp\rangle-\langle X,\mf H\rangle\langle Z,h\nu\rangle+ \langle\nabla_{X^\perp}(Z^\perp),h\nu\rangle\Big)\,d\mc H^n+\int_{\partial\Sigma}\langle Z^\perp,h\nu\rangle\langle X^\top,\bm\eta\rangle\,d\mc H^{n-1}.\nonumber
	\end{align}
	
	By \eqref{eq:2nd variaiton of area} and \eqref{eq:2nd variation of integral of h}, we have
	\begin{align*}
		&\frac{\partial}{\partial t}\Big|_{t=0}\int_{\Sigma(t,0)}\dv Z-h\langle Z,\nu\rangle\, d\mc H^n\\
		=&\int_{\Sigma}\Big(\langle\nabla^\perp(X^\perp),\nabla^\perp(Z^{\perp})\rangle-\Ric(X^\perp,Z^\perp)-|A|^2\langle X^\perp,Z^\perp\rangle-(\partial_\nu h)\langle X^\perp,Z^\perp\rangle\Big)\,d\mc H^n\\
		+&\int_{\partial \Sigma}\langle\nabla_{X^\perp}Z^\perp,\nu_{\partial M}\rangle\,d\mc H^{n-1}+\int_\Sigma\wti \Xi_1(X,Z,\mf H)\, d\mc H^{n}+\int_{\partial\Sigma}\wti\Xi_2(X,Z,\mf H,\bm\eta,\nu_{\partial M})\, d\mc H^{n-1},
	\end{align*}
	where
	\begin{align*}
		\wti\Xi_1(X,Z,\mf H)&=\Xi_1(X,Z,\mf H)+\langle X^\perp,\mf H\rangle\langle Z^\perp,h\nu\rangle- \langle\nabla_{X^\perp}(Z^\perp),h\nu\rangle\\
		&=\langle X^\perp,\mf H\rangle\langle Z^\perp,\mf H+h\nu\rangle-\langle\nabla_{X^\perp}(Z^{\perp}),\mf H+h\nu\rangle-\langle[X^\perp,Z^\top],\mf H\rangle\\
		&=\langle X^\perp,\mf H\rangle\langle Z^\perp,\mf H+h\nu\rangle-\langle\nabla_{X^\perp}(Z^{\perp}),\mf H+h\nu\rangle-A(X^\top,Z^\top)\langle\nu,\mf H\rangle,
	\end{align*}
	and 
	\begin{align*}
		\wti \Xi_2(X,Z,\mf H,\bm\eta,\nu_{\partial M})&=\Xi_2(X,Z,\mf H,\bm\eta,\nu_{\partial M})-
		\langle Z^\perp,h\nu\rangle\langle X^\top,\bm\eta\rangle\\
		&=\langle\nabla_{X^\perp}Z^\perp,\bm\eta-\nu_{\partial M}\rangle+\langle [X^\perp,Z^\top],\bm\eta\rangle+\mathrm{div}_{\Sigma}(Z^\top)\langle X^\top,\bm\eta\rangle\\
		&\ \  \ \ -\langle Z^\perp,\mf H+h\nu\rangle\langle X^\top,\bm\eta\rangle-\langle X^\perp,\mf H\rangle\langle Z^\top,\bm\eta\rangle.
	\end{align*}
	
	In the equality of $\wti\Xi_1(X,Z,\mf H)$, we used the following identity:
	\begin{align*}
		\langle[X^\perp,Z^\top],\nu\rangle&=\langle\nabla_{X^\perp}(Z^\top)-\nabla_{Z^\top}(X^\perp),\nu\rangle\\
		&=\langle\nabla_{X}(Z^\top)-\nabla_{X^\top}(Z^\top)-\nabla_{Z^\top}(X^\perp),\nu\rangle\\
		&=-\langle Z^\top,\nabla_{X}\nu\rangle-\langle\nabla_{Z^\top}(X^\perp),\nu\rangle+A(X^\top,Z^\top)\\
		&=\langle Z^\top,\nabla\langle X,\nu\rangle\rangle-\langle\nabla_{Z^\top}(X^\perp),\nu\rangle+A(X^\top,Z^\top)\\
		&= A(X^\top,Z^\top).
	\end{align*}
	
	We remark that 
	\begin{align*}
		|\wti\Xi_1(X,Z,\mf H)|+|\wti\Xi_2(X,Z,\mf H,\bm\eta,\nu_{\partial M})|
		\leq C\big(|X|(|\mf H+h\nu|+| \bm\eta-\nu_{\partial M} |+|Z^\top|)+|X^\top|\big).
	\end{align*}

	\section{Cut-off trick}\label{section:cut-off}
	In this section, we provide a lemma which has been used in Part 3 of the proof of Theorem \ref{thm:compactness for FPMC}\ref{compactness thm:index decreasing}. Such a result has also been used in \cite{GLWZ19}.
	
	\begin{lemma}
		Let $(M^{n+1},\partial M,g)$ be a compact manifold with boundary of dimension $(n+1)\geq 3$. Let $(\Sigma,\partial \Sigma)\subset (M,\partial M)$ be an almost embedded free boundary $h$-hypersurface and $\varphi$ be a smooth function on $\Sigma$. Then for any $p\in\Sigma$, there exists a family of cut-off functions $(\xi_r)_{0<r<\epsilon}$ for some $\epsilon>0$ so that $\xi_r(p)=0$
		\[\mathrm{II}_\Sigma(\varphi,\varphi)=\lim_{r\rightarrow 0}\mathrm{II}_\Sigma(\xi_r\varphi,\xi_r\varphi).\]
	\end{lemma}
	\begin{proof}
		Set 
		\begin{equation}
			\xi_r(x)=\left\{\begin{aligned}
				& 0, &|x|<r^2\\
				& 2-\frac{\log |x|}{\log r}, &r^2\leq |x|\leq r\\
				& 1, &|x|>r
			\end{aligned}\right.,
		\end{equation}
		where $|x|=\dist_M(x,p)$. Then we have $\int_\Sigma|\nabla\xi_r|^2<C(n)/|\log r|\rightarrow 0$ and $\xi_r\rightarrow 1$ as $r\rightarrow 0$. Then it suffices to prove $\int_{\Sigma}|\nabla(\xi_r\varphi)|^2\rightarrow\int_{\Sigma}|\nabla\varphi|^2$ as $r\rightarrow 0$. This follows from $\int_\Sigma|\nabla\xi_r|^2\rightarrow 0$

	\end{proof}

	\section{Local $h$-foliation with free boundary}\label{app:sec:local foliation}
	The following proposition is a generalization of minimal foliation given by B. White \cite{Whi87}. This description has already been stated in \cite{Wang20}*{Proposition A.2}.
	\begin{proposition}\label{prop:h foliation with boundary}
		Let $(M^{n+1},\partial M,g)$ be a compact Riemannian manifold with boundary, and let
		$(\Sigma,\partial \Sigma)\subset (M,\partial M)$ be an embedded, free boundary minimal hypersurface. Given a point $p\in\partial \Sigma$, there exist $\epsilon>0$ and a neighborhood $U\subset M$ of $p$ such that if $h:U\rightarrow \mb R$ is a smooth function with $\|h\|_{C^{2,\alpha}}<\epsilon$ and 
		\[w:\Sigma\cap U\rightarrow \mb R \text{ satisfies } \|w\|_{C^{2,\alpha}}<\epsilon,\]
		then for any $t\in(-\epsilon,\epsilon)$, there exists a $C^{2,\alpha}$-function $v_t: U\cap \Sigma\rightarrow \mb R$, whose graph $G_t$ meets $\partial M$ orthogonally along $U\cap \partial \Sigma$ and satisfies:
		\[H_{G_t} = h|_{G_t} ,\]
		(where $H_{G_t}$ is evaluated with respect to the upward pointing normal of $G_t$), and
		\[v_t(x) = w(x) + t, \text{ if } x\in \partial (U\cap \Sigma)\cap \mathrm{Int} M.\]
		Furthermore, $v_t$ depends on $t,h,w$ in $C^1$ and the graphs $\{G_t : t\in[-\epsilon, \epsilon]\}$ forms a foliation.
	\end{proposition}
	\begin{proof}
		The proof follows from \cite{Whi87}*{Appendix} together with the free boundary version \cite{ACS17}*{Section 3}. The only modification is that we need to use the following map to replace $\Phi$ in \cite{ACS17}*{Section 3}:
		\[\Psi: \mb R \times X \times Y\times Y \times Y \rightarrow Z_1 \times Z_2 \times Z_3.\]
		The map $\Psi$ is defined by
		\[\Psi(t,g,h,w,u) = (H_{g(t+ w + u)}-h, g(N_g(t + w + u), \nu_g (t + w + u)), u|_{\Gamma_2} );\]
		here all the notions are the same as \cite{ACS17}*{Section 3}.
	\end{proof}

	\bibliographystyle{amsalpha}
	\bibliography{minmax}
\end{document}